\documentclass[11pt,a4paper]{article}

\usepackage{authblk}

\makeatletter
\newcommand{\subjclass}[2][2010]{%
  \let\@oldtitle\@title%
  \gdef\@title{\@oldtitle\footnotetext{#1 \emph{Mathematics subject classification.} #2}}%
}
\newcommand{\keywords}[1]{%
  \let\@@oldtitle\@title%
  \gdef\@title{\@@oldtitle\footnotetext{\emph{Key words and phrases.} #1.}}%
}
\makeatother

\AtEndDocument{\bigskip{\footnotesize%
  \textit{E-mail address}, Xiaoqin Guo: \texttt{guoxq@ucmail.uc.edu} \par
  \addvspace{\medskipamount}
  \textit{E-mail address}, Hung Vinh Tran: \texttt{hung@math.wisc.edu}
}}

\usepackage{etoolbox}
\usepackage{comment,amsmath,amssymb,amsthm,bbm,mathtools,mathrsfs}
\usepackage{tikz,pgfplots,color,setspace,lmodern}
\usepackage[]{stix}
\usetikzlibrary{shapes, backgrounds,patterns,decorations.pathreplacing}
\usepackage{multicol,float}
\usepackage{makeidx}
\usepackage{hyperref, enumerate}
\pgfplotsset{compat=newest}
\makeindex

\def\var{{\rm Var}}
\def\cov{{\rm Cov}}

\def\XXint#1#2#3{{\setbox0=\hbox{$#1{#2#3}{\int}$ }
\vcenter{\hbox{$#2#3$ }}\kern-.6\wd0}}

\newcommand\error{\varepsilon}

\newcommand{\qe}{E_\omega}

\newcommand{\E} {\textnormal{\textsf{E}}}
\newcommand{\dd}{\mathrm{d}}

\newcommand{\B}{\mathbb{B}}

\newcommand{\R} {\mathbb{R}}
\newcommand{\Q} {\mathbb{Q}}
\newcommand{\Z} {\mathbb{Z}}
\newcommand{\nn}{\nonumber}
\newcommand{\N} {\mathbb{N}}
\newcommand{\dist} {\textnormal{dist}}

\newcommand{\tx}{\mathscr{H}}

\newcommand{\evp}[1]{\bar\omega^{#1}}
\newcommand{\vd}[1]{\partial'_{#1}}  
\newcommand{\fct}{\mathrm{H}} 
\newcommand{\der}[1]{{\mathcal D}_{#1}}
\newcommand{\nder}[1]{{\mathbb D}_{#1}}
\newcommand{\apk}{\phi^{\scriptscriptstyle\rm AP}}
\newcommand{\apg}{G^{\scriptscriptstyle\rm AP}}
\newcommand{\lock}{\phi^{\scriptscriptstyle\rm loc}}
\newcommand{\locg}{G^{\scriptscriptstyle\rm loc}}
\newcommand{\loca}{A^{\scriptscriptstyle\rm loc}}
\newcommand{\bern}{\mathfrak{B}}
\newcommand{\krt}{{\pmb{\phi}}}

\DeclarePairedDelimiter{\abs}{\lvert}{\rvert}

\DeclarePairedDelimiter{\norm}{\lVert}{\rVert}
\providecommand{\Abs}[1]{\Bigr\lvert#1\Bigl\rvert}

\providecommand{\mb}[1]{\mathbb{#1}}

\providecommand{\ms}[1]{\mathscr{#1}}
\providecommand{\mf}[1]{\mathfrak{#1}}

\DeclareMathOperator*{\osc}{osc}

\newcommand{\tr}{{\rm tr}}
\newcommand{\diag}{{\rm diag}}

\newtheorem{thmx}{Theorem}
\newtheorem{propx}[thmx]{Proposition}

\newtheorem{theorem}{Theorem}
\newtheorem{lemma}[theorem]{Lemma}
\newtheorem{corollary}[theorem]{Corollary}
\newtheorem{proposition}[theorem]{Proposition}

\theoremstyle{definition}  
\newtheorem{definition}[theorem]{Definition}
\newtheorem{remark}[theorem]{Remark}

\begin{document}

\title{Optimal convergence rates in stochastic homogenization in a balanced random environment}

\author[1]{Xiaoqin Guo \thanks{The work of XG is supported by Simons Foundation through Collaboration Grant for Mathematicians \#852943.}}
\author[2]{Hung V. Tran  \thanks{HT is supported in part by NSF CAREER grant DMS-1843320 and a Vilas Faculty Early-Career Investigator Award.}}

\affil[1]
{
Department of Mathematical Sciences, 
University of Cincinnati , 2815 Commons Way, Cincinnati, OH 45221, USA}
\affil[2]
{
Department of Mathematics, 
University of Wisconsin Madison, 480 Lincoln  Drive, Madison, WI 53706, USA}



\maketitle

\begin{abstract}
We consider random walks in a uniformly elliptic, balanced, i.i.d. random environment in $\Z^d$ for $d\geq 2$.
We first derive a quantitative law of large numbers for the invariant measure, which is nearly optimal.
A mixing property of the field of the invariant measure is then achieved.
We next obtain rates of convergence for the homogenization of the Dirichlet problem for non-divergence form difference operators, which are generically optimal for $d\geq 3$ and nearly optimal when $d=2$. Furthermore, we establish the existence, stationarity and uniqueness properties of the corrector problem for all dimensions $d\ge 2$. 
Afterwards, we quantify the ergodicity of the environmental process for both the continuous-time and discrete-time random walks, and as a consequence, we get explicit convergence rates for the quenched central limit theorem of the balanced random walk.
\end{abstract}

\tableofcontents


\section{Introduction}\label{sec:intro}
In this paper, we consider random walks in a uniformly elliptic, balanced, i.i.d. random environment in $\Z^d$ for $d\geq 2$.
Our main goals are two-fold.
Firstly, we derive a quantitative large-scale average of the invariant measure, which is nearly optimal, in Theorem \ref{thm:ave-rho-quant}.
A mixing property of the field of the invariant measure is achieved.
Secondly, we obtain rates of convergence for the homogenization of the Dirichlet problem in Theorem \ref{thm:opt-quant-homo}.
When $d\geq 3$, the convergence rate is $O(R^{-1})$, which is generically optimal.
Afterwards, we quantify the ergodicity of the environmental process for both the continuous- and discrete-time random walks in Theorem~\ref{thm:quant-ergo}, and as a consequence, we get explicit convergence rates for the quenched central limit theorem (QCLT) of the balanced random walk.

\subsection{Settings}
Let $\mb S_{d\times d}$ denote the set of $d\times d$ positive-definite diagonal matrices. A map 
\[
\omega:\Z^d\to\mb S_{d\times d}
\] is called an {\it environment}. We denote the set of all environments by $\Omega$. Let $\mb P$ be a probability measure on $\Omega$ so that 
\[
\left\{\omega(x)=\mathrm{diag}[\omega_1(x),\ldots, \omega_d(x)], x\in\Z^d\right\}
\] 
are i.i.d.\,under $\mb P$. Expectation with respect to $\mb P$ is denoted by $\mb E$ or $E_{\mb P}$. 

\begin{definition}\label{def:differences}
Let $\{e_1,\ldots,e_d\}$ be the  canonical basis for $\R^d$.  Let $U=\{e\in\Z^d:|e|_2=1\}$ be the set of unit vectors.
Define the difference operators $\nabla=(\nabla_e)_{e\in U}$,  and 
$\nabla^2=(\nabla^2_i)_{i=1}^d$ by
\begin{equation}\label{eq:def-nabla}
\nabla_e u(x)=u(x+e)-u(x), \quad
\nabla_i^2u(x)=u(x+e_i)+u(x-e_i)-2u(x).   
\end{equation}
Note that $\nabla, \nabla^2$ are linear operators. 
We also write, for $e,\ell\in U$, 
\[
\nabla_{e,\ell}^2=-\nabla_e\nabla_{\ell}.
\]
\end{definition}

For $r>0$, $y\in\R^d$ we let
\[
\B_r(y)=\left\{x\in\R^d: |x-y|<r\right\}, \quad
B_r(y)=\B_r(y)\cap\Z^d
\]
denote the continuous and discrete balls with center $y$ and radius $r$, respectively.
When $y=0$, we also write
$\B_r=\B_r(0)$ and $B_r=B_r(0)$. 
 For any $B\subset\Z^d$, its {\it discrete boundary} is defined as
\[
\partial B:=\left\{z\in\Z^d\setminus B: \dist(z,x)=1 \text{ for some }x\in B\right\}.
\]
Let $\bar B=B\cup\partial B$. By abuse of notations, whenever confusion does not occur,  we also use $\partial A$ and $\bar A$ to denote the usual continuous boundary and closure of $A\subset\R^d$, respectively.

 For $x\in\Z^d$, a {\it spatial shift} $\theta_x:\Omega\to\Omega$ is defined by 
 \[
 (\theta_x\omega)(\cdot)=\omega(x+\cdot).
 \] 
In a random environment $\omega\in\Omega$, we consider the discrete elliptic Dirichlet problem
\begin{equation}\label{eq:elliptic-dirich}
\left\{
\begin{array}{lr}
\tfrac 12\tr\big(\omega(x)\nabla^2u(x)\big)=\frac{1}{R^2}f\left(\tfrac{x}{R}\right)\zeta(\theta_x\omega) & x\in B_R,\\[5 pt]
u(x)=g\left(\tfrac{x}{|x|}\right) & x\in \partial B_R,
\end{array}
\right.
\end{equation}
where $f\in\R^{\B_1}, g\in\R^{\partial\B_1}$ are functions with good regularity properties and $\zeta\in\R^\Omega$ satisfies suitable integrability conditions (A special case is $\zeta\equiv 1$).  Stochastic homogenization studies  (for $\mb P$-almost all $\omega$)  the convergence of $u$ to the solution $\bar u$ of a deterministic {\it effective equation}
\begin{equation}\label{eq:effective-ellip}
\left\{
\begin{array}{lr}
\tfrac 12\tr(\bar a D^2\bar{u})
=f\bar\psi &\text{ in }\B_1,\\ 
\bar u=g &\text{ on }\partial \B_1,
\end{array}
\right.
\end{equation}
as $R\to\infty$. Here $D^2 \bar{u}$ denotes the Hessian matrix of $\bar{u}$ and $\bar a=\bar a(\mb P)\in\mb S_{d\times d}$ and $\bar\psi=\bar\psi(\mb P,\zeta)\in\R$ are {\it deterministic} and do not depend on the realization of the random environment (see the statement of Proposition \ref{thm:quant-homo} for formulas for $\bar{a}$ and $\bar{\psi}$).

The {\it non-divergence form} difference equation \eqref{eq:elliptic-dirich} is used to describe random walks in a random environment (RWRE) in $\Z^d$. To be specific, 
we set
\begin{equation}\label{def:omegaxei}
\omega(x,x\pm e_i):=\frac{\omega_i(x)}{2\tr\omega(x)} \quad  \text{ for } i=1,\ldots d,
\end{equation}
 and $\omega(x,y)=0$ if $|x-y|\neq 1$. Namely, we normalize $\omega$ to get a transition probability.  We remark that the configuration of $\{\omega(x,y):x,y\in\Z^d\}$ is also called a {\it balanced environment} in the literature \cite{L-82,GZ-12,BD-14}.
\begin{definition}\label{def:rwre-discrete}
For each fixed $\omega\in\Omega$, the random walk  $(X_n)_{n\ge 0}$ in the environment $\omega$ with $X_0=x$ is a Markov chain  in $\Z^d$ with transition probability $P_\omega^x$ specified by
\begin{equation}\label{eq:def-RW}
P_\omega^x\left(X_{n+1}=z|X_n=y\right)=\omega(y,z).
\end{equation}
 \end{definition} 
The expectation with respect to $P_\omega^x$ is written as $E_\omega^x$. When the starting point of the random walk is $0$, we sometimes omit the superscript and simply write $P_\omega^0, E_\omega^0$ as $P_\omega$ and $E_\omega$, respectively.  Notice that for random walks $(X_n)$ in an environment $\omega$, 
\begin{equation}\label{def:omegabar}
\bar\omega^i=\theta_{X_i}\omega\in\Omega, \quad i\ge 0,
\end{equation}
is also a Markov chain, called the {\it environment viewed from the particle} process. By abuse of notation, we enlarge our probability space so that $P_\omega$ still denotes the joint law of the random walks and $(\bar\omega^i)_{i\ge 0}$. 

We also consider the {\it continuous-time} RWRE $(Y_t)$ on $\Z^d$. Set, for $\omega\in\Omega$,
\begin{equation}
\label{eq:def-a}
a(x)=a_\omega(x):=\frac{\omega(x)}{\tr\omega(x)}=\diag[2\omega(x,x+e_1),\ldots, 2\omega(x,x+e_d)], \quad x\in\Z^d,
\end{equation}
and write the $i$-th diagonal entry of $a(x)$ as $a_i(x)=\frac{\omega_i(x)}{\tr\omega(x)}$.
\begin{definition}\label{def:rwre-continuous}
Let $(Y_t)_{t\ge 0}$ be the Markov process on $\Z^d$ with generator \begin{equation}\label{eq:def-of-L}
L_\omega u(x)=\sum_y\omega(x,y)[u(y)-u(x)]=\tfrac{1}{2}\tr(a(x)\nabla^2 u).
\end{equation}
\end{definition}
By abuse of notation, we also denote by $P_\omega^x$  the quenched law of $(Y_t)$. If there is no ambiguity from the context, we also write, for $x,y\in\Z^d$, $n\in\Z, t\in\R$, the transition kernels of the discrete and continuous time walks as
\[
p_n^\omega(x,y)=P_\omega^x(X_n=y), \quad\text{ and }\quad
p_t^\omega(x,y)=P_\omega^x(Y_t=y),
\]
respectively.

\subsection{Main assumptions}
Throughout the paper, the following assumptions are always in force.
\begin{enumerate}
\item[(A1)] $\left\{\omega(x), x\in \Z^d \right\}$ are i.i.d.\,under the probability measure $\mb P$.
\item[(A2)] $\frac{\omega}{\tr\omega}\ge 2\kappa{\rm I}$ for $\mb P$-almost every $\omega$ and some constant 
$\kappa\in(0,\tfrac{1}{2d}]$.
\item[(A3)] $\psi\in L^\infty(\mb P)$ is a bounded  measurable function of $\omega(0)$.
\end{enumerate} 
In the paper, we use $c, C$ to denote positive constants which may change from line to line but only depend on the dimension $d$ and the ellipticity constant $\kappa$ unless otherwise stated. We write $A
\lesssim B$ if $A\le CB$, and $A\asymp B$ if $A\lesssim B$ and $A\gtrsim B$. 
We also use notations $A\lesssim_j B$, $A\asymp_j B$ to indicate that the multiplicative constant depends on the variable $j$ other than $(d,\kappa)$.

\subsection{Earlier results in the literature}

We first recall the following quenched central limit theorem (QCLT) proved by Lawler \cite{L-82}, which is a discrete version of Papanicolaou,  Varadhan \cite{PV-82}.
\begin{thmx}\label{thm:QCLT}
Assume {\rm(A2)} and that law $\mb P$ of the environment is ergodic under spatial shifts $\{\theta_x:x\in\Z^d\}$. Then 
\begin{enumerate}[(i)]
\item There exists a probability measure $\mb Q\approx\mb P$ such that $(\evp{i})_{i\ge 0}$ is an ergodic (with respect to time shifts) sequence under law $\mb Q\times P_\omega$.
\item For $\mb P$-almost every $\omega$, the rescaled path $X_{n^2t}/n$ converges weakly (under law $P_\omega$) to a Brownian motion with covariance matrix 
\begin{equation}
\label{eq:def-abar}
\bar a=\diag[\bar a_1,\ldots,\bar a_d]:=E_\Q[a]=E_\Q[\tfrac{\omega(0)}{\tr\omega(0)}]>0.
\end{equation}
\end{enumerate}
\end{thmx}

QCLT for the balanced RWRE in static environments under weaker ellipticity assumptions can be found at \cite{GZ-12, BD-14}. For dynamic balanced random environment, QCLT was established in  \cite{DGR-15} and finer results concerning the local limit theorem and heat kernel estimates was obtained at \cite{DG-19}.  When the RWRE is allowed to make long jumps,  non-CLT stable limits of the balanced random walk is considered in \cite{CCKW-21-1,CCKW-21-2}.
 We refer to the lecture notes \cite{BoSz-02, OZ-04, MB-11, DreRam-14, Kumagai-14} for QCLT results in different models of RWRE.

We are moreover interested in characterizing the invariant measure $\Q$. Denote the Radon-Nikodym derivative of $\Q$ with respect to $\mb P$ as
\begin{equation}\label{eq:def-rho}
\rho(\omega)=\dd\mb Q/\dd\mb P.
\end{equation}
For any $x\in\Z^d$ and finite set $A\subset\Z^d$, we define
\[
\rho_\omega(x):=\rho(\theta_x\omega)
\quad \text{ and }
\quad
\rho_\omega(A)=\sum_{x\in A}\rho_\omega(x).
\]

As an important feature of the non-divergence form model,  $\rho_\omega$ does not have deterministic (nonzero) upper and lower bounds. Moreover, the heat kernel $p_t^\omega(\cdot,\cdot)$ is not expected to have deterministic Gaussian bounds.

For $r\ge 0, t>0$, define a function
\begin{equation}\label{eq:def-mf-h}
\mf h(r,t)=\frac{r^2}{r\vee t}+r\log(\frac{r}{t}\vee 1), \quad
r\ge 0, t>0.
\end{equation} 
The following result was obtained by Guo, Tran \cite{GT-22}.
\begin{thmx}\label{thm:hk-bounds}
Assume {\rm(A1), (A2)}, and $d\ge 2$.
Let $s=s(d,\kappa)=2+\tfrac{1}{2\kappa}-d\ge 2$.
For any $\error\in(0,1)$, there exists a random variable $\tx(\omega)=\tx(\omega,d,\kappa,\error)>0$ with 
$\mb E[\exp(c\tx^{d-\error})]<\infty$ such that the following properties hold.
\begin{enumerate}[(a)]
\item\label{item:rho-bounds} For $\mb P$-almost all $\omega$,
\[
c\tx^{-s}\le \rho(\omega)\le C\tx^{d-1}.
\]
\item\label{item:rho-hke-expmmt}
Recall the function $\mf h$ in \eqref{eq:def-mf-h}. 
For any $r\ge 1$ and $\mb P$-almost all $\omega$,
\[
c\tx^{-s}\le
\frac{r^d\rho_\omega(0)}{\rho_\omega(B_r)}
\le
C\tx^{d-1}.
\]
\item\label{item:hke-expmmt}
For any $x\in\Z^d, t>0$, and $\mb P$-almost all $\omega$,
\begin{align*}
p_t^\omega(x,0)
&\le  
C\tx^{d-1}(1+t)^{-d/2}e^{-c\mf h(|x|,t)},\\
p_t^\omega(x,0)&\ge
c\tx^{-s}(1+t)^{-d/2}e^{-C|x|^2/t}.
\end{align*}	
\end{enumerate}
\end{thmx}
\begin{remark}
In the PDE setting,  positive and negative algebraic moment bounds  and volume doubling property of $\rho$ were proved by Bauman \cite{Baum84}.  
The $L^p$ integrability of the heat kernel moment was proved by Fabes and Stroock \cite{FS84}. 
Deterministic heat kernel bounds in terms of $\rho$ was shown by Escauriaza \cite{Esc00} in the PDE setting, and by Mustapha \cite{Mustapha-06} for discrete time balanced random walks.  
In the more general dynamic ergodic balanced environment setting,  the bounds 
\begin{equation}\label{eq:ergodic-hke}
\frac{c\rho_\omega(0)}{\rho_\omega(B_{\sqrt t})}e^{-C|x|^2/t}
\le 
p_t^\omega(x,0)
\le \frac{C\rho_\omega(0)}{\rho_\omega(B_{\sqrt t})}e^{-c\mf h(|x|,t)}
\end{equation}
were proved by Deuschel, Guo \cite[Theorem 11]{DG-19}.  Recently, Armstrong, Fehrman, Lin \cite{AFL-22} obtain an algebraic rate of convergence for the heat kernels.
\end{remark}

 We now state a quantitative homogenization result  in Guo, Peterson, Tran \cite[Theorem 1.5]{GPT-19}, which can be considered as a discrete version of Armstrong, Smart \cite[Theorem 1.2]{AS-14}.
\begin{propx}
\label{thm:quant-homo}
Assume {\rm(A1), (A2)}.
Recall the measure $\Q$ in Theorem~\ref{thm:QCLT}. 
Suppose $g\in C^\alpha(\partial\B_1)$, $f\in C^{\alpha}(\B_1)$ for some $\alpha\in(0,1]$, and $\zeta$ is a measurable function of $\omega(0)$ with $\bar\psi:=\norm{\zeta/\tr\omega(0)}_\infty <\infty$.
Let $\bar u$ be the solution of the Dirichlet problem \eqref{eq:effective-ellip} with $\bar a=E_{\Q}[\omega(0)/\tr\omega(0)]>0$ and $\bar\psi$ as above.

For any $\error\in(0,1)$,  let $\tx=\tx(\omega,d,\kappa,\error)$ be the same as in Theorem~\ref{thm:hk-bounds}. Then,
there exists a constant $\beta=\beta(d,\kappa,\error)\in(0,1)$ such that for any $R>0$, the solution $u$ of \eqref{eq:elliptic-dirich} satisfies
\[
\max_{x\in B_R}\Abs{u(x)-\bar u(\tfrac{x}{R})}
\lesssim 
A\big(1+\ms (\tfrac{\tx}{R})^{1-\error/d}\big)R^{-\alpha\beta},
\]
where  $A=\norm{f}_{C^{0,\alpha}(\B_1)}\norm{\tfrac{\zeta}{\tr\omega(0)}}_\infty+[g]_{C^{0,\alpha}(\partial\B_1)}$.
\end{propx}

When the balanced environment is allowed to be non-elliptic and genuinely $d$-dimensional, 
(weak) quantitative results and Harnack inequalities for non-divergence form difference operators were obtained by Berger, Cohen, Deuschel, Guo  \cite{BCDG-18}, and Berger, Criens \cite{BC-22} for  $\omega$-harmonic and $\omega$-caloric functions, respectively.

Let us also give a brief overview of the quantitative homogenization of non-divergence form operators in the continuous PDE setting. Yurinski derived a second moment estimate of the homogenization error in \cite{Yur-88} for linear elliptic case. 
Caffarelli, Souganidis \cite{CaSu10} proved a logarithmic convergence rate for the fully nonlinear case. 
Afterwards, Armstrong, Smart \cite{AS-14}, and Lin, Smart \cite{LS-15} achieved an algebraic convergence rate for fully nonlinear elliptic equations, and fully nonlinear parabolic equations, respectively.
Armstrong, Lin \cite{AL-17} obtained quantitative estimates for the approximate corrector problems.

For $d\ge 2$ and any finite set $A\subset\Z^d$,
denote  the exit time from $A$  by  
\begin{equation}\label{def:tau}
\tau(A)=\tau(A;X)=\inf\{n\ge 0:X_n\notin A\}.
\end{equation}

\begin{definition}
\label{def:green}
For $R\ge 1$, $\omega\in\Omega$, $x\in \Z^d$, $S\subset\Z^d$,  the {\it Green function} $G_R(\cdot,\cdot)$ in the ball $B_R$ for the balanced random walk is defined by
\[
G_R(x, S)=G_R^\omega(x,S):=
E_\omega^x\big[\int_0^{\tau(B_R)}\mathbbm{1}_{Y_t\in S}\dd t\big],
 \quad x\in\bar B_R.
\]
We also write $G_R(x,y):=G_R^\omega(x,\{y\})$ and $G_R(x):=G_R(x,0)$. 
When $d\ge 3$,  for any finite set $S\subset\Z^d$, the Green function on the whole space can be defined as
\[
G^\omega(x, S)=\int_0^\infty p_t^\omega(x,S)\dd t<\infty.
\]
When $d=2$,  for any $x,y\subset\Z^d$,  the {\it potential kernel} is defined as
\begin{equation}
\label{eq:def-potential}
A(x,y)=A^\omega(x,y)=\int_{0}^\infty [p^\omega_t(y,y)-p^\omega_t(x,y)]\dd t,
\quad x\in\Z^2.
\end{equation}
\end{definition}

The bounds for the Green functions and the potential kernel were proved in Guo, Tran \cite{GT-22}, which was based on the idea of Armstrong, Lin \cite[Proposition 4.1]{AL-17}.

\begin{thmx}\label{thm:Green-bound}
Assume {\rm(A1), (A2)}.  For $\error>0$, let $s>0$, $\tx=\tx(\omega,d,\kappa,\error)>0$ be as in Theorem~\ref{thm:hk-bounds}. 
For $r> 0$, let
\begin{equation}\label{eq:def_ur}
U(r):=\left\{
\begin{array}{lr}
-\log r & \quad d=2,\\
r^{2-d} &\quad d\ge 3.
\end{array}
\right.
\end{equation}
Then $\mb P$-almost surely, for all $x\in B_R$,
\[
\tx^{-s}[U(|x|+1)-U(R+2)]
\lesssim
G_R^\omega(x,0)\lesssim 
\tx^{d-1}[U(|x|+1)-U(R+2)].
\]
As consequences,
$\mb P$-almost surely,  for all $x\in\Z^d$, 
\begin{align*}
\tx^{-s}\log(|x|+1)\lesssim
A^\omega(x,0)
\lesssim  \tx\log(|x|+1), &\text{ when }d=2;\\
\tx^{-s}(1+|x|)^{2-d}\lesssim
G^\omega(x,0)
\lesssim
\tx^{d-1}(1+|x|)^{2-d}, &\text{ when }d\ge 3.
\end{align*}
\end{thmx}

Recall the continuous time RWRE $(Y_t)_{t\ge 0}$ in Definition~\ref{def:rwre-continuous}. 
Define the semi-group $P_t$, $t\ge 0$, on $\R^\Omega$ by
\begin{equation}\label{eq:def-semigroup}
P_t \zeta(\omega)=E_\omega^0[\zeta(\theta_{Y_t}\omega)]=\sum_z p_t^\omega(0,z)\zeta(\theta_z\omega).
\end{equation}
 The following theorem from Guo, Tran \cite{GT-22} estimates the optimal speed of decorrelation of the environmental process $\evp{t}$ from the original environment.
\begin{thmx}
\label{thm:var_decay}
Assume {\rm(A1), (A2)}, and $d\ge 3$. For $t\ge 0$ and any measurable function $\zeta\in\R^\Omega$ of $\omega(0)$ with $\norm{\zeta}_\infty\le 1$, we have
\begin{align}
&\var_{\Q}(P_t \zeta)\le C(1+t)^{-d/2};\label{eq:var_decay_q}\\
&\norm{P_t\zeta}_1+\norm{P_t\zeta-\mb E[P_t\zeta]}_p\le C_p(1+t)^{-d/4}
\quad \text{ for all }p\in(0,2). \label{eq:var_decay_p}
 \end{align}	 
\end{thmx}

\subsection{Main results}
The field $\{\rho_\omega(x): x\in\Z^d\}$ of the invariant measure,  which governs the long term behavior of the diffusion and which determines the effective PDE,  plays a central role in the  theory of homogenization of  non-divergence form equations.

We first obtain a rate of convergence of the average $\rho_\omega(B_R)/|B_R|$ of the invariant measure to $1$ as $R \to \infty$. 

\begin{theorem}
\label{thm:ave-rho-quant}
Assume {\rm(A1), (A2)}.
For any $d\ge 2, p\in(0,0.5)$, $t>0$ and $R\ge 2$,
\[
P\left(
\Abs{\tfrac{\rho_\omega(B_R)}{|B_R|}-1}
\ge 
tR^{-d/2}\log R
\right)
\le 
C_p\exp(-ct^p).
\]
\end{theorem}
Note that the rate $R^{-d/2}\log R$ is very close to the size $R^{-d/2}$ of the diffusive scaling. In other words,  to some extent the field $(\rho_\omega(x))_{x\in\Z^d}$ behaves quite similarly to i.i.d. random variables.  Hence,  we expect the rate $R^{-d/2}\log R$ obtained here to be either optimal or nearly optimal.
For non-divergence form PDEs,  the volume-doubling property for the measure $\rho_\omega(\cdot)$ was proved by Bauman \cite{Baum84}. 
An algebraic convergence rate $R^{-\gamma}$ for some $\gamma \in (0,1)$ was proved recently by Armstrong, Fehrman, Lin \cite[Theorem 1.4]{AFL-22}.

In the course of obtaining our homogenization results in this paper,  sensitivity estimates together with an Efron-Stein type inequality are used to control fluctuations of a random field around its mean. This method was used in the stochastic homogenization of divergence-form operators, e.g., \cite{NS-97, GNO-15}.  To facilitate this strategy,  obtaining sensitivity estimates (with respect to the change of the environment) is crucial, and $C^{1,1}$ estimates for the random equation is necessary,  cf.  e.g., \cite{GNO-15, AL-17}.
 To obtain $C^{1,1}$ regularity for the heterogeneous equation, we follow the idea of Armstrong, Lin \cite{AL-17} who generalized the compactness argument of Avellaneda, Lin \cite{AL-87} to the random non-divergence form setting.
 
The key observation in the proof of Theorem~\ref{thm:ave-rho-quant} is explained as follows.
 Although the invariant measure $\rho_\omega(x)$ does not have an explicit expression,  it can be interpreted as the long term frequency of visits to location $x$.  Hence,  modifying the local value of the environment is related to the Green function of the RWRE.  Guided by this intuition, we will obtain a formula for the sensitivity estimate of the invariant measure in terms of the Green function.

\medskip

As indicated in Theorem~\ref{thm:ave-rho-quant}, the field $\{\rho_\omega(x): x\in\Z^d\}$ is expected to have weak enough correlation so that the behavior of its mean fluctuation over $B_R$ resembles (up to a logarithmic factor) that of i.i.d.  random variables.  The following proposition reveals the localization and correlation properties of the invariant measure.

\begin{proposition}
\label{prop:correlation}
Assume {\rm (A1), (A2)}. 
\begin{enumerate}[(i)]
\item\label{item:prop-cor-1} There exists a random variable $\ms Y>0$ with $\mb E[\ms Y^p]<C_p$, $\forall p<2/5$, such that for any $x\in\Z^d, r\ge 1$,  letting $\rho_r(x)=\rho_{r,\omega}(x):=\mb E[\rho_\omega(x)|\omega(y):y\in B_r(x)]$, 
\[
\abs{\rho(x)-\rho_r(x)}
\le 
\left\{
\begin{array}{lr}
\ms Y r^{-1}\log r, &d=2\\
\ms Y r^{-d/2}, & d\ge 3.
\end{array}
\right.
\]
\item 
For any $x,y\in\Z^d$ with $x\neq y$,
\[
\left|\cov_{\mb P}(\rho(x),\rho(y))\right|
\lesssim
\left\{
\begin{array}{lr}
|x-y|^{-2}[\log(1+|x-y|)]^3 &d=2\\
|x-y|^{-d}\log(1+|x-y|) & d\ge 3.
\end{array}
\right.
\]
\end{enumerate}
\end{proposition}
This is perhaps the first characterization of the correlation structure of the invariant measure (with algebraic mixing rates) for the non-divergence form operator $\tr(\omega\nabla^2 u)$ in a random environment.

\medskip

Next, we derive rates of convergence for the stochastic homogenization of the Dirichlet problem \eqref{eq:elliptic-dirich} for non-divergence form difference operators.

\begin{theorem}
\label{thm:opt-quant-homo}
Assume {\rm(A1), (A2)}.
Suppose $f,g$ are both in $C^4(\R^d)$, and $\zeta$ is a measurable function of $\omega(0)$ with $\norm{\zeta/\tr\omega(0)}_\infty <\infty$.
Let $\bar u$ be the solution of the Dirichlet problem \eqref{eq:effective-ellip} with $\bar a=E_{\Q}[\omega(0)/\tr\omega(0)]>0$ and $\bar\psi:=E_{\Q}[\zeta/\tr\omega(0)]$.

For any $\error\in(0,1)$,  $R\ge 2$, there exists a random variable $\ms Y=\ms Y(R,\error, \omega)>1$ with $\mb E[\exp(\ms Y^{{d}/{(2d+2)}-\error})]<C$ such that the solution $u$ of \eqref{eq:elliptic-dirich} satisfies
\[
\max_{x\in B_R}|u(x)-\bar{u}(\tfrac{x}{R})|
\lesssim
\left\{
\begin{array}{lr}
\tfrac{1}R\norm{\bar u}_{C^4(\bar\B_1)}\ms Y &\text{ when }d\ge 3\\
\frac{(\log R)^{2+\error}}R\norm{\bar u}_{C^4(\bar\B_1)}\ms Y &\text{ when }d=2.
\end{array}
\right.
\]
\end{theorem}
Thus, for $d\geq 3$, we obtain the optimal rate of convergence for the homogenization of the Dirichlet problem, which is generically of scale $R^{-1}$.
This is consistent with the generically optimal rate $R^{-1}$ for the periodic setting.
See the classical books Bensoussan, Lions, Papanicolaou \cite{BLP}, Jikov, Kozlov, Oleinik \cite{JKO}  for the derivation, and Guo, Tran, Yu \cite{GTY-19}, Sprekeler, Tran \cite{ST-21}, Guo, Sprekeler, Tran \cite{GST-22} for discussions on the optimality of the rates.
We also refer the reader to Jing, Zhang \cite{Jing-Zhang} for the optimal convergence rate in the presence of a large drift.
It is not clear to us what the optimal rate is when $d=2$, which deserves further analysis.

To prove Theorem \ref{thm:opt-quant-homo}, we apply the classical method of two-scale expansions and regularity estimates of the correctors in Section \ref{subsec:local corr}.
As evident from the two-scale expansion (Lemma~\ref{lem:two-scale-exp}), it is not the size of the correctors, but rather the {\it gradient of the correctors} that determines the rate of the homogenization of the non-divergence form problem.  However,  a stationary corrector on the whole space was constructed by Armstrong, Lin \cite[Section 7]{AL-17} only for $d\ge 5$, and for $d<5$ the approximate corrector (cf.  \eqref{eq:eq_of_phihat}) which is usually used in the literature does not possess enough regularity for optimal estimates.   To overcome these challenges, we construct local correctors (cf. Definition~\ref{def:new-corrector}) which have the desired regularity inside the ball.  Roughly speaking, the approximate corrector corresponds to the RWRE subject to an exponential killing time, while our local corrector only kills the RWRE outside of $B_R$ and as a result it does not ``feel"  any perturbations inside the ball $B_R$. By doing this we sacrifice the stationarity of the corrector, but will gain better regularity in $B_R$.  More detailed probability intuition can be found below \eqref{eq:local-func}.

We believe that our construction of the local correctors is new in the literature.

Furthermore, we establish the existence, stationarity, and uniqueness of the (global) corrector, completing the corrector theory of the  non-divergence form operator in the i.i.d. environment for all dimensions $d\ge 2$. To this end, define, for 
$R\ge 1$, the dimension-dependent functions $\mu=\mu_d$ and $\delta=\delta_d$ as
\begin{equation}\label{eq:def-mu}
\mu(R):=
\left\{
\begin{array}{lr}
R & d=2\\
R^{1/2} &d=3\\
(1\vee \log R)^{1/2} &d=4\\
1 &d\ge 5,
\end{array}
\right.
\end{equation}
\begin{equation}\label{eq:def-delta}
\delta(R):=
\left\{
\begin{array}{lr}
1 &\text{ when }d\ge 3\\
(1\vee \log R)^{3/2} &\text{ when }d=2.
\end{array}
\right.
\end{equation}

\begin{theorem}
\label{thm:global_krt}
Let $\psi$ be an $L^\infty(\mb P)$-bounded function of $\omega(0)$ with $\norm{\psi}_\infty=1$.
For each $d\ge 2$ and $\mb P$-a.e. $\omega$, there exists a function $\krt=\krt_\omega:\Z^d\to\R$ that solves 
\[
L_\omega\krt(x)=\psi(\theta_x\omega)-\bar\psi \quad\text{ for }x\in\Z^d
\]
with the following properties
\begin{enumerate}[(i)]
\item\label{item:gkrt-2} When $d\ge 5$, $\mb E[\exp(c|\krt(x)/\mu(|x|)|^p)]<C_p$ for any $p\in(0,\tfrac{2d}{3d+2}),x\in\Z^d$;

When $d=3,4$, $\mb E[\exp(c|\krt(x)/\mu(|x|)|^p)]<C_p$ for any $p\in(0,\tfrac{2d}{3d+4}),x\in\Z^d$;

When $d=2$,
$\mb E\left[
\exp\big(
C\abs{\tfrac{\krt(x)}{|x|\log(|x|\vee 2)^{3/2}}}^p
\big)
\right]\le C_p$ for any $p\in(0,\tfrac13), x\in\Z^d$;
\item\label{item:gkrt-3} $\mb E[\exp(c|\nabla\krt(x)/\delta(|x|)|^q)]\le C_q$ for any $q\in(0,\tfrac{d}{2d+2}), x\in\Z^d$;
\item\label{item:gkrt-2nd-der}
 $\mb E[\exp(c|\nabla^2\krt(x)|^r)]\le C_r$ for any $r\in(0,\tfrac{d}{2}), x\in\Z^d$;
\item\label{item:gkrt-4}(Stationarity properties)
\begin{itemize}
\item When $d\ge 5$,  the field $\{\krt_\omega(x):x\in\Z^d\}$ is stationary (under $\mb P$);
\item When $d\ge 3$, 
 $\nabla\krt_\omega(x)=\nabla\krt_{\theta_x\omega}(0)$ for all $x\in\Z^d$, and the field $\{\nabla\krt_\omega(x):x\in\Z^d\}$  is stationary;
\item When $d=2$, $\nabla^2\krt_\omega(x)=\nabla^2\krt_{\theta_x\omega}(0)$ for all $x\in\Z^d$, and the field $\{\nabla^2\krt_\omega(x):x\in\Z^d\}$  is stationary.
\end{itemize}
\end{enumerate}
Moreover,  such a corrector $\krt$ is unique up to an additive constant when $d\ge 3$,  and is unique up to an affine transformation when $d=2$.
\end{theorem}

%
\medskip

Note that the effective matrix $\bar a=E_{\Q}[\omega/\tr\omega]$ does not have an explicit expression. 
Even though by Birkhoff's ergodic theorem, $\Q$ can be approximated by
\[
\lim_{n\to\infty} \frac{1}{n} \sum_{i=0}^{n-1}E_\omega[\psi(\theta_{X_i}\omega)] = E_\Q[\psi] \quad \mb P\text{-a.s.}
\]
for any $L^1$ function $\psi$ on environments, in order to better understand the effective matrix $\bar a$ it is important to quantify the speed of this convergence.
To this end, we set, for $T\ge 1$,
\begin{equation}
\nu(T)=
\label{eq:mu2}
\left\{
\begin{array}{lr}
T^{-1/2} & d=2\\
T^{-3/4} &d=3\\
T^{-1}(\log T)^{1/2} &d=4\\
T^{-1} &d\ge 5.
\end{array}
\right.
\end{equation}
We will quantify the ergodicity of the environmental process for both the continuous- and discrete-time random walks in a balanced random environment.
\begin{theorem}
\label{thm:quant-ergo}
Assume ({\rm A1), (A2), (A3)}. Let $\nu$ be as in \eqref{eq:mu2}. For any $0<p<\tfrac{2d}{3d+2}$, there exists $C=C(d,\kappa,p)$ such that for $T, n\ge 2$ and any $t\ge 0$,
\begin{align*}
&\mb P\left(
\Abs{
\tfrac{1}{T}E_\omega\big[\int_0^T \psi(\theta_{Y_s}\omega)\dd s\big]-\bar\psi
}\ge 
t\nu(T)\norm{\psi}_\infty
\right)
\le 
C\exp(-\tfrac{t^p}{C}), \\
&\mb P\left(
\Abs{
\tfrac{1}{n}E_\omega\big[\sum_{i=0}^{n-1} \psi(\theta_{X_i}\omega)\big]-\bar\psi
}\ge 
t\nu(n)\norm{\psi}_\infty
\right)
\le 
C\exp(-\tfrac{t^p}{C}).
\end{align*}
\end{theorem}
\begin{remark}We remark that an (unknown) algebraic rate for the convergence of the ergodic average was obtained in \cite[Theorem~1.2]{GPT-19} in the discrete setting and recently in \cite[Theorem~1.6]{AFL-22} in the PDE setting.

Recall that Theorem~\ref{thm:var_decay} (from  \cite{GT-22}) states that, when $d\ge 3$, the typical size of $P_t\psi-\bar\psi$ is of scale $t^{-d/4}$. 
Observe that the typical size $\nu(T)$ of the ergodic average in Theorem~\ref{thm:quant-ergo} satisfies (for $T\ge 2$)
\[
\nu(T)
\lesssim
\frac 1T\int_1^T t^{-d/4}\dd t.
\]
(The sign $\lesssim$ can be replaced by $\asymp$ except for $d=4$ when $\nu(T)$ is a $\sqrt{\log T}$ factor smaller than the right side.)
Hence, Theorem~\ref{thm:quant-ergo} can be regarded as the integral version of Theorem~\ref{thm:var_decay} which  holds for all $d\ge 2$ and which has much better  exponential integrability.
\end{remark}

As a consequence of Theorem~\ref{thm:quant-ergo}, we obtain explicit convergence rates for the QCLT of the balanced random walk.
\begin{corollary}
\label{cor:quant-clt}
Assume {\rm (A1), (A2)}.  For any $0<q<\tfrac{2d}{5(3d+2)}$, $n\ge 2$, there exists a random variable $\ms Y=\ms Y(\omega,q;n,\kappa,d)$ with $\mb E[\exp(\ms Y^q)]\le C$ such that, $\mb P$-almost surely,
for any unit vector $\ell\in\R^d$,
\[
\sup_{r\in\R}
\Abs{
P_\omega
\left(
X_{n}\cdot\ell/\sqrt n \le r\sqrt{\ell^T\bar a\ell}
\right)
-\Phi(r)
}
\le 
C\nu(n)^{1/5}\ms Y,
\]
where $\Phi(r)=(2\pi)^{-1/2}\int_{-\infty}^r e^{-x^2/2}\dd x$ for $r \in \R$.
\end{corollary}

An algebraic rate for the QCLT was proved in \cite[Theorem 1.3]{GPT-19}.
We remark that for the model of random walk in random conductances,  algebraic rates similar to ours was proved in \cite{AN-17} for dimensions $d\ge 3$.

\newcounter{stepctr}[theorem]
\newcommand{\steps}{\stepcounter{stepctr}{\bf Step \thestepctr.}}

\section{Large scale \texorpdfstring{$C^{0,1}$}{TEXT} and \texorpdfstring{$C^{1,1}$}{TEXT} estimates}
In this section, we apply the ideas of Avellaneda, Lin \cite{AL-87, AL-89} in the periodic setting to the  discrete random setting. 
The key idea is quite intuitive and natural: large-scale solutions of $L_\omega u(x)=\psi(\theta_x\omega)+f(x)$ are well-approximated by those of the homogenized equation with an algebraic rate thanks to Proposition \ref{thm:quant-homo}. 
As the latter are harmonic, they possess rather nice estimates (see Proposition \ref{prop:harmonic} below on the scaling property.) 
Therefore, by iterations, scalings and the triangle inequalities, the better regularity of the homogenized equation is inherited by the heterogeneous equation. 
It is crucial to note two points here.
Firstly, in each iteration step, the scaling is done by using the nice estimates in Proposition \ref{prop:harmonic} of the homogenized limit, and the triangle inequality and Proposition \ref{thm:quant-homo} are used to pass this estimate to the solution $u$ of the heterogeneous equation. 
Secondly, in the random setting, one can only go down to radii greater than the homogenization radius in the iterations, which therefore gives us only large scale estimates.
The generalization of this idea to the random non-divergence form PDE setting was first done by Armstrong, Lin \cite{AL-17} who made the observation that an algebraic rate is sufficient for such an iteration.

The main result in this section, Theorem \ref{thm:c11}, can be considered as a discrete version of  Armstrong, Lin \cite[Theorem 3.1, Corollary 3.4]{AL-17} in terms of the large scale $C^{0,1}$ and $C^{1,1}$ regularity.  We remark that for $\omega$-harmonic functions in a genuinely $d$-dimensional balanced environment, a $C^{0,1-\error}$ regularity was achieved by Berger, Cohen, Deuschel, Guo \cite[Corollary 1.4]{BCDG-18} using coupling arguments.

As can be seen in the following Subsection \ref{subsec:regularity},  this sort of compactness argument, although is applied to the random setting here, is deterministic in its core.

\subsection{Some regularity properties of deterministic functions}\label{subsec:regularity}
This subsection contains the key tools for the compactness arguments used in our paper. It is completely deterministic and can be read independently of other parts of the paper.  The lemmas presented here concern large scale $C^{k,1}$, $k\ge 0$, properties of deterministic functions. 

For any function $f$ on a set $A$ and $\alpha\in(0,1]$, define
\[
\osc_A f:=\sup_{x,y\in A}|f(x)-f(y)|,
\qquad
[f]_{\alpha;A}=\sup_{x,y\in A, x\neq y}\frac{|f(x)-f(y)|}{|x-y|^\alpha},
\]
and, if $A$ is a finite set, for $p\in(0,\infty)$, we define 
\[
\norm{f}_{p;A}:=\left(\frac{1}{\#A}\sum_{x\in A}|f|^p\right)^{1/p}, \quad
\norm{f}_{\infty;A}=\max_{x\in A}|f(x)|.
\]

For any $j\ge 0$, let $\fct_j$ denote the set of $j$-th order polynomials, with $\fct_0=\R$. In fact,  in our paper we will only use the cases $j=0,1,2$. 

Define,  for function $f:\R^d\to\R$ and a bounded set $A\subset\R^d$,  $j\ge 1$, 
\begin{equation}\label{eq:def-der}
\der{A}^{j}(f)=\inf_{p\in\fct_{j-1}}\sup_{A}|f-p|=\frac{1}{2}\inf_{p\in\fct_{j-1}}\osc_{A}(f-p).
\end{equation}
$\der{A}^j$ satisfies the triangle inequality. Namely, $\der{A}^j(f\pm g)\le\der{A}^j(f)+\der{A}^j(g)$. 
When $A=B_R$ is the discrete ball, $R>0$,  we simply write 
\[\der{R}^j:=\der{B_R}^j.\] 
Note that for $j\ge 1$,  the above term normalized
\begin{equation}\label{eq:def-nder}
\nder{R}^j(f):=\frac{\der{R}^j(f)}{R^j}
\end{equation}
is a large scale analogue of the $j$-th order derivative.

For any $r>0$,  $\theta\in(0,\tfrac13)$, define a sequence of exponentially increasing radii $(r_k)_{k\ge 0}$ by
\[
r_k=r_k(r,\theta):=\theta^{-k}r, \quad k\ge 0.
\]

The following elementary lemma confirms the intuition that ``the integral of the $(j+1)$-th derivative is the $j$-th derivative".

\begin{lemma}\label{lem:integral}
For any function $f:\Z^d\to\R$ and $r>0, \theta\in(0,\tfrac13), n\in\N, j\ge 1$, 
\[
\nder{r_0}^j(f)\le \nder{r_n}^j(f)+3\theta^{-j}\sum_{k=0}^n r_k\nder{r_k}^{j+1}(f).
\]
\end{lemma}
\begin{proof}
For any $j$-th order homogeneous polynomials $p,q\in\fct_{j}$,  $R>r>0$, by the triangle inequality,
\begin{align*}
\der{r}^j(p)
&\le \der{r}^j(q)+\der{r}^j(p-q)\\
&\le 
(\tfrac{r}{R})^j\der{R}^j(q)+\der{r}^j(f-p)+\der{r}^j(f-q)
\end{align*}
where in the second inequality we used the fact that $\der{r}^j(q)= (\tfrac{r}{R})^j\der{R}^j(q)$ for all $j$-th order homogeneous polynomial $q$.  Hence, by the inequality above,
\begin{align*}
\der{r}^j(f)
&\le 
\der{r}^j(f-p)+\der{r}^j(p)\\
&\le 
2\der{r}^j(f-p)+\der{r}^j(f-q)+(\tfrac{r}{R})^j\der{R}^j(q)\\
&\le 
2\der{r}^j(f-p)+\der{r}^j(f-q)+(\tfrac{r}{R})^j[\der{R}^j(f)+\der{R}^j(f-q)]\\
&\le 
2[\der{r}^j(f-p)+\der{R}^j(f-q)]+(\tfrac{r}{R})^j\der{R}^j(f).
\end{align*}
Taking infimum over all $j$-th order homogeneous polynomials $p,q\in\fct_j$, we get
\[
\der{r}^j(f)
\le 
2[\der{r}^{j+1}(f)+\der{R}^{j+1}(f)]+(\tfrac{r}{R})^j\der{R}^j(f).
\]
Replacing $r, R$ by $r_k, r_{k+1}$ ,  and using notation \eqref{eq:def-nder}, the above inequality yields
\[
\nder{r_k}^j(f)-\nder{r_{k+1}}^{j}(f)
\le 
2[r_k\nder{r_k}^{j+1}(f)+\theta^{-j}r_{k+1}\nder{r_{k+1}}^{j+1}(f)].
\]
Summing both sides over $k=0,\ldots,n-1$,  the lemma is proved.
\end{proof}

The following lemma will be crucially employed later in our derivation of large scale regularity estimates in Subsection \ref{sec:regularity}.

\begin{lemma}\label{lem:deterministic}
Let $j\ge 1$, $m\in\N$, $r,\alpha>0$. Let $A_r\ge 0$ be a constant depending on $r$. If  for $f:\Z^d\to\R$ , $k=0,\ldots,m-1$, and  all $\theta\in(0,\tfrac13)$,
\begin{equation}\label{eq:iteration-assum}
\der{r_k}^{j+1}(f)\lesssim_j
\theta^{j+1}\der{r_{k+1}}^{j+1}(f)+r_{k+1}^{-\alpha}\der{r_{k+1}}^j(f)+r_{k+1}^jA_{r_{k+1}},
\end{equation}
then there exist $\theta=\theta(j), N=N(j,\alpha)$ such that, for $N\le r\le R\le r_m$,
\[
\der{r}^j(f)\le 13\theta^{-2j}\left(\tfrac{r}{R}\right)^j\der{R}^j(f)+\sum_{k\ge 1:r_k\le R}A_{r_k}.
\]
\end{lemma}

\begin{proof}
For the simplicity of notations, we suppress the dependency on $f$.  Let $n=n(R,\theta)\le m$ be such that $r_n\le R<r_{n+1}$. Display \eqref{eq:iteration-assum} is equivalent to 
\[
r_k\nder{r_k}^{j+1}\lesssim_j 
\theta r_{k+1}\nder{r_{k+1}}^{j+1}+\theta^{-j}r_{k+1}^{-\alpha}\nder{r_{k+1}}^j+\theta^{-j}A_{r_{k+1}}.
\]
Summing this inequality over $k=0,\ldots,n-1$, we have
\begin{equation}\label{eq:221130-1}
\sum_{k=0}^{n-1}r_k\nder{r_k}^{j+1}
\lesssim_j 
\theta\sum_{k=1}^n r_k\nder{r_k}^{j+1}
+\theta^{-j}\sum_{k=1}^n r_k^{-\alpha}\nder{r_{k}}^j+\theta^{-j}\sum_{k=1}^n A_{r_k}.
\end{equation}
Moreover,  by Lemma~\ref{lem:integral},  $\nder{r_{k}}^j\le \nder{r_{n}}^j+3\theta^{-j}\sum_{\ell=k}^nr_\ell\nder{r_\ell}^{j+1}$. Hence
\begin{align}\label{eq:221130-2}
\sum_{k=1}^n r_k^{-\alpha}\nder{r_{k}}^j
&\le 
\sum_{k=1}^n r_k^{-\alpha}
\left(\nder{r_{n}}^j+3\theta^{-j}\sum_{\ell=k}^nr_\ell\nder{r_\ell}^{j+1}\right)\nn\\
&\le 
C_\alpha r^{-\alpha}\nder{r_{n}}^j
+C_\alpha\theta^{-j} r^{-\alpha}\sum_{\ell=1}^nr_\ell\nder{r_\ell}^{j+1},
\end{align}
where $C_\alpha=1-3^{-\alpha}$. Choosing $\theta=\theta(j)\in(0,\tfrac13)$ sufficiently small, 
when $r\ge N$ for some $N=N(j,\alpha)$, we get from  \eqref{eq:221130-1} and \eqref{eq:221130-2} that
\[
\sum_{k=0}^{n-1}r_k\nder{r_k}^{j+1}
\le 
\frac{1}{2}\sum_{k=1}^n r_k\nder{r_k}^{j+1}+\nder{r_{n}}^j+C_j\theta^{-j}\sum_{k=1}^n A_{r_k}
\]
which implies (Note that $r_n\nder{r_n}^{j+1}\le \nder{r_n}^{j}$)
\[
\sum_{k=0}^{n}r_k\nder{r_k}^{j+1}
\le 4\nder{r_{n}}^j+C_j\theta^{-j}\sum_{k=1}^n A_{r_k}.
\]
This inequality, together with Lemma~\ref{lem:integral}, yields for $r\ge N$,  $\theta=\theta(j)\in(0,\tfrac13)$,
\[
\nder{r_0}^j\le 13\theta^{-j}\nder{r_n}^j+C_j\theta^{-2j}\sum_{k=1}^n A_{r_k}
\le 
 13\theta^{-2j}\nder{R}^j+\sum_{k=1}^n A_{r_k}.
\]
The lemma is proved.
\end{proof}

\begin{remark}
In this subsection we consider $f$ as a function on $\Z^d$ and defined $\fct_j$ to be the set of $j$-th order polynomials just for our convenience.   One may let $f$ be a function on $\R^d$ and  redefine $\fct_j$'s to be other sub-spaces of the polynomials (e.g., the set of harmonic polynomials) and Lemmas \ref{lem:integral}, \ref{lem:deterministic} still hold.
\end{remark}


\subsection{Large scale regularity}\label{sec:regularity}
The goal of this section is to apply Lemma~\ref{lem:deterministic} to obtain $C^{0,1}$ and $C^{1,1}$ regularities for the heterogeneous equations in our random setting.
\begin{theorem}
\label{thm:c11}
Assume {\rm(A1), (A2)},  and that $\psi$ is a local function. Let $R\ge 1$.  
There exists $\alpha=\alpha(d,\kappa)\in(0,\tfrac13)$ such that,
for any any $u$ with $L_\omega u(x)=\psi(\theta_x\omega)+f(x)$ on $B_R$, $j\in\{1,2\}$, $\tx\le r<R$, 
\begin{equation}
\label{eq:oscillation-estimates}
\frac{1}{r^{j}}\inf_{p\in\fct_{j-1}}\osc_{B_r}(u-p)
\lesssim
\frac{1}{R^{j}}
\inf_{p\in\fct_{j-1}}\osc_{B_R}(u-p)+A_{j},
\end{equation}
where the terms $A_j=A_j(R,r)$ have the following bounds (for any $\sigma\in(0,1]$)
\begin{align*}
&A_1\le R^{1-\alpha}\norm{\psi}_\infty+R\norm{f}_\infty
\text{ and }
A_1\le R^{1-\alpha}\norm{\psi+f(0)}_\infty+R^{1+\sigma}[f]_{\sigma;B_R},
\\
&A_2\le r^{-\alpha}\norm{\psi}_\infty+\log(\tfrac{R}{r})\norm{f}_\infty
\text{ and }
A_2\le r^{-\alpha}\norm{\psi+f(0)}_\infty+R^\sigma[f]_{\sigma;B_R}.
\end{align*}
In particular,  recalling the operator $\nabla^2_i$ in \eqref{eq:def-nabla},  for any $R>1$, $j=1,2$,
\begin{equation}\label{eq:c2}
|\nabla^ju(0)|
\lesssim
(\frac{\tx}{R})^j 
\left(
\norm{u-u(0)}_{1;B_R}+R^2\norm{\psi+f(0)}_\infty+R^{2+\sigma}[f]_{\sigma;B_{R}}.
\right)
\end{equation}
\end{theorem}

\begin{remark}
A weakness of Theorem~\ref{thm:c11} is that estimate \eqref{eq:oscillation-estimates} is not applicable when $\psi$ is not a local function, in which case $\psi$ is forced to be absorbed into the term $f(x)$ which usually yields unsatisfactory bounds.
\end{remark}

As a consequence of \eqref{eq:c2}, any $\omega$-harmonic function on $\Z^d$ with sublinear growth is a constant. That is, if $L_\omega u=0$ on $\Z^d$ and $\max_{B_R}|u|=o(R)$ for all $R>0$, then $u$ is constant.
To prove Theorem~\ref{thm:c11}, it suffices to prove the following lemma.

\begin{lemma}
\label{lem:before-induction}
There exists $\gamma=\gamma(d,\kappa)$ such that, for $R\ge \tx$,  $\theta\in(0,\tfrac{1}{3})$, $1\le j\le 3$ and any $u$ with $L_\omega u(x)=\psi(\theta_x\omega)+f(x)$ for $x\in B_R$, we have
\[
\der{\theta R}^{j}(u)
\lesssim 
R^{-\gamma\beta}\der{R}^2(u)+\theta^{j}\der{R}^{j}(u)+R^{2-\gamma\beta}\norm{\psi}_{\infty}+R^2\norm{f}_{d;B_R}.
\]
\end{lemma}
The proof of Lemma~\ref{lem:before-induction} uses the following fact of deterministic harmonic functions.  For completeness, we include its proof in Section~\ref{asec:harmonic} of the Appendix.
\begin{proposition}\label{prop:harmonic}
Recall the notation in \eqref{eq:def-der}. Let $c_0$ be a constant.
Let $v$ be a function satisfying $\tr\bar aD^2v=c_0$ in $\bar\B_R$. Then, for $\theta\in(0,\tfrac13)$, $j\in\{1,2,3\}$ and $R\ge 1$, 
\begin{equation}\label{eq:221231-1}
\der{\B_{\theta R}}^j(v)\le C \theta^j\der{\B_{R/2}}^{j}(v).
\end{equation}
We also have
\begin{equation}\label{eq:221231-2}
\der{\B_{\theta R}}^j(v)
\lesssim
\tfrac{\theta^j}{R}(\sup_{\partial\B_{2R/3}}|v|+R^2|c_0|)+\theta^j\der{2R/3}^j(v).
\end{equation}
\end{proposition}
\begin{remark}
Property \eqref{eq:221231-1} for deterministic harmonic functions ($c_0=0$) was used in \cite[Lemma 3.3]{AL-17} to obtain regularities of the heterogeneous solution in the PDE setting.  Comparing to \cite[Corollary 3.4]{AL-17}, here by allowing $c_0\neq 0$ we will gain a tiny  improvement for the coefficient of $\norm{\psi+f(0)}_\infty$ in the $C^{0,1}$ estimate by an $R^{-\alpha}$ factor (cf. Theorem~\ref{thm:c11}).  Note that in the discrete setting, we will need \eqref{eq:221231-2} as well because of discretization.  It would be also clear later in Section~\ref{sec:mixing} that the $\log R$ factor in the bound of $A_2$ will help us achieve the $\log R$ factor in Theorem~\ref{thm:ave-rho-quant}.
\end{remark}


\begin{proof}
[Proof of Lemma~\ref{lem:before-induction}]
By the H\"older estimate of Krylov-Safonov, there exists $\gamma=\gamma(d,\kappa)>0$ such that, for $r\in (0,R)$, 
\begin{equation}
\label{eq:osc}
\osc_{B_r}u
\lesssim
\left(\tfrac{r}{R}\right)^\gamma 
(\osc_{B_R}u+R^2\norm{\psi+f}_{d;B_R}).
\end{equation}
Note that this allows us to extend $u$ to be a function $\tilde{u}\in C^\gamma(\R^d)$ with $[\tilde{u}]_{\gamma;\R^d}= [u]_{\gamma;\bar B_{2R/3}}$. Indeed,  define the function $\tilde{u}$ as
\[
\tilde{u}(x)=\min_{y\in \bar B_{2R/3}}\left\{u(y)+|x-y|^\sigma[u]_{\sigma;\bar B_{2R/3}}\right\}.
\]
It is straightforward to check that $\tilde{u}=u$ in $\bar B_{2R/3}$ and $[\tilde{u}]_{\gamma;\R^d}\le [u]_{\gamma;\bar B_{2R/3}}$.
By  \eqref{eq:osc}, 
\begin{equation}\label{eq:holder-tilde}
[\tilde{u}]_{\gamma;\R^d}
=[u]_{\gamma;\bar B_{2R/3}}
\lesssim 
R^{-\gamma}(\max_{B_R}|u|+R^2\norm{\psi+f}_{d;B_R}).
\end{equation}

Let $\bar v:\bar\B_{2/3}\to\R$ be the solution of
\[
\left\{
\begin{array}{lr}
\frac{1}{2}\tr(\bar a D^2\bar v)=R^2\bar\psi &\text{ in }\B_{2/3}\\
\bar v(x)=\tilde u(Rx) &\text{ for }x\in\partial\B_{2/3}.
\end{array}
\right.
\]
First, write $A:=R^{2-\gamma\beta}\norm{\psi}_\infty+R^2\norm{f}_{d;B_R}$.  We will show that, for $R\ge \tx$,
\begin{equation}
\label{eq:quant-osc}
\max_{x\in B_{2R/3}}
|u(x)-\bar v(\tfrac{x}{R})|
\lesssim
R^{-\gamma\beta}\max_{B_R}|u|+A.
\end{equation}
To this end, let $u_1:\bar B_{2R/3}\to\R$  be the solution of 
\[
\left\{
\begin{array}{lr}
L_\omega u_1=\psi(\theta_x\omega) & \text{ in }B_{2R/3}\\
u_1(x)=\tilde u(\tfrac{2Rx}{3|x|}) & x\in\partial B_{2R/3}.
\end{array}
\right.
\]
By Proposition~\ref{thm:quant-homo} and \eqref{eq:holder-tilde}, when $R\ge \tx$,  noting that $[\tilde{u}(R\cdot )]_{\gamma;\R^d}\le R^{\gamma}[\tilde{u}(\cdot )]_{\gamma;\R^d}$, 
\begin{align*}
\max_{x\in B_{2R/3}}|u_1(x)-\bar v(\tfrac{x}{R})|
&\lesssim 
R^{-\gamma\beta}([\tilde{u}(R\cdot )]_{\gamma;\R^d}+R^2\norm{\psi}_\infty)\\
&\lesssim
R^{-\gamma\beta}(\max_{B_R}|u|+R^2\norm{\psi}_\infty+R^2\norm{f}_{d;B_R}).
\end{align*}
Moreover, by the ABP maximum principle,
\begin{align*}
\max_{B_{2R/3}}|u-u_1|
&\le \max_{x\in \partial B_{2R/3}}|u(x)-\tilde u(\tfrac{2Rx}{3|x|})|
+CR^2\norm{f}_{d;B_R}\\
&\lesssim
[\tilde{u}]_{\gamma;\R^d}+R^2\norm{f}_{d;B_R}\\
&\stackrel{\eqref{eq:holder-tilde}}{\lesssim}
R^{-\gamma}(\max_{B_R}|u|+R^2\norm{\psi}_{\infty})+R^2\norm{f}_{d;B_R}.
\end{align*}
Combining the two inequalities above, 
display \eqref{eq:quant-osc} is proved.

By the triangle inequality and Proposition~\ref{prop:harmonic}, for $1\le j\le 3$,
\begin{align}\label{eq:221201-1}
\der{\theta R}^{j}(u)
&\le \max_{B_{R/2}}|u-\bar v(\tfrac{\cdot}{R})|+\der{\theta R}^{j}(\bar v(\tfrac{\cdot}{R}))\nn\\
&\le 
 \max_{B_{R/2}}|u-\bar v(\tfrac{\cdot}{R})|+\tfrac{C\theta^j}{R}(\sup_{\partial\B_{2/3}}|\bar v|+R^2|\bar\psi|)+C\theta^j\der{2R/3}^j(\bar v(\tfrac{\cdot}{R}))\nn\\
 &\lesssim
  \max_{B_{2R/3}}|u-\bar v(\tfrac{\cdot}{R})|
  +\tfrac{\theta^j}{R}(\sup_{\partial\B_{2/3}}|\bar v|+R^2|\bar\psi|)+\theta^j\der{2R/3}^j( u).
 \end{align}
Since
$\sup_{\partial\B_{2/3}}|\bar v|=\sup_{\partial\B_{2R/3}}|\tilde{u}|\le \max_{B_{2R/3}}|u|+[\tilde{u}]_{\gamma;\bar{B}_{2R/3}}\le \max_{B_R}|u|+A$, 
by \eqref{eq:quant-osc} and \eqref{eq:221201-1}, we have, for $1\le j\le 3$,
\begin{align*}
\der{\theta R}^{j}(u)
\lesssim
R^{-\gamma\beta}\max_{B_{R}}|u|+A+\theta^j\der{R}^j(u).
\end{align*}
 
Finally, note that since every $p\in\fct_1$ is $\omega$-harmonic,  $(u-p)$ still solves
$L_\omega (u-p)=\psi(\theta_x\omega)+f(x)$ for $x\in B_R$.
Therefore, substituting $u$ by $(u-p)$ in the above inequality and optimizing over $p\in\fct_{1}$,
the lemma follows.
\end{proof}

\begin{proof}[Proof of Theorem~\ref{thm:c11}]
By Lemma~\ref{lem:before-induction} and Lemma~\ref{lem:deterministic},  there exists $\theta=\theta(d,\kappa)\in(0,\tfrac13)$ such that \eqref{eq:oscillation-estimates}  holds with
the terms $A_j$, $j\in\{1,2\}$ satisfying
\[
A_j=\sum_{k\ge 1: r_k\le R}r_k^{2-\alpha-j}\norm{\psi}_\infty+r_k^{2-j}\norm{f}_{d;B_{r_k}}.
\]
Note that $\norm{f-f(0)}_{d;B_r}\lesssim r^\sigma[f]_{\sigma;B_r}$ for all $\sigma\in(0,1]$. 
The bounds of $A_1,A_2$ in the theorem follow immediately.

To prove \eqref{eq:c2}, note that $|\nabla u(0)|\le\osc_{\bar B_1}u$ and 
 that for any $\ell\in\fct_1$, 
$
\abs{\nabla^2u(0)}
=\abs{\nabla^2(u-\ell)}
\lesssim
\osc_{\bar B_{\sqrt 2}}(u-\ell)$. Hence, by \eqref{eq:oscillation-estimates}, we get
\[
|\nabla^ju(0)|\lesssim(\frac{\tx}{R})^j 
\left(
\osc_{B_{R/2}}u+R^2\norm{\psi+f(0)}_\infty+R^{2+\sigma}[f]_{\sigma;B_{R/2}}
\right).
\]
By the Harnack inequality and the ABP inequality, we have
\begin{equation}\label{eq:200524-2}
\osc_{B_{R/2}}u\lesssim
\norm{u-u(0)}_{1;B_R}+R^2\norm{\psi+f}_{d;B_R}.
\end{equation}
Display \eqref{eq:c2} follows by using again $\norm{f-f(0)}_{d;B_R}\lesssim R^\sigma[f]_{\sigma;B_R}$ for $\sigma\in(0,1]$.
\end{proof}

\section{Mixing properties of the invariant measure}\label{sec:mixing}
The goal of this section is to investigate the mixing properties of the field $\{\rho_\omega(x):x\in\Z^d\}$ of the invariant measure. 
We will obtain a rate of convergence (Theorem \ref{thm:ave-rho-quant}) of the average of the invariant measure over balls $B_R$. We will also quantify the correlation of the field (Proposition~\ref{prop:correlation}).

The Efron-Stein inequality \eqref{eq:bblm} of Boucheron, Bousquet, and Massart \cite{BBLM-05} will be used in our derivation of quantitative estimates.

Let $\omega'(x), x\in\Z^d$, be i.i.d.  copies of $\omega(x), x\in\Z^d$. For any $y\in\Z^d$, let $\omega'_y\in\Omega$ be the environment such that
\[
\omega'_y(x)=
\left\{
\begin{array}{lr}
&\omega(x) \quad \text{ if }x\neq y,\\
&\omega'(y) \quad \text{ if }x=y.
\end{array}
\right.
\]
That is, $\omega'_y$ is a modification of $\omega$ only at location $y$. 
For any measurable function $Z$ of the environment $\omega$, we write, for $y\in\Z^d$,
\begin{equation}\label{eq:def_vert_der}
Z_y'=Z(\omega_y'), \quad
\vd{y} Z(\omega)=Z'_y-Z, 
\end{equation}
and set
\begin{equation}\label{def:v}
V(Z)=\sum_{y\in\Z^d}(\vd{y}Z)^2. 
\end{equation}
With abuse of notations, we enlarge the probability space and still use $\mb P$ to denote the distribution of both $\omega, \omega'$.

The $L_p$ version of Efron-Stein inequality in \cite[Theorem~3]{BBLM-05} states that
\begin{equation}\label{eq:bblm}
\mb E[|Z-\mb EZ|^q]\le Cq^{q/2}\mb E[V^{q/2}] \quad \text{ for }q\ge 2.
\end{equation}
The following variation of \eqref{eq:bblm} will be useful in our paper.
\begin{lemma}
\label{lem:efron-stein} Let $p>0$. 
Let $Z$ be a measurable function of the environment.  Assume that there exist  $f\in(0,\infty)^{\Z^d}$ with $F=[\sum_{y\in \Z^d}f(y)^2]^{1/2}<\infty$ and a random variable $\ms X>0$ with $\mb E[\exp(c\ms X^p)]<\infty$  such that 
$\mb E[\abs{\vd y Z/f(y)}^n)]\lesssim \mb E[\ms X^n]$ for all $y\in\Z^d$, $n\ge 1$.  Then there exists $C=C(p)>0$ such that
\[
\mb E\left[\exp\big(
C|F^{-1}(Z-\mb EZ)|^{2p/(2+p)}
\big)\right]
\lesssim 
E[\exp(c\ms X^p)].
\]
\end{lemma}
The proof uses the fact that for $\alpha\in[0,1)$, there exists $c=c(\alpha)>0$ such that
\begin{equation}
\label{eq:fact-analytical}
\sum_{n=1}^\infty
\frac{c^n}{n!}x^{n}n^{\alpha n}\le \exp(x^{1/(1-\alpha)}) \quad\text{ for all }x>0.
\end{equation}
Indeed, when $x>0$, putting $c=e^{-\alpha}/2$ and using inequality $\tfrac{n^n}{n!}\le e^n$, 
\begin{align*}
\sum_{n=1}^\infty
\frac{c^n}{n!}x^{n}n^{\alpha n}
\le 
\sum_{n=1}^\infty 2^{-n}\frac{x^n}{(n!)^{1-\alpha}}
=\sum_{n=1}^\infty 2^{-n}\big(\frac{x^{n/(1-\alpha)}}{n!}\big)^{1-\alpha}
\le 
\exp(x^{1/(1-\alpha)}),
\end{align*}
where we used $\tfrac{y^n}{n!}\le e^y$ for $y\ge 0$ in the last inequality.  
\begin{proof}
Set $\ms X_y:=\abs{\vd y Z/f(y)}$.
By Jensen's inequality,  for any $q\ge 2$,
\[
V^{q/2}
=F^q\big(\sum_y\ms X_y^2\frac{f(y)^2}{F^2}\big)^{q/2}
\le 
F^{q}\sum_y\ms X_y^q\frac{f(y)^2}{F^2}.
\]
Taking expectations on both sides, and using \eqref{eq:bblm}, we get, for $q\ge 2$,
\[
\mb E\left[
|(Z-\mb EZ)/F|^{q}
\right]
\lesssim 
q^{q/2}\mb E[
\ms X^q
].
\]
The lemma then follows by using the fact \eqref{eq:fact-analytical}.
\end{proof}

The following facts concerning the operation $\vd y$ will be useful for later computations.
For any measurable functions $f,g$ on $\Omega$, 
\begin{align}\label{eq:inte-by-part}
&\mb E[(\vd{y}f)g]=\mb E[f(\vd{y}g)],\\
& \vd{y}(fg)=(\vd{y}f)g+f_y'(\vd{y}g)=(\vd{y}f)g_y'+f(\vd{y}g).\label{eq:prod-rule}
\end{align}
Recall that, as defined in \eqref{eq:def-a},  $a(y)=a_\omega(y)$ denotes the diagonal matrix
\[
a(y)=\frac{\omega(y)}{\tr\omega(y)}=\diag[2\omega(y,y+e_1),\ldots, 2\omega(y,y+e_d)].
\]

It follows  from the product rule \eqref{eq:prod-rule} that for any $u=u_\omega$ that solves $L_\omega u(x)=\xi(x,\omega)$ for some function $\xi:\Z^d\times\Omega\to\R$,  the vertical derivative $\vd y u(x)$ satisfies
\begin{equation}\label{eq:vdlaplace}
L_{\omega}(\vd y u)(x)=-\tfrac12\tr(\vd ya(y)\nabla^2 u'_y(y))\mathbbm1_{x=y}+\vd y\xi(x,\omega),
\end{equation}
where  we used the fact that $\vd y\omega(x)=0$ for $x\neq y$,  and $u_y'(x):=u_{\omega'_y}(x)$. 

\subsection{A sensitivity estimate of the invariant measure}
The main contribution of this subsection is a formula for the ``vertical derivative" of the invariant measure $\rho$.

\begin{definition}
For $t>0$, we let $V(t,\omega)=\sum_{x\in\Z^d}p_t^\omega(x,0)$. Let $\Q_t$ be the probability measure on $\Omega$ defined by 
\[
\Q_t(\dd\omega)=V(t,\omega)\mb P(\dd\omega).
\]
\end{definition}

We remark that by Theorem~\ref{thm:hk-bounds},  
\begin{equation}
\label{eq:V-bound}
V(t,\omega)\lesssim\tx^{d-1} \quad \text{ for $\mb P$-a.e.  }\omega
\end{equation}
and so $\Q_t$ is well-defined.  Note that for any bounded measurable function $\zeta$ on $\Omega$,  we have $E_{\Q_t}[\zeta]=\mb E[P_t\zeta]$. In other words,  $\Q_t$ is the distribution of the environment viewed from the particle at time $t$.  It is natural to expect that $\Q_t\to\Q$ as $t\to\infty$.

For any function $u$ of the environment,  we denote by $u_\omega$ the corresponding function on $\Z^d$ defined by $u_\omega(x):=u(\theta_x\omega)$.

\begin{lemma}
\label{lem:weak-conv}
As $t\to \infty$, $\Q_t$ converges weakly to $\Q$.
\end{lemma}
\begin{proof}
Since $\{\Q_t\}$ is a sequence of probability measures on the compact space $\Omega$, it has a weak convergent subsequence $\{\Q_{t_k}\}$ which has a weak limit $\Q_\infty$.

To prove that $\Q_\infty$ is an invariant measure for the Markov chain $(\theta_{Y_t}\omega)$, it suffices to show that for any bounded measurable function $f$ on $\Omega$, 
\[
E_{\Q_\infty}[L_\omega f_\omega(0)]=0.
\]
Indeed, by the translation invariance of the measure $\mb P$, for any $e$ with $|e|=1$,
\begin{align}\label{eq:221219-1}
\mb E[\omega(0,e)V(t,\omega)f(\theta_e\omega)]
&=
\mb E[
\omega(-e,0)V(t,\theta_{-e}\omega)f(\omega)]\nn
\\&
=\mb E[\rho_\omega(0)\omega^*(0,-e)\tilde{V}(t,\theta_{-e}\omega)f(\omega)],
\end{align}
where $\omega^*(x,y):=\rho_\omega(y)\omega(y,x)/\rho_\omega(x)$ denotes the {\it adjoint} of $\omega$,  cf.  e.g., \cite{DG-19}, and $\tilde{V}(t,\omega):=V(t,\omega)/\rho(\omega)$.  Noting that 
$\sum_{y}\omega^*(x,y)=1$,  we have
\begin{align}\label{eq:QLf}
E_{\Q_t}[L_\omega f_\omega(0)]
&=\mb E[V(t,\omega)\sum_e \omega(0,e)[f(\theta_e\omega)-f(\omega)]]\nn\\
&\stackrel{\eqref{eq:221219-1}}{=}
\mb E[\rho_\omega(0)\sum_e\omega^*(0,e)[\tilde{V}(t,\theta_e\omega)-\tilde{V}(t,\omega)]f(\omega)]\nn\\
&=E_\Q[f(\omega)L_{\omega^*}\tilde{V}_\omega(t,0)],
\end{align}
where $\tilde{V}_\omega(t,x):=\tilde{V}(t,\theta_x\omega)$,  and $L_{\omega^*}$ only acts on the spatial ($\Z^d$) coordinate of the function $\tilde{V}_\omega:\R\times\Z^d\to\R$ of space and time.
Observe that $\tilde{V}_\omega$ solves the parabolic equation
\[
(\partial_t-L_{\omega^*})\tilde{V}_\omega=0 \quad
\text{ in }(0,\infty)\times\Z^d.
\]
By the H\"older estimate \cite[Corollary 7]{DG-19} and the Harnack inequality \cite[Theorem 6]{DG-19} for the operator $(\partial_t-L_{\omega^*})$,   there exists $\gamma=\gamma(d,\kappa)>0$ such that, for $t>1$,
\begin{align}\label{eq:osc-v-tilde}
\max_{e:|e|=1}|\tilde{V}_\omega(t,e)-\tilde{V}_\omega(t,0)|
&\lesssim
t^{-\gamma}\sup_{(s,x)\in (0.5t,t)\times B_{\sqrt t}}\tilde{V}_\omega(s,x)\nn\\
&\lesssim
t^{-\gamma}\tilde{V}_\omega(2t,0)\\
&\stackrel{\eqref{eq:V-bound}}
\lesssim t^{-\gamma}\rho^{-1}\tx^{d-1}.\nn
\end{align}
Thus, by \eqref{eq:QLf}, $\abs{\E_{\Q_t}[L_\omega f_\omega(0)]}
\lesssim  t^{-\gamma}\mb E[\tx^{d-1}]\norm{f}_\infty\lesssim  t^{-\gamma}\norm{f}_\infty$ .
In particular, for any bounded measurable function $f$ on $\Omega$,
\[
E_{\Q_\infty}[L_\omega f_\omega(0)]
=\lim_{k\to\infty}
E_{\Q_{t_k}}[L_\omega f_\omega(0)]=0
\]
which implies that $\Q_\infty$ is an invariant measure for the Markov chain $(\theta_{Y_t}\omega)$.  Moreover,  for any bounded measurable function $f:\Omega\to\R$ and $p>0$, 
\[
E_{\Q_t}[f]=\mb E[V_\omega(t,0)f(\omega)]\lesssim\mb E[\tx^{d-1}f]
\lesssim_p\norm{f}_{L^p(\mb P)},
\]
and so $E_{\Q_\infty}[f]\lesssim_p\norm{f}_{L^p(\mb P)}$ which implies $\Q\ll\mb P$. Therefore, by the same argument as in \cite[(4)]{GZ-12},  we have $\Q_\infty=\Q$.
\end{proof}

Before stating the formula for $\vd{y}\rho$ in the following proposition, 
we remark that although the global Green function $G^\omega(x,y)$ is only defined for $d\ge 3$, the second order difference $\nabla_{i;1}^2 G(x,y)$ can be defined for all dimensions, where  $\nabla_{i;1}^2$ is $\nabla_{i}^2$ applied to the first $\Z^d$ coordinate. That is, for any fixed $y$, $\nabla_{i;1}^2G(\cdot, y):=\nabla_{i}^2G(\cdot, y)$. Indeed,  recalling $A(x,y)$ in \eqref{eq:def-potential}, we can set
\[
\nabla_{i;1}^2  G(x,y):=-\nabla_{i;1}^2  A(x,y) \qquad\text{ when }d=2.
\]
Since $G(\cdot,\cdot)$ is not defined in Definition~\ref{def:green} for $d=2$, for the convenience of notations,  throughout this section we denote
\begin{equation}\label{eq:nabla-d2}
G(x,y):=-A(x,y), \text{ and }G(x,S)=-\sum_{y\in S}A(x,y)\quad \text{ when }d=2.
\end{equation}

\begin{proposition}
\label{prop:sensitivity-rho}
Recall the vertical derivative $\vd y$ and the notation $\omega'_y$ as in \eqref{eq:def_vert_der}. 
For any $x,y\in\Z^d$, $\mb P$-almost surely,
\[
\vd{y}\rho_\omega(x)=\rho_\omega(y)\sum_{i=1}^d(\vd{y}\omega)(y,y+e_i)\nabla_{i;1}^2 G^{\omega_y'}(y,x).
\]
\end{proposition}

\begin{proof}
It suffices to consider the case $x=0$. The formula for general $x$ will follow from the fact that
$\vd{y}\rho_\omega(x)
=\vd{y-x}\rho_{\theta_x\omega}(0)$.
We divide the proof into several steps.

\medskip

\noindent {\bf Step 1.} First, we will show a formula for $\vd{y}V(t,\omega)$:
\begin{equation}
\label{eq:vd-V-formula}
\vd{y}V(t,\omega)=\int_0^t V_\omega(t-s,y)\sum_{i=1}^d(\vd{y}\omega)(y,y+e_i)\nabla_i^2 p_s^{\omega_y'}(y,0)\dd s,
\end{equation}
where $V_\omega(s,y)=V(s,\theta_y\omega)$, and $V_y'(s,y)=V_{\omega_y'}(s,y)$.

Indeed,  notice that $u(x,t)=p_t^\omega(x,0)$ satisfies 	$u(x,0)=\mathbbm1_{x=0}$ and
\begin{equation}
\label{eq:heat-eq-u}
(\partial_t-L_\omega)u(x,t)=0\quad\text{ for }(x,t)\in\Z^d\times(0,\infty).
\end{equation}
By the equation above and formula  \eqref{eq:vdlaplace},  we have
\[
L_\omega(\vd{y}u)(x,t)
=-\sum_{i=1}^d (\vd{y}\omega)(y,y+e_i)\nabla_i^2 u_y'(y,t)\mathbbm1_{x=y}+\partial_t [\vd{y}u(x,t)].
\]
Hence, for every fixed $y\in\Z^d$,  $\vd{y}u(x,t)$ solves the heat equation
\[
\left\{
\begin{array}{lr}
(\partial_t-L_\omega)\vd{y}u(x,t)=\sum_{i=1}^d (\vd{y}\omega)(y,y+e_i)\nabla_i^2 u_y'(y,t)\mathbbm1_{x=y}
&\text{ for }(x,t)\in \Z^d\times(0,\infty),\\
\vd{y}u(x,0)=0 &\text{ for }x\in\Z^d
\end{array}
\right.
\]
whose solution can be represented by Duhamel's formula
\begin{align*}
\vd{y}u(x,t)
=
\sum_{i=1}^d\int_0^t p_{t-s}^\omega(x,y)
(\vd{y}\omega)(y,y+e_i)\nabla_i^2 u_y'(y,s)\dd s
\end{align*}
Recall that $u(x,t)=p_t^\omega(x,0)$. Summing the above equality over all $x\in\Z^d$,  we obtain
formula \eqref{eq:vd-V-formula}. 

\noindent{\bf Step 2.} We claim that the integrand in \eqref{eq:vd-V-formula} has the following bound: $\forall s\in(0,t)$,
\begin{equation}
\label{eq:uniform-bd}
\Abs{V_\omega(t-s,y)\sum_{i=1}^d(\vd{y}\omega)(y,y+e_i)\nabla_i^2 p_s^{\omega_y'}(y,0)}
\lesssim
(\tx_y\tx_y')^{d-1}(1+s)^{-\gamma-0.5d}.
\end{equation}
Indeed,  by \eqref{eq:heat-eq-u} and applying the Harnack inequality (Corollary~\ref{acor:Kry-Saf}) for the operator $(\partial_t-L_\omega)$ in a similar manner as in \eqref{eq:osc-v-tilde},  we have
\begin{align*}
&\Abs{V_\omega(t-s,y)\sum_{i=1}^d(\vd{y}\omega)(y,y+e_i)\nabla_i^2 p_s^{\omega_y'}(y,0)}\\
&\stackrel{\eqref{eq:V-bound}}\lesssim 
\tx_y^{d-1}
\osc_{\bar B_1(y)}p_s^{\omega_y'}(\cdot,0)
\lesssim 
\tx_y^{d-1}
s^{-\gamma}p_{2s}^{\omega_y'}(y,0)
\end{align*}
for $s>1$.  Hence \eqref{eq:uniform-bd} follows from Theorem~\ref{thm:hk-bounds} when $s>1$.
When $s\le 1$,  \eqref{eq:uniform-bd} is a trivial consequence of \eqref{eq:V-bound} since $\abs{\nabla_i^2 p_s^{\omega_y'}(y,0)}\le 2$.  Display \eqref{eq:uniform-bd} is proved.

\noindent{\bf Step 3.} 
For any bounded measurable function $f$ on $\Omega$, by Lemma~\ref{lem:weak-conv} and \eqref{eq:inte-by-part},
\begin{align*}
\mb E[(\vd{y}\rho)f]=\mb E[\rho(\vd{y}f)]
=\lim_{t\to\infty}\mb E[V(t,\omega)(\vd{y}f)]
=\lim_{t\to\infty}\mb E[(\vd{y}V(t,\omega))f].
\end{align*}
Furthermore,  by \eqref{eq:vd-V-formula}, \eqref{eq:uniform-bd}, and the dominated convergence theorem,  we get
\begin{align*}
\mb E[(\vd{y}\rho)f]
&=
\int_0^\infty 
\lim_{t\to\infty}\mb E\left[
V_\omega(t-s,y)\sum_{i=1}^d(\vd{y}\omega)(y,y+e_i)\nabla_i^2 p_s^{\omega_y'}(y,0)
\mathbbm1_{t>s}f
\right]\dd s\\
&\stackrel{Lemma~\ref{lem:weak-conv}}=
\int_0^\infty 
\mb E\left[
\rho f\sum_{i=1}^d(\vd{y}\omega)(y,y+e_i)\nabla_i^2 p_s^{\omega_y'}(y,0)
\right]\dd s\\
&=
\mb E\left[
\rho f\sum_{i=1}^d(\vd{y}\omega)(y,y+e_i)\nabla_i^2 G^{\omega_y'}(y,0)
\right].
\end{align*}
Proposition~\ref{prop:sensitivity-rho} follows.
\end{proof}

\subsection{Rate of convergence for the average of the invariant measure: Proof of Theorem~\ref{thm:ave-rho-quant}}

Now we will proceed to prove one of the main theorems in this paper, Theorem~\ref{thm:ave-rho-quant}.

It will be clear in the proof that the $\log R$ term in the $C^{1,1}$ bound of Theorem~\ref{thm:c11}
is important for us to obtain the logarithmic term in Theorem~\ref{thm:ave-rho-quant}.

The following simple fact of random variables will be used in the proof.
\begin{lemma}\label{fact:max}
Let $p>0$ and let $(X_i)_{i=1}^\infty$ be non-negative random variables with $E[\exp(X_i^p)]<C$ for all $i \geq 1$.  
Then, for $M_n=\max_{1\le i\le n}X_i-(2\log n)^{1/p}$, we have, for all $n\geq 1$,
\[
E[\exp(cM_n^p)]<C.
\]
\end{lemma}

\begin{proof}
[Proof of Theorem~\ref{thm:ave-rho-quant}]
We divide the proof into several steps.

\noindent {\bf Step 1.} 
  Let $u:\R_+\to\R_+$ be the function
 \begin{equation}\label{eq:def-u}
u(r)=\left\{
\begin{array}{lr}
\log (r+1) &\text{ when }d=2\\
(r+1)^{2-d} &\text{ when }d\ge 3.
\end{array}
\right. 
 \end{equation}
 
 We will show that for any $\error>0,R\ge 2$, there exists a random variable $\tx^*(\omega)=\tx^*(R,\omega;d,\kappa,\error)>0$ with $\mb E[\exp(c\tx^{*d-\error})]<C$ such that, $\mb P$-a.s.,
 \begin{equation}
 \label{eq:bd-g-y-br}
\osc_{B_{(|y|+R)/2}(y)}G^\omega(\cdot,B_R)
\lesssim
\left\{
\begin{array}{lr}
\tx^{*d-1}u(|y|)R^d &\text{ if }|y|>4R\\
\tx^{*d-1}R^2\log R &\text{ if }|y|\le 4R.
\end{array}
\right.
 \end{equation}
 Indeed, by Theorem~\ref{thm:Green-bound},  for $z\in\Z^d$,
\begin{equation}\label{eq:g-ball-bd}
\abs{G(z,B_R)}\lesssim\sum_{x\in B_R}\tx_x^{d-1}u(|z-x|).
\end{equation}

When $|y|>4R$,  $u(|z-x|)\asymp u(|y|)$ for all $x\in B_R, z\in B_{(|y|+R)/2}(y)$,  and so
\begin{equation}\label{eq:g-ball-bd-1}
\abs{G(z,B_R)}
\stackrel{\eqref{eq:g-ball-bd}}\lesssim \sum_{x\in B_R}\tx_x^{d-1}u(|y|)
\lesssim
\tx_1^{*d-1}u(|y|)R^d,
\end{equation}
where $\tx_1^*=(\tfrac1{|B_R|}\sum_{x\in B_R}\tx_x^{d-1})^{1/(d-1)}$.

When $|y|\le 4R$ and $d=2$,  for all $z\in B_{(|y|+R)/2}(y)$, we have $u(|z-x|)\lesssim\log R$ $\forall x\in B_R$, and so
\begin{equation}
\label{eq:g-ball-bd-2}
|G(z,B_R)|\stackrel{\eqref{eq:g-ball-bd}}\lesssim \log R\sum_{B_R}\tx_x^{d-1}=\tx_1^{*d-1}R^2\log R.
\end{equation}

When  $|y|\le 4R$ and $d\ge 3$, for all $z\in B_{(|y|+R)/2}(y)$,  \eqref{eq:g-ball-bd} yields
\begin{align}\label{eq:g-ball-bd-3}
\abs{G(z,B_R)}
&\lesssim
[\tx_2^*+(2d\log R)^{1/(d-1)}]^{d-1}\sum_{x\in B_{4R}}u(|x|)\nn\\
&\lesssim 
(\tx_2^{*d-1}+\log R)R^2,
\end{align}
where $\tx_2^*=[\max_{x\in B_R}\tx_x-(2d\log R)^{1/(d-1)}]_+$.  
Recall $\tx=\tx(\omega,d,\kappa,\error)$ in Theorem~\ref{thm:hk-bounds}.
Noting that for $t>1$ and $p=d-\error>d-1$, by Lemma~\ref{fact:max} we have
$\mb E[\exp(c\tx_2^{*d-\error})]\le C$.  Note also that, by Jensen's inequality,  $\mb E[\exp(c\tx_1^{d-\error})]\le C$.

Setting $\tx^*=\tx^*_1+\tx^*_2$,  \eqref{eq:bd-g-y-br} follows from \eqref{eq:g-ball-bd-1}, \eqref{eq:g-ball-bd-2}, and \eqref{eq:g-ball-bd-3}.

\noindent{\bf Step 2.} Next, we will show that
 \begin{equation}
 \label{eq:green-c2-bd}
|\nabla^2 G^\omega(y,B_R)|
\lesssim
\left\{
\begin{array}{lr}
\tx_y^2\tx^{*d-1}|y|^{-2}u(|y|)R^d &\text{ if }|y|>4R\\
\tx_y^2\tx^{*d-1}\log R &\text{ if }|y|\le 4R,
\end{array}
\right.
 \end{equation}
 where the operator $\nabla^2$ is only applied to the first $\Z^d$ coordinate of $G(\cdot,\cdot)$.  

When $|y|>4R$, by Theorem~\ref{thm:c11} and \eqref{eq:bd-g-y-br},
\begin{align*}
|\nabla^2 G(y,B_R)|
\lesssim
\frac{\tx_y^2}{|y|^2}\osc_{B_{|y|/2}(y)}G(\cdot,B_R)
\lesssim
\tx_y^2\tx_1^{*d-1}|y|^{-2}u(|y|)R^d.
\end{align*}

When $|y|\le 4R$,  applying Theorem~\ref{thm:c11} (with $\psi=0$, $f=-\mathbbm1_{B_R}$) again, we get
\begin{align*}
|\nabla^2 G(y,B_R)|
\lesssim
\frac{\tx_y^2}{R^2}\osc_{B_{R/2}(y)}G(\cdot,B_R)+\tx_y^2\log R
\stackrel{\eqref{eq:bd-g-y-br}}\lesssim
\tx_y^2\tx^{*d-1}\log R.
\end{align*}

\noindent{\bf Step 3.}
By Proposition~\ref{prop:sensitivity-rho}, Theorem~\ref{thm:hk-bounds}, and \eqref{eq:green-c2-bd},
\begin{align}\label{eq:rho-sensi}
\Abs{\vd{y}\tfrac{\rho_\omega(B_R)}{|B_R|}}
\lesssim
R^{-d}\rho_\omega(y)\abs{\nabla^2 G^{\omega_y'}(y,B_R)}
\lesssim
\ms J_y(\omega,\omega')w(|y|),
\end{align}
where $\ms J_y(\omega,\omega'):=\tx_y^{d-1}(\omega)\tx_y^2(\omega_y')\tx^{*d-1}(\omega_y')$, and
\[
w(r)=\left\{
\begin{array}{lr}
r^{-2}u(r) &\text{ if }r>4R\\
R^{-d}\log R &\text{ if }r\le 4R.
\end{array}
\right.
\]
Note that $\mb E[\ms J_y^n]\le \mb E[\tx^{2dn}]$ for all $y\in\Z^d,n\ge 1$, and
\[
[\sum_{x\in\Z^d}w^2(|x|)]^{1/2}
\asymp R^{-d/2}\log R=:F(R).
\]
Applying  Lemma~\ref{lem:efron-stein} to  $Z(\omega)=\rho_\omega(B_R)/|B_R|$, 
we get
\[
\mb E\left[\exp\big(
C_\error|(Z-1)/F(R)|^{2(0.5-\error)/(2+0.5-\error)}
\big)\right]
\lesssim 
E[\exp(c\tx^{2d(0.5-\error)})]<\infty.
\]
The theorem follows.
\end{proof}

\subsection{Correlation structure of the field of the invariant measure}

In this subsection we will investigate the mixing property of the field by showing the rate of decay of its correlations.
Intuitively, since $\rho_\omega(x)$ is determined by the long term frequency of visits of the RWRE to $x$,  the influence of environments at remote locations will be small.

Our proof uses the following covariance version of Efron-Stein inequality
\begin{equation}\label{eq:cov-ef-st}
\abs{\cov_{\mb P}(F,G)}
\le 
\sum_y \norm{\vd y F}_{L^{2}(\mb P)}\norm{\vd yG}_{L^{2}(\mb P)},
\quad \forall F,G\in L^2(\mb P).
\end{equation}
Such an inequality can be obtained using the same martingale-difference decomposition as  in Efron-Stein's inequality.  See \cite[Lemma 3]{GO-12}, \cite[(4.4)]{Gu-Mourrat-16}. 


\begin{proof}[Proof of Proposition~\ref{prop:correlation}]
{\bf (i)} 
Recall the function $u(r)$ defined in \eqref{eq:def-u}. 
For $y\notin B_r$,  $n\ge 1$, by Proposition~\ref{prop:sensitivity-rho} and Theorems~\ref{thm:hk-bounds}, \ref{thm:Green-bound}, \ref{thm:c11},
\begin{align*}
\norm{\vd y(\rho-\rho_r)}_{L^n(\mb P)}
&=\norm{\vd y \rho}_{L^n(\mb P)}\\
&\lesssim
\norm{\rho(y)|\nabla^2 G^{\omega_y'}(y,0)|}_{L^n(\mb P)}\\
&\lesssim 
\norm{\tfrac{u(|y|)}{|y|^2}(\tx_y\tx(\omega_y'))^{d-1} \tx_y^2(\omega_y')}_{L^n(\mb P)}\\
&\lesssim
\tfrac{u(|y|)}{|y|^2}\norm{\tx^{2d}}_{L^n(\mb P)}.
\end{align*}
Note that for $r\ge 2$,
\[
\sum_{y\notin B_r}|y|^{-4}u(|y|)^2
\asymp
\int_r^\infty
s^{-4}u(s)^2 s^{d-1}\dd s
\lesssim
\left\{
\begin{array}{lr}
r^{-2}(\log r)^2, &d=2\\
r^{-d}, & d\ge 3
\end{array}
\right.
=:F^2.
\]
Then, applying Lemma~\ref{lem:efron-stein} to $Z=\rho(0)-\rho_r(0)$ (as a field over $\Z^d\setminus B_r$),  we obtain
\[
\mb E\left[\exp\big(
C|F^{-1}(Z-\mb EZ)|^{2q/(2+q)}
\big)\right]
\lesssim 
E[\exp(c\tx^{2dq})], \quad \forall q\in(0,1/2).
\]
Statement \eqref{item:prop-cor-1} is proved.

{\bf (ii)}
By \eqref{eq:cov-ef-st} and Proposition~\ref{prop:sensitivity-rho}, we have
\begin{align*}
\abs{\cov(\rho(0),\rho(x))}
\lesssim
\sum_y \norm{\rho(y)\nabla^2 G^{\omega_y'}(y,0)}_{L^{2}(\mb P)}\norm{\rho(y)\nabla^2 G^{\omega_y'}(y,x)}_{L^{2}(\mb P)},
\end{align*}
where $\nabla^2$ is only applied to the first $\Z^d$ coordinate in the argument of $G^{\omega_y'}(\cdot,\cdot)$.  
Let $u$ be as in \eqref{eq:def-u}. 
Using the bounds of $\rho$ in Theorem~\ref{thm:hk-bounds} and applying the $C^{1,1}$ estimates Theorem~\ref{thm:c11} to $G^{\omega_y'}$,  we further get
\begin{align*}
\abs{\cov(\rho(0),\rho(x))}
&\lesssim
\sum_y
\frac{u(|y|)}{(1+|y|)^2}\frac{u(|x-y|)}{(1+|x-y|)^2}\\
&\lesssim
\left\{
\begin{array}{lr}
\frac{\log(2+|x|)}{(1+|x|)^d} & d\ge 3\\
\frac{\log(2+|x|)^3}{(1+|x|)^2} &d=2.
\end{array}
\right.
\end{align*}
The verification of the last inequality is similar to \cite[Lemma 9.1]{Gu-Mourrat-16}.
\end{proof}

\section{Quantification of the diffusive behavior}
In this section, $\psi$ is always assumed to be a local function.

\subsection{Estimates of the approximate corrector}\label{subsec:approx_cor}
We consider the function $\apk:\Z^d\to\R$ defined as
\begin{equation}\label{eq:def-phihat}
\apk(x)=\apk_\omega(x;\psi,R)=
-\int_0^\infty e^{-t/R^2}E_\omega^x\big[\psi(\theta_{Y_t}\omega)-E_\Q[\psi]\big]\dd t.
\end{equation}
where $R\ge 1$, and $\psi$ is measurable function of $\omega(0)$.
Notice that $\apk$ is stationary, i.e., $\apk_\omega(x)=\apk_{\theta_x\omega}(0)$. Moreover, 
 $\apk$ is a solution of 
\begin{equation}\label{eq:eq_of_phihat}
L_\omega\apk(x)=\tfrac{1}{R^2}\apk(x)+\psi(\theta_x\omega)-E_\Q[\psi],
\quad 
x\in\Z^d.
\end{equation}
Without loss of generality, assume $E_\Q[\psi]=0$. 
Clearly,  by the definition of $\apk$ in \eqref{eq:def-phihat}, for any $\omega\in\Omega$,
\begin{equation}
\label{eq:hatphi-bound}
\sup_{x\in\Z^d}|\apk(x)|\le R^2\norm{\psi}_\infty.
\end{equation}
and so $\norm{\tfrac{1}{R^2}\apk(x)+\psi(\theta_x\omega)}_\infty\le 2\norm{\psi}_\infty$. 
By \eqref{eq:hatphi-bound} and the H\"older estimate \eqref{eq:osc}, 
\[
[\apk]_{\gamma;B_{R/2}}\lesssim
R^{-\gamma}[\max_{B_R}|\apk|+R^2\norm{R^{-2}\apk+\psi}_{d;B_R}]\lesssim R^{2-\gamma}\norm{\psi}_\infty.
\]
Hence,  for any $2\le D\le R$, applying \eqref{eq:c2} to $f=\apk/R^2$ and $\sigma=\gamma$ in $B_D$, we get
\begin{align}
|\nabla\apk(0)|
&\lesssim 
\tx(D\norm{\psi}_\infty+\tfrac{1}{D}\norm{\apk}_{1;B_D}), \label{eq:phihat_c01}
\\
\abs{\nabla^2\apk(0)}
&\lesssim 
\tx^{2}\norm{\psi}_\infty. \label{eq:phihat_c11}
\end{align}

The goal of this subsection is to establish the optimal rate of convergence of the approximate corrector.  Recall the function $\mu(R)$ defined in \eqref{eq:def-mu}.
\begin{lemma}\label{lem:approx-cor-concentration}
Assume that $\psi(\omega)=\psi(\omega(0))$ is a bounded function of  $\omega(0)$.
For any $0<p<\tfrac{2d}{3d+2}$, there exists $C=C(d,\kappa,p)$ such that for $t\ge 0$, $R\ge 2$,
\[
\mb P\left(
\abs{\apk(0)}\ge 
t\mu(R)\norm{\psi}_\infty
\right)\le C\exp(-\tfrac{1}{C}t^p).
\]
\end{lemma}

The continuous version of Lemma~\ref{lem:approx-cor-concentration} was proved earlier by Armstrong, Lin \cite{AL-17}.
Our result in two dimensions ($d=2$) is slightly better than that in \cite{AL-17}.

\begin{proof}
[Proof of Lemma~\ref{lem:approx-cor-concentration}]
We now obtain Lemma~\ref{lem:approx-cor-concentration} using the concentration inequality \eqref{eq:bblm}. 
We will need a bound for $\vd{y}\apk(0)$. 
Recall the vertical derivative $\vd y$ and the notation $\omega'_y$ as in \eqref{eq:def_vert_der}. 
By \eqref{eq:eq_of_phihat} and formula \eqref{eq:vdlaplace}, 
$\vd{y}\apk(x)$ satisfies, for $x,y\in\Z^d$,
\begin{align}\label{eq:L-apg}
L_{\omega}(\vd{y}\apk)(x)
=
R^{-2}\vd{y}\apk(x)+\left[\vd y\psi(\theta_y\omega)-\tfrac12\tr\big(\vd y a(y)\nabla^2\apk_{\omega'_y}(y)\big)\right]\mathbbm{1}_{x=y}.
\end{align}	
Denoting the Green function associated to the operator in  \eqref{eq:eq_of_phihat} by
\begin{equation}\label{eq:def-apg}
\apg_\omega(x,y)=\int_0^\infty
e^{-t/R^2}p_t^{\omega}(x,y)\dd t, 
\end{equation}
equality \eqref{eq:L-apg} yields
\begin{equation}\label{eq:w-appr-gf}
\vd{y}\apk(x)=
-\left[\vd y\psi(\theta_y\omega)-\tfrac12\tr\big(\vd y a(y)\nabla^2\apk_{\omega'_y}(y)\big)\right]
\apg_\omega(x,y).
\end{equation}
Note that by Theorem~\ref{thm:hk-bounds}\eqref{item:hke-expmmt}, 
\begin{align}\label{eq:verti-der-phi}
\apg_\omega(0,y)
&\lesssim
\tx_y^{d-1}\int_0^\infty(1+t)^{-d/2}\exp\left[-\tfrac{t}{R^2}-c\mf h(|y|,t)\right]\dd t\nn\\
&\lesssim
\tx_y^{d-1}e^{-c|y|/R}\nu(|y|),
\end{align}
where  $\tx_y(\omega)=\tx(\theta_y\omega)$ and 
\begin{equation}\label{eq:nu}
\nu(r)=\nu_R(r)=
\left\{
\begin{array}{lr}
1+\log(\tfrac{R}{(r+1)\wedge R}) & d=2\\
(r+1)^{2-d} &d\ge 3.
\end{array}
\right.
\end{equation}
Thus,  with $\tx_y':=\tx(\theta_y\omega'_y)$, 
by \eqref{eq:w-appr-gf},\eqref{eq:verti-der-phi}, and \eqref{eq:phihat_c11}, we get
\begin{align}\label{eq:vdbound}
|\vd{y}\apk(0)|
\lesssim
{{\tx}_y'}^{2}\tx_y^{d-1}\norm{\psi}_\infty e^{-c|y|/R}
\nu(|y|).
\end{align}	
Notice that $\nu_R$ depends on $R$ only in $d=2$. 
Recall $\mu(R)$ in \eqref{eq:def-mu}. Notice that 
\begin{equation}
\label{eq:nu2-sum}
\sum_{y\in\Z^d}e^{-2c|y|/R}\nu(|y|)^2
\lesssim
\int_0^\infty 
e^{-2cr/R}\nu(r)^2 r^{d-1}\dd r
\lesssim
\mu(R)^2.
\end{equation}
The verification of inequalities \eqref{eq:verti-der-phi} and \eqref{eq:nu2-sum} are included in the Appendix.

Since $\mb E[({{\tx}_y'}^{2}\tx_y^{d-1})^n]\le \mb E[\tx^{(d+1)n}]$ for all $y\in\Z^d, n\ge 1$,
applying Lemma~\ref{lem:efron-stein} to
$Z(\omega):=\frac{\apk(0)}{\norm{\psi}_\infty \mu(R)}$,
with $p=\frac{2(d-\error)/(d+1)}{2+(d-\error)/(d+1)}$,
we get
\[
\mb E[\exp(c|Z-\mb EZ|^p)]
\lesssim
\mb E[
\exp(c|\tx^{d+1}|^{(d-\error)/(d+1)})
]
<C.
\]
In particular, $\mb E[|Z-\mb EZ|^2]<C$.

To prove Lemma~\ref{lem:approx-cor-concentration}, by Chebyshev's inequality,  we only need to show that
\begin{equation}\label{item:old-rho-integ}
\mb E[\exp(c|Z|^p)]<C.
\end{equation}
It suffices to show that $|\mb EZ|<C$.
Since $\mb Q$ is an invariant measure for $(\theta_{Y_t}\omega)_{t\ge 0}$, we have $E_{\Q}E_\omega^0[\psi(\theta_{Y_t}\omega)]=E_\Q[\psi]=0$ for all $t\ge 0$. Hence, by \eqref{eq:def-phihat}, we know
\[
E_\Q[\apk(0)]=0
\]
and so $E_\Q[Z]=0$. Further, by H\"older's inequality and Theorem~\ref{thm:hk-bounds},
\begin{align*}
|\mb EZ|=
\abs{E_\Q[Z-\mb EZ]}\le 
E_\Q[|Z-\mb EZ|]
\le 
\norm{\rho}_{L^2(\mb P)} \norm{Z-\mb EZ}_{L^2(\mb P)}
\le C
\end{align*}	
Therefore, we obtain \eqref{item:old-rho-integ}. Lemma~\ref{lem:approx-cor-concentration} follows.
\end{proof}

As a consequence of 
 Lemma~\ref{lem:approx-cor-concentration} and Theorem~\ref{thm:c11},  we have the following $C^{0,1}$ estimate of the approximate corrector $\apk$.
\begin{corollary}
\label{cor:c01-apk}
Let $\psi, R\ge 2, \apk$ be as in Lemma~\ref{lem:approx-cor-concentration}.  For any $0<s<\tfrac{2d}{3d+4}$, there exists a random variable $\ms Y=\ms Y(R,s,\omega)>0$ with $\mb E[\exp(\ms Y^s)]<\infty$ such that, for $\mb P$-a.e. $\omega$ and $x\in\Z^d$,
\[
|\nabla\apk_R(x)|\le \ms Y(\theta_x\omega)\sqrt{\mu(R)}.
\]
\end{corollary}
\begin{proof}[Proof of Corollary~\ref{cor:c01-apk}]
For any $2\le D\le R$ and $p\in(0,d)$,  by \eqref{eq:phihat_c01} and Lemma~\ref{lem:approx-cor-concentration}, there exists a random variable $\ms Y^*(R,p,\omega)$ with $\mb E[\exp(\ms Y^{*p})]<C$ such that 
\[
|\nabla\apk_R(0)|
\lesssim 
\tx\big(D\norm{\psi}_\infty+\tfrac{1}{D}\ms Y^{*(3d+2)/2}\mu(R)
\big).
\]
Putting $D=\sqrt{\mu(R)}$,  we obtain the corollary.
\end{proof}

\subsection{Quantification of the ergodicity of the environmental process: Proof of Theorem~\ref{thm:quant-ergo}}

In this section we will derive the optimal rates of convergence (as $t\to\infty$) of the ergodic average $\tfrac 1t  E_\omega[\int_0^t\psi(\evp{s})\dd s]$, where $\evp s$ denotes the process of the environment viewed from the particle:
\[
\evp{s}:=\theta_{Y_s}\omega.
\]

With Lemma~\ref{lem:approx-cor-concentration}, it may be tempting to compare the approximate corrector $\apk$ in \eqref{eq:def-phihat} to the corrector within a finite ball $B_R$, i.e., the solution $u$ to the Dirichlet problem $L_\omega u=\psi_\omega$ in $B_R$ with $u=0$ on $\partial B_R$.   However,  such comparison involves controlling the boundary error $\max_{\partial B_R}\apk$ which would result in an extra $\log R$ factor.  In what follows, 
we will follow the argument of Kipnis and Varadhan \cite{KV-86} to approximate  $E_\omega[\int_0^T \psi(\theta_{Y_s}\omega)\dd s]$ with a martingale using the approximate corrector.

\begin{proof}[Proof of Theorem~\ref{thm:quant-ergo}]
Without loss of generality, assume $\norm{\psi}_\infty=1$ and $\bar\psi=0$.

 First, we will construct a martingale (for both continuous and discrete time cases) using the approximate corrector. 

For any fixed $T>1$, let $\phi:\Omega\to\R$ denote the function
\[
\phi(\omega)=\phi_{\psi,T}(\omega):=\apk(0;\psi,\sqrt T,\omega),
\]
where $\apk$ is as in \eqref{eq:def-phihat}. Then, for a.s. $\omega\in\Omega$, the process $(M_t)_{t\ge 0}$ defined by
\begin{align}
\label{eq:mart-constr}
M_t:&=\phi(\theta_{Y_t}\omega)-\phi(\theta_{Y_0}\omega)-\int_0^t L_\omega\phi(\theta_{Y_s}\omega)\dd s\nn\\
&\stackrel{\eqref{eq:eq_of_phihat}}
=\phi(\evp{t})-\phi(\evp 0)-\int_0^t [\tfrac 1T\phi(\evp s)+\psi(\evp s)]\dd s
\end{align}
is a $P_\omega$-martingale with respect to the filtration $\ms F_t=\sigma(Y_s: s\le t)$. 
Similarly, for discrete-time RWRE,  we have that
\[
N_n:=\phi(\evp n)-\phi(\evp 0)-\sum_{i=0}^{n-1}[\tfrac 1T\phi(\evp i)+\psi(\evp i)]
\]
is a $P_\omega$-martingale with respect to the filtration $\ms F_n=\sigma(X_i: i\le n)$. 

Next, we will derive an exponential moment bounds for $\int_0^t P_s\psi\dd s$ and $\sum_{i=0}^{n}P_i\psi$,
where the operator $P_s$ is as in \eqref{eq:def-semigroup}.  We will only provide a proof for the continuous-time case,  because the argument for the discrete-time setting is exactly the same. 
 Since $E_\omega[M_s]=E_\omega[M_0]=0$,  taking expectations in \eqref{eq:mart-constr}, we get
\begin{align}\label{eq:221228-1}
\int_0^t P_s\psi\dd s=
P_t\phi-\phi
-\tfrac 1T\int_0^t P_s\phi \dd s.
\end{align}
Since the process $(\evp s)$ is a stationary sequence under the measure $\Q\times P_\omega$, we have,  by Jensen's inequality, for any $t\ge 0$,  $q\ge 1$, 
\begin{align*}
\norm{P_t\phi}_{L^q(\Q)}^q=E_{\Q}[|E_\omega\phi(\evp t)|^q]\le 
E_{\Q\times P_\omega}[|\phi(\evp t)|^q]=E_{\Q}[|\phi|^q].
\end{align*}
Hence, taking the $L^q(\Q)$-norms on both sides of \eqref{eq:221228-1},  we get
\[
\norm{\int_0^T P_s\psi\dd s}_{L^q(\Q)}\le 3\norm{\phi}_{L^q(\Q)}, \quad \forall q\ge 1
\]
which implies
\begin{align*}
&E_{\Q}\left[\exp\bigg(c\Abs{
\int_0^T P_s\psi\dd s
\big/
\mu(\sqrt T)
}^p\big)\right]\\
&\le 
E_{\Q}\left[\exp\bigg(c\Abs{3\phi/\mu(\sqrt{T})}^p\bigg)\right]\\
&\le 
\norm{\rho}_{L^2(\mb P)}E_{\mb P}\left[\exp\bigg(0.5c\Abs{3\phi/\mu(\sqrt{T})}^p\bigg)\right]^{1/2}
\stackrel{\eqref{item:old-rho-integ}}\le C,
\end{align*}
where we used H\"older's inequality in the second inequality.  

Note that $\nu(T)=T^{-1}\mu(\sqrt T)$ as defined in \eqref{eq:mu2}. 
 The theorem follows from the above moment bound and Chebyshev's inequality.
\end{proof}

\subsection{A Berry-Esseen estimate for the QCLT: Proof of Corollary~\ref{cor:quant-clt}}

To prove Corollary~\ref{cor:quant-clt} we will apply the Berry-Esseen estimates for martingales by Heyde and Brown \cite{HB-70}. Here we will use the version in \cite[Theorem 2]{EH-88} which is also applicable to the continuous-time setting.
\begin{proof}[Proof of Corollary~\ref{cor:quant-clt}]
For any unit vector $\ell\in\R^d$,  let $\psi_0(\omega)=\ell^T\tfrac{\omega(0)}{\tr\omega(0)}\ell$,  $\psi=\psi_0-E_{\Q}[\psi_0]$.  Following the notations in \cite{EH-88},  we set
\begin{align*}
N_{n,2}:&=
\qe\left[
\Abs{
\sum_{k=0}^{n-1}E_\omega\big[\tfrac1{\sqrt n}\big((X_{k+1}-X_k)\cdot\ell\big)^2
|\ms F_k\big]
-\ell^T\bar a\ell
}^{2}
\right]\\
&=
\frac{1}{n^2}\qe\left[
\big(
\sum_{k=0}^{n-1}\psi(\evp k)
\big)^{2}
\right],
\end{align*}
\[
L_{n,2}:=
\sum_{k=0}^{n-1}E_\omega[
\abs{\tfrac{1}{\sqrt n}
(X_{k+1}-X_k)\cdot\ell
}^4
]
=
\frac{1}{n^2}
\qe
\left[
\sum_{k=0}^{n-1}\psi_0(\evp k)
\right].
\]
The term $N_{n,2}$ can be further written as
\begin{align*}
n^2 N_{n,2}
= 
2\sum_{i=0}^{n-1}E_\omega\left[\psi(\bar\omega^i)\sum_{j=0}^{n-i-1}\psi(\bar\omega^{i+j})\right]
=
2\sum_{i=0}^{n-1}E_\omega\left[\psi(\bar\omega^i)\qe^{X_i}\left[\sum_{j=0}^{n-i-1}\psi(\bar\omega^{j})\right]\right].
\end{align*}
Hence, for any $q\ge 1$, using the fact that $(\evp i)$ is a stationary sequence under $\Q\times P_\omega$, we get (note $\norm{\psi_0}_\infty\lesssim 1$)
\begin{align*}
\norm{N_{n,2}}_{L^q(\Q)}
\lesssim
\frac{1}{n^2}\sum_{i=0}^{n-1}
\norm{
\sum_{j=0}^{n-i-1}P_j\psi
}_{L^q(\Q\times P_\omega)}
\end{align*}
which, by Jensen's inequality and the fact $\tfrac1{n^2}\sum_{k=1}^n\mu(\sqrt k)\asymp\nu(n)$, implies that for any $0<p<\tfrac{2d}{3d+2}$,
\begin{align*}
&E_{\Q}
\left[
\exp\big(
c\abs{
N_{n,2}/\nu(n)}^p
\big)
\right]
\\&
\lesssim
\frac{1}{n^2\nu(n)}
\sum_{k=1}^{n}\mu(\sqrt k)
E_\Q\left[
\exp\big(c\Abs{
\sum_{j=0}^{k-1}P_j\psi
\big/\mu(\sqrt k)}^p
\big)
\right]
\stackrel{Theorem~\ref{thm:quant-ergo}}{\le}C.
\end{align*}
Thus,  using the moment bound of $\rho^{-1}$ in Theorem~\ref{thm:hk-bounds}, by H\"older's inequality,
\[
E_{\mb P}
\left[
\exp\big(
0.5c\abs{
N_{n,2}/\nu(n)}^p
\big)
\right]
\le 
\norm{\rho^{-1/2}}_{L^2(\mb P)}
E_{\Q}
\left[
\exp\big(
c\abs{
N_{n,2}/\nu(n)}^p
\big)
\right]^{1/2}
\le C.
\]
By Theorem~\ref{thm:quant-ergo} we already know that $E_{\mb P}[\exp\big(c|nL_{n,2}|^p\big)]\le C$. Therefore, we conclude that there exists a random variable $\ms Y^5$ with $E_{\mb P}[\exp({\ms Y}^{5p})]<\infty$ such that 
\[
L_{n,2}+N_{n,2}\le C\nu(n)\ms Y^5.
\]
The corollary follows by applying \cite[Theorem 2]{EH-88}.
\end{proof}

\section{Homogenization of the Dirichlet problem}\label{sec:quant_hom}
In this section we will investigate the rate of convergence of the solution of the Dirichlet problem \eqref{eq:elliptic-dirich}.  With $a$ as defined in \eqref{eq:def-a},  problem \eqref{eq:elliptic-dirich} is equivalent to
\[
\left\{
\begin{array}{lr}
L_\omega u(x)=\tfrac 12\tr\big(a(x)\nabla^2u(x)\big)=\frac{1}{R^2}f\big(\tfrac{x}{R}\big)\tfrac{\zeta(\theta_x\omega)}{\tr\omega(x)} & x\in B_R,\\[5 pt]
u(x)=g\big(\tfrac{x}{|x|}\big) & x\in \partial B_R,
\end{array}
\right.
\]

Throughout this section $\psi:\Omega\to\R$ always denotes an $L^\infty(\mb P)$ bounded measurable function of $\omega(0)$.  With abuse of notation write $\psi(x)=\psi_\omega(x):=\psi(\theta_x\omega)$.

\begin{definition}\label{def:loc-global}
A function $\phi=\phi_\omega:\Z^d\to\R$ is called a {\it local corrector} of the Dirichlet problem \eqref{eq:elliptic-dirich} (associated to $\psi$) if it satisfies 
\begin{equation}\label{local-corrector-eq}
L_\omega\phi(x)=\psi(x)-\mb E_\Q[\psi]
\quad\text{ for }x\in B_R.
\end{equation}  
A function $\phi=\phi_\omega:\Z^d\to\R$ is called a (global) {\it corrector} of the operator $L_\omega$ (associated to $\psi$) if it satisfies 
\begin{equation}\label{global-corrector-eq}
L_\omega\phi(x)=\psi(x)-\mb E_\Q[\psi]
\quad\text{ for all }x\in\Z^d.
\end{equation}  
\end{definition}

Recall that by \cite[Corollary~7]{GT-22}, when $d\ge 5$, the (global) corrector $\phi_\omega(x)$ exists. Moreover, the corrector $\phi_\omega$ for $d\ge 5$ is {\it stationary} in the sense that 
$\phi_\omega(0)$ and $\phi_\omega(x)$ have the same distribution under $\mb P$ for all $x\in\Z^d$.
\subsection{The two-scale expansion}
Using the classical method two-scale expansion, we will compare the solutions $u$, $\bar u$ of the heterogeneous equation \eqref{eq:elliptic-dirich} and the effective equation \eqref{eq:effective-ellip}. 

\begin{lemma}
\label{lem:two-scale-exp}
Recall $a$ and $\bar a$ as in \eqref{eq:def-a}, \eqref{eq:def-abar}. 
Let $\omega\in\Omega$, $R>0$.  Let $u$ and $\bar u$ be the solutions of \eqref{eq:elliptic-dirich}, \eqref{eq:effective-ellip}, respectively.  
Let $v^k,\xi:\Z^d\to\R$ be functions that satisfy
\begin{align}\label{eq:def-vk}
L_\omega v^k(x)&=\tfrac12(a_k(x)-\bar a_k), \quad \text{ for } x\in B_R, k=1,\ldots,d\\
L_\omega\xi(x)&=\tfrac{\zeta(\theta_x\omega)}{\tr\omega(x)}-E_\Q[\tfrac{\zeta(\theta_x\omega)}{\tr\omega(x)}] \quad \text{ for } x\in B_R.\nn
\end{align}
That is,  $v^k,\xi$ are local correctors associated to $\tfrac12a_k(x)=\omega(x,x+e_k)$ and $\tfrac{\zeta(\theta_x\omega)}{\tr\omega(x)}$, respectively.  
Then 
\begin{align*}
&\max_{x\in B_R}|u(x)-\bar{u}(\tfrac xR)|\\
&\lesssim
\tfrac 1R\norm{\bar u}_{C^4(\bar{\B}_1)}[\sum_{k=1}^d(
\norm{\nabla v^k}_{d;B_R}+
\tfrac{1}{R}\osc_{\bar B_R}v^k)+\norm{\nabla \xi}_{d;B_R}+\tfrac{1}{R}\osc_{\bar B_R}\xi+\tfrac1R].
\end{align*}
\end{lemma}

The proof is similar to that in the periodic setting (see \cite{GTY-19, ST-21, GST-22} for example).
The only differences are that we need the ``environment corrector" $\xi$,  and we use the local corrector $v^k$ here instead of the global corrector (whose existence is guaranteed in the periodic setting).

\begin{proof}
We can replace the function $g$ in  \eqref{eq:elliptic-dirich} by $\bar u$,  because doing this only introduces an error of size $CR^{-1}\norm{\bar u}_{C^4(\bar{\B}_1)}$ to $|u(x)-\bar u(\tfrac{x}{R})|$.

Without loss of generality, assume $v^k(0)=0$ for all $k=1,\ldots,d$. Let 
\[
\psi_\omega(x):=\tfrac{\zeta(\theta_x\omega)}{\tr\omega(x)}, \quad x\in\Z^d.
\]
Consider the function
\begin{equation}\label{eq:def-w}
w(x)=u(x)-\bar u(\tfrac{x}{R})
+\tfrac{1}{R^2}v^k(x)\partial_{kk}\bar u(\tfrac{x}{R})
-\tfrac{1}{R^2}f(\tfrac{x}R)\xi(x),
\quad x\in\bar B_R,
\end{equation}
where we use the convention of summation over repeated indices. 
Note that 
\begin{align*}
L_\omega u(x)
&=\tfrac1{R^2}f(\tfrac xR)(\psi(x)-\bar\psi)+\tfrac1{R^2}f(\tfrac xR)\bar\psi\\
&=\tfrac1{R^2}\big[f(\tfrac xR)L_\omega\xi(x)+
\tfrac1{2}\bar{a}_k\partial_{kk}\bar u(\tfrac{x}{R})\big]
\end{align*}
 and $\abs{L_\omega[\bar u(\tfrac{x}{R})]-\tfrac1{2R^2}\tr[a(x) D^2\bar{u}(\tfrac xR)]}\lesssim\tfrac1{R^4}\norm{\bar u}_{C^4}$.
Then, applying formula
\[
L_\omega(uv)=uL_\omega v+vL_\omega u+\sum_{y:y\sim x}\omega(x,y)[u(y)-u(x)][v(y)-v(x)]
\]
to the last two terms of \eqref{eq:def-w}, we get,  for any $x\in B_R$,
\begin{align*}
&\abs{L_\omega w(x)}\\
&=
\tfrac1{R^2}\Abs{
f(\tfrac xR)L_\omega\xi(x)+\tfrac1{2}\bar{a}_k\partial_{kk}\bar u(\tfrac{x}{R})-R^2L_\omega[\bar{u}(\tfrac xR)]
+\tfrac1{2}(a_k-\bar a_k)\partial_{kk}\bar u(\tfrac{x}{R})
\nn\\&
+v^kL_\omega[\partial_{kk}\bar u(\tfrac{x}{R})]
+\sum_{y\sim x}\omega(x,y)[\partial_{kk}\bar{u}(\tfrac{y}{R})-\partial_{kk}\bar u(\tfrac{x}{R})][v^k(y)-v^k(x)]
\nn\\
&-f(\tfrac{x}R)L_\omega\xi(x)-L_\omega[f(\tfrac{x}R)]\xi(x)
-\sum_{y\sim x}\omega(x,y)[f(\tfrac{y}{R})-f(\tfrac{x}{R})][\xi(y)-\xi(x)]
}
\nn\\
&\lesssim \tfrac1{R^2}\Abs{
v^kL_\omega[\partial_{kk}\bar u(\tfrac{x}{R})]
+\sum_{y\sim x}\omega(x,y)[\partial_{kk}\bar{u}(\tfrac{y}{R})-\partial_{kk}\bar u(\tfrac{x}{R})][v^k(y)-v^k(x)]\nn\\
&-L_\omega[f(\tfrac{x}R)]\xi(x)-\sum_{y\sim x}\omega(x,y)[f(\tfrac{y}{R})-f(\tfrac{x}{R})][\xi(y)-\xi(x)]
}+\tfrac1{R^4}\norm{\bar u}_{C^4}
\nn\\
&\lesssim
\tfrac1{R^3}\norm{\bar u}_{C^4(\bar\B_1)}\sum_{k=1}^d\big(
\tfrac1{R}|v^k(x)|+\tfrac1{R}|\xi(x)|+|\nabla v^k(x)|+|\nabla \xi(x)|+\tfrac1R
\big).
\end{align*}	
Thus, by the above inequality, \eqref{eq:def-w}, and the ABP maximum principle,
\begin{align*}
\max_{B_R}|w|
&\lesssim
R^2\norm{L_\omega w}_{d;B_R}+\tfrac1{R^2}\max_{x\in\partial B_R}|v^k(x)\partial_{kk}\bar u(\tfrac{x}{R})-f(\tfrac{x}R)\xi(x)|\\
&\lesssim
\tfrac1{R^2}\norm{\bar u}_{C^4(\bar\B_1)}\sum_{k=1}^d
\big(
\norm{|v^k|+|\xi|}_{d;B_R}+R\norm{|\nabla v^k|+|\nabla\xi|}_{d;B_R}\\&+1+\max_{\partial B_R}(|v^k|+|\xi|)
\big).
\end{align*}
The lemma follows from the above inequality and \eqref{eq:def-w}.
\end{proof}

\begin{remark}
 By Lemma~\ref{lem:two-scale-exp}, to control the homogenization error, it suffices to control the size of local correctors and their discrete gradients.

When $d\ge 5$,  by  \cite[Corollary~7]{GT-22}, the stationary corrector $\phi$ exists.
Hence it is not surprising that the generically optimal rate $R^{-1}$ of homogenization can be achieved  for $d\ge 5$.

Unfortunately,  for $d<5$,  there is no stationary corrector at our disposition.
Although (as demonstrated in the first arXiv version \cite{GT-23-1} of this paper) one may use the approximate correctors $\apk$  together with the two-scale expansion argument to quantify the homogenization,  the $C^{0,1}$ estimate in Theorem~\ref{thm:c11} only yields a bound for $\nabla\apk$ with size $\sqrt{\mu(R)}$, cf. Corollary~\ref{cor:c01-apk}, which is not good enough for us to obtain optimal rates. 
 
In this regard,  
it is not the size of the corrector,  but its $C^{0,1}$ regularity that is posing the biggest challenge in the course of obtaining the optimal  homogenization rates.
 To resolve this issue,  our strategy is to construct a local corrector and use sensitivity estimates as in \eqref{eq:w-appr-gf} to obtain optimal bounds for its gradients.  Firstly, note that such an argument does not work for $\apk$.  Indeed,  bounding $\vd y\nabla\apk$ (cf.  \eqref{eq:w-appr-gf}) involves estimating $\nabla\apg_\omega(\cdot,y)$. However, $\apg_\omega(\cdot,y)$, which solves
 \[
L_\omega \apg(\cdot,y)=\tfrac{1}{R^2}\apg(\cdot,y)-\mathbbm1_y,
\]
is nowhere $\omega$-harmonic, and so our $C^{0,1}$ theory (Theorem~\ref{thm:c11}) is not applicable to obtain a desired bound for $\nabla\apg_\omega(\cdot,y)$.
Secondly,  it might be tempting to construct a local corrector by solving a Dirichlet problem in a finite region.  For instance, one may solve \eqref{local-corrector-eq} by imposing a zero boundary condition on $\partial B_R$.  However, an obvious defect for such a construction is that 
 $C^{0,1}$ and $C^{1,1}$ bounds (using Theorem~\ref{thm:c11}) near the boundary blows up, which makes an optimal estimate impossible.  Intuitively, this is due to the fact that the RWRE which starts near the boundary cannot survive long enough (before hitting the boundary) to ``feel" the homogenization.


To overcome this challenge, we will construct a local corrector such that (1) the corresponding Green function is $\omega$-harmonic in $B_{2R}$; (2) the corresponding RWRE survives a time of scale $\approx R^2$. 
\end{remark}

\subsection{The construction of a local corrector}
Recall the continuous-time random walk $(Y_t)_{t\ge 0}$ in Definition~\ref{def:rwre-continuous}. 
Let $\psi(\omega)=\psi(\omega(0))$ be a function of $\omega(0)$. 
Write $\psi_\omega(x)=\psi(\theta_x\omega)$. 

\begin{definition}\label{def:new-corrector}
For $R>1,\omega\in\Omega$, and any function $\eta:\Z^d\to[0,1]$ with the property $\eta(x)=1$ for all $ x\notin{B_{3R}}$, let the function  $\lock=\lock_{\omega,R}(x;\eta,\psi):\Z^d\to\R$ be the solution of
\begin{equation}\label{eq:def-lock}
\left\{
\begin{array}{lr}
L_\omega\lock=\tfrac{1}{R^2}\lock\eta+\psi_\omega-E_\Q[\psi] &\text{in }\Z^d,\\
\lock \text{ is bounded; that is, }|\lock|\le M=M(\omega,S,d,\kappa)<\infty. &
\end{array}
\right.
\end{equation}
With abuse of notation, if $R$ (or $\omega$) is fixed,  we may  simply write $\lock_{\omega,R}$ as $\lock_\omega$ (or $\lock_R$)
whenever confusion does not occur.
\end{definition}

Note that when $\eta|_{B_R}=0$,  the function $\lock_\omega(x;\eta,\psi)$ becomes a local corrector,  according to Definition~\ref{def:loc-global}. 
When $\eta\equiv 1$,  then $\lock_\omega(x;\eta,\psi)=\apk(x)$ is the approximate corrector. 

The existence and uniqueness of $\lock$ will be established in Proposition~\ref{prop:exist-unique-lock}.

\begin{proposition}
\label{prop:exist-unique-lock}
For any $R>1, \omega\in\Omega$, any bounded functions $f:\Z^d\to\R$ and $\eta:\Z^d\to[0,1]$ with $\eta|_{\Z^d\setminus B_{3R}}=1$,   there exists a unique solution $u$ to the problem
\begin{equation}\label{eq:local-func}
\left\{
\begin{array}{lr}
L_\omega u=\tfrac{1}{R^2} u\eta+f &\text{in }\Z^d,\\
u\text{ is bounded; that is, }|u|\le M=M(\omega,S,d,\kappa)<\infty. &
\end{array}
\right.
\end{equation}
\end{proposition}

Our proof of the existence is constructive. 

Let us explain the probabilistic intuition behind the above definitions.  Set 
\[
\tilde \eta(x)=\frac{\eta(x)}{R^2+\eta(x)}\in[0,\tfrac{1}{R+1}], \quad x\in\Z^d.
\]
Imagine at every site $x\in\Z^d$ there is a clock which rings every geometric (with parameter $\tilde \eta(x)$) units of time,  and let $T$ be the first time that the discrete-time  random walk $(X_n)$ hears the clock (at its current location) rings.  
To be rigorous,  let $\{\bern_n(x):n\in\N, x\in\Z^d\}$ be a family of independent Bernoulli random variables with $P(\bern_n(x)=1)=1-P(\bern_n(x)=0)=\tilde \eta(x)$.  Define the stopping time $T$ as
\begin{equation}\label{eq:def-stoppingtime}
T=T(\eta)=\inf\{n\ge 0: \bern_n(X_n)=1\}.
\end{equation}
Clearly $E_\omega^x[T]<\infty$, $\forall x\in\Z^d$. 
One can check that problem \eqref{eq:local-func} has a solution 
\begin{equation}\label{eq:phi-expression}
u(x):=-E_\omega^x\left[\sum_{n=0}^T \big(1-\tilde \eta(X_n)\big)f(X_n)\right].
\end{equation}

For any $y\in\Z^d$, define the Green function corresponding to problem \eqref{eq:local-func} by 
\[
\locg_\omega(x,y)=\locg_{\omega}(x,y;\eta,R)=E_\omega^x\left[\sum_{n=0}^T \big(1-\tilde \eta(y)\big)\mathbbm1_{X_n=y}\right].
\]
Note that $(1-\tilde \eta)\asymp 1$, and so
$\locg_{\omega}(x,y)$ is roughly the expected amount of time that the discrete-time random walk (starting from $x$) spent at $y$ before the clock $T$ rings.  By formula \eqref{eq:phi-expression}, 
$\locg_{\omega;y}(\cdot)=\locg_{\omega;y}(\cdot;\eta,R):=\locg_\omega(\cdot,y;\eta,R)$ solves the equation 
\begin{equation}\label{eq:LG}
L_\omega \locg_{\omega;y}=\tfrac1{R^2} \locg_{\omega;y}\eta-\mathbbm{1}_{y} \quad\text{ in }\Z^d.
\end{equation}

Then,  the solution \eqref{eq:phi-expression} of problem \eqref{eq:local-func}  can be written as
\begin{equation}\label{eq:phi-in-green}
u(x)=-\sum_y \locg_\omega(x,y)f(y).
\end{equation}

\begin{proof}
[Proof of Proposition~\ref{prop:exist-unique-lock}]
A solution to \eqref{eq:local-func} is constructed as in \eqref{eq:phi-expression}. 

To prove the uniqueness, it suffices to show that $u\equiv 0$ is the only solution of 
\[
\left\{
\begin{array}{lr}
L_\omega u=\tfrac{1}{R^2} u\eta&\text{in }\Z^d,\\
u\text{ is bounded. }&
\end{array}
\right.
\]
We will prove this by contradiction.

Assume there is a solution $u$ of the above problem with $u\not\equiv 0$, say, $\sup_{\Z^d}u>0$.

 First,  letting $S=\{x:\eta(x)=0\}\subset B_{3R}$,  then $u$ is $\omega$-harmonic on $S$. 
By the ABP maximum principle,  $\max_S u\le \max_{\partial S}u$ and so
there exists $x_0\notin S$ with $u(x_0)>0$ and $u(x_0)\ge\max_S u$. Since $L_\omega u(x_0)=E^{x_0}_\omega[u(X_1)-u(x_0)]>0$, there exists a neighboring point $x_1\notin S$ of $x_0$ with $u(x_1)=\max_{|y-x_0|=1}u(y)>u(x_0)$.  Repeating this argument, we obtain an infinite sequence of points 
 $\{x_n:n\in\N\}\subset\Z^d\setminus S$ with $u(x_n)=\max_{|y-x_{n-1}|=1}u(y)>u(x_{n-1})>0$ for all $n\ge 1$. 

Next,  whenever $x_n\notin B_{3R}$, since 
\[
\tfrac1{R^2}u(x_n)=L_\omega u(x_n)=E^{x_n}_\omega[u(X_1)-u(x_n)]\le u(x_{n+1})-u(x_n)
\]
which implies $u(x_{n+1})\ge (1+\tfrac1{R^2})u(x_n)$, 
we conclude that  $u(x_m)\ge (1+\tfrac{1}{R^2})^{m-n} u(x_n)>0$ for all $m\ge n$, which contradicts the property that $u$ is a bounded function. 

Therefore, $u\equiv 0$ and our proof is complete.
\end{proof}

By Proposition~\ref{prop:exist-unique-lock}, \eqref{eq:phi-expression} is the only solution to problem \eqref{eq:local-func}. 

\medskip

Define a sequence of exit times from balls centered at $x$ as
\begin{equation}\label{eq:def-tau-bigball}
\tau_0(x):=0, \quad \tau_k(x)=\inf\{t:X_n\notin B_{kR}(x)\}, \quad k\ge 1.
\end{equation}
We may simply write $\tau_k(0)$ as $\tau_k$.

The following are some properties of the stopping time $T$ defined in \eqref{eq:def-stoppingtime}.
\begin{lemma}
\label{lem:stoppingtime}
Let $R>1$. For any $\eta:\Z^d\to[0,1]$ with $\eta|_{\Z^d\setminus B_{3R}}=1$, let $\lock=\lock(x;\eta,\psi)$ and $T=T(\eta)$  be as defined in Definition~\ref{def:new-corrector} and  \eqref{eq:def-stoppingtime}..
\begin{enumerate}[(i)]
\item\label{item:stoptime-1} Recall $\tau_n(x)$ in \eqref{eq:def-tau-bigball}. For all $k\ge 1, x\in \Z^d$, 
\begin{equation}
\label{eq:comp-tau-T}
P_\omega^x(T>\tau_k(x))\lesssim e^{-ck}.
\end{equation}
\item\label{item:stoptime-2} $E^x_\omega[T]\lesssim R^2$.
\end{enumerate}
\end{lemma} 

\begin{proof}
\noindent{\bf (i)} 
The intuition is as follows. From all but two spheres $\partial B_{kR}(x)$,  the RW has probability less than $1-c$ of not being killed by the exponential clock before reaching the next level $\partial B_{(k+1)R}(x)$. Hence, by iteration,  it would be exponentially hard to reach $\partial B_{kR}(x)$.

Indeed,  let $n_0\in\N$ be a constant to be determined later.
Since $(X_n)$ is a martingale on any $\omega\in\Omega$,  by Doob's inequality, for $k>0$,
\begin{align*}
P_\omega(\tau_{n_0}\le k R^2)
&\le 
\sum_{e:|e|=1}P_\omega(\sup_{n\le kR^2}X_n\cdot e\ge n_0 R/\sqrt{d})\\
&\le
\frac{\sqrt d}{n_0 R}
\sum_{e:|e|=1}E_\omega[(X_{kR^2}\cdot e)_+]\lesssim \frac{\sqrt{kR^2}}{n_0 R}
\lesssim
\frac{\sqrt k}{n_0}.
\end{align*}
Hence, when $n_0>8$ is chosen appropriately and $x$ satisfies $B_{n_0R}(x)\cap B_{3R}=\emptyset$,
\begin{align*}
P_\omega^x(\tau_{n_0}(x)<T)
&\le P_\omega^x(\tau_{n_0}(x)<n_0R^2)+P_\omega^x(T\ge n_0R^2)\\
&\le 
\tfrac{C\sqrt{n_0}}{n_0}+(1-\tfrac1{1+R^2})^{n_0R^2}<e^{-1},
\end{align*}
where we used the fact that $T$ behaves like a geometric random variable with parameter $\tfrac1{1+R^2}$ when we consider the RW outside of $B_{3R}$. For $k\ge 1$, whenever $\big(B_{(k+1)n_0R}(x)\setminus B_{(k-1)n_0R}(x)\big)\cap B_{3R}=\emptyset$, we have $B_{n_0R}(z)\cap B_{3R}=\emptyset$ for all $z\in\partial B_{kn_0R}(x)$ and thus
\begin{align*}
P_\omega^x(\tau_{(k+1)n_0}(x)<T|\tau_{kn_0}(x)<T)\le \max_{z\in\partial B_{n_0kR}}P_\omega^z(\tau_{n_0}(z)<T)<e^{-1}.
\end{align*}
Since for any $x\in\Z^d$, there are at most two $k$'s such that  $(B_{(k+1)n_0R}(x)\setminus B_{(k-1)n_0R}(x))\cap B_{3R}\neq\emptyset$, the above inequality implies, for all $k\ge 1$, $x\in\Z^d$
\[
P_\omega^x(\tau_{n_0k}(x)<T)
\le \prod_{j=1}^{k-1}P_\omega^x(\tau_{(j+1)n_0}(x)<T|\tau_{jn_0}(x)<T)
\le e^{3-k}.
\]
Inequality \eqref{eq:comp-tau-T} is proved.

\medskip

\noindent{\bf (ii) }
Since $(|X_n|^2-n)_{n\ge 0}$ is a $P_\omega$-martingale, by the optional stopping lemma one has 
$\qe^z[\tau_k]=E_\omega^z [|X_{\tau_k}|^2]-|z|^2\le (k+1)^2R^2$ for $z\in B_{kR}$.  
Further, observing that $T\le \tau_1+\sum_{k=1}^\infty(\tau_{k+1}-\tau_k)\mathbbm{1}_{T>\tau_k}$, by the strong Markov property and \eqref{eq:comp-tau-T},
\begin{align*}
E^x_\omega[T]
&\le E_\omega^x[\tau_1(x)]+\sum_{k=1}^\infty P_\omega^x(\tau_k(x)<T)\max_{z\in\partial B_{kR}(x)}E_\omega^z[\tau_{k+1}(x)]\\
&\lesssim
R^2+\sum_{k\ge 1}e^{-ck}(kR)^2\lesssim R^2.
\end{align*}

Our proof of Lemma~\ref{lem:stoppingtime} is complete.
\end{proof}

For $R\ge 1$,  we are interested in the set $H=H_R$ of functions defined by
\begin{equation}\label{eq:eta0}
H=\{
\eta\in[0,1]^{\R^d}: \eta|_{\R^d\setminus\B_{3 R}}=1, \eta|_{\B_{2R}}=0,
|D^i\eta|\lesssim\tfrac1{R^i}, \,i=1,2
\}
\end{equation}

We will derive the following estimates for local correctors $\lock_\omega(x;\eta,\psi), \eta\in H$.
\begin{lemma}
\label{lem:phi-bd-c11}
Let $R>1$. Recall $\lock=\lock_\omega(x;\eta,\psi)$ in  Definition~\ref{def:new-corrector} and  $H$ in \eqref{eq:eta0}.
\begin{enumerate}[(a)]
\item\label{item:new-corr-1} For any $\eta\in[0,1]^{\Z^d}$ with $\eta|_{\Z^d\setminus B_{3R}}=1$, we have
$\norm{\lock_\omega(x;\eta,\psi)}_\infty\lesssim R^2\norm{\psi}_\infty$.
\item\label{item:new-corr-2}
When $\eta\in H$, we have
 $|\nabla^2\lock_\omega(x;\eta,\psi)|\lesssim \tx_x^2\norm{\psi}_\infty$ for  all $x\in\Z^d, \eta\in H$.
\end{enumerate}
\end{lemma}
\begin{proof}
\eqref{item:new-corr-1} is an immediate consequence of Lemma~\ref{lem:stoppingtime}\eqref{item:stoptime-2} and formula \eqref{eq:phi-expression}.

To prove \eqref{item:new-corr-2},  for the simplicity of notations we simply write $\lock=\lock_\omega(x;\eta,\psi)$, $\eta\in H$.
By the H\"older estimate of Krylov-Safonov as in \eqref{eq:osc},
there exists $\gamma(d,\kappa)>0$ such that for any $x\in\Z^d, \eta\in H$,
\[
[\lock]_{\gamma;B_{R/2}(x)}\lesssim
R^{-\gamma}[\max_{B_R(x)}|\lock|+R^2\norm{\tfrac{1}{R^2}\lock\eta+\psi_\omega}_{d;B_R(x)}]
\lesssim 
R^{2-\gamma}\norm{\psi}_\infty,
\]
where we used statement \eqref{item:new-corr-1} in the last inequality.

Further, by Theorem~\ref{thm:c11} and statement \eqref{item:new-corr-1}, for any $x\in\Z^d, \eta\in H$,
\begin{align*}
|\nabla^2\lock(x)|
&\lesssim\tx_x^2\left(
\tfrac1{R^2}\osc_{B_{R/2}(x)}\lock+\norm{\psi+\tfrac1{R^2}\lock\eta}_\infty+R^\gamma[\tfrac1{R^2}\lock\eta]_{\gamma; B_{R/2}(x)}
\right)
\\
&\lesssim
\tx_x^2\left(
\norm{\psi}_\infty+R^{\gamma-2}[\lock]_{\gamma; B_{R/2}(x)}\norm{\eta}_\infty+R^{\gamma-2}\norm{\lock}_\infty [\eta]_{\gamma; B_{R/2}(x)}\
\right)
\\
&\lesssim
\tx_x^2\norm{\psi}_\infty.
\end{align*}
Statement \eqref{item:new-corr-2} is proved.
\end{proof}

Now we are ready to derive some estimates of the Green function $\locg_\omega$.

\begin{proposition}
\label{prop:new-corr-prop}
Let $R>1$.  Let $\locg_{\omega,y}(x)$, $\nu=\nu_R, H$ be as in \eqref{eq:LG}, \eqref{eq:nu}, \eqref{eq:eta0}.
\begin{enumerate}[(a)]
\item\label{item:new-corr-3} 
There exists $c_1=c_1(\kappa,d)>0$ such that, for any $x, y\in\Z^d$,  and   any $\eta\in[0,1]^{\Z^d}$ with $\eta|_{\Z^d\setminus B_{3R}}=1$, we have
\[\locg_\omega(x,y;\eta)\lesssim 
\tx_y^{d-1}e^{-c_1|x-y|/R}\nu(|x-y|).
\]
\item\label{item:new-corr-4}  When $\eta\in H$, for any $x\in B_R, y\in\Z^d$,  writing $\tx_{x,y}=\tx_x+\tx_y$, we have
\[
\abs{\nabla \locg_{\omega;y}(x;\eta)}\lesssim \tx_{x,y}^d
(|x-y|\wedge R+ 1)^{-1}e^{-c_1|x-y|/R}\nu(|x-y|).
\]
\end{enumerate}
\end{proposition}

Recall the Green functions $G_R=G_R^\omega$ and $G$ in Definition~\ref{def:green}.  

\begin{proof}
{\eqref{item:new-corr-3}}
First, consider $d\ge 3$.  By Theorem~\ref{thm:Green-bound} we have
$\locg_\omega(x,y;\eta)\lesssim G(x,y)\lesssim\tx^{d-1}_y\nu(|x-y|)$.  Hence \eqref{item:new-corr-3} is true when $|x-y|\le 4R$.  When $|x-y|>4R$,  without loss of generality,  assume $|x-y|=2nR$ for some $n\in\N$.  Recall $\tau_k(x)$ in \eqref{eq:def-tau-bigball}. Then, by \eqref{eq:comp-tau-T} and Theorem~\ref{thm:Green-bound},
\begin{align*}
\locg_\omega(x,y;\eta)
&\le 
P_\omega^x(T>\tau_n)\max_{z\in\partial B_{nR}(x)}G(z,y)\\&
\lesssim 
e^{-cn}\tx_y^{d-1}\max_{z\in\partial B_{nR}(x)}\nu(|z-y|)\\&
\lesssim
\tx_y^{d-1}e^{-c|x-y|/R}\nu(|x-y|).
\end{align*}
Hence \eqref{item:new-corr-3} is proved for $d\ge 3$. It remains to consider $d=2$.

When $d=2$ and $|x-y|\le 2R$,  by the strong Markov property,
\begin{align*}
\locg_\omega(x,y;\eta)
&\le 
E_\omega^x[
\sum_{n=0}^{\tau_2(x)-1}\mathbbm1_{X_n=y}+\sum_{k\ge 1}\mathbbm1_{T>\tau_k(x)}\sum_{n=\tau_k(x)}^{\tau_{k+1}(x)-1}\mathbbm1_{X_n=y}
]\\
&\le G_{2R}^{\theta_x\omega}(0,y-x)+
\sum_{k\ge 2}P_\omega^x(T>\tau_k(x))G_{(k+1)R}^{\theta_x\omega}(0,y-x)\\
&\lesssim 
\tx_y^{d-1}\big[\nu_{2R}(|x-y|)+\sum_{k\ge 2}e^{-ck}(\log(k+1)+\log R-\log|x-y|)\big]\\
&\lesssim
\tx_y^{d-1}\nu_R(|x-y|).
\end{align*}
In particular, $\max_{z\in\partial B_R(y)}\locg_\omega(z,y;\eta)\lesssim\tx_y^{d-1}$.

When $d=2$ and $|x-y|$>2R,  we let $K=\inf\{t\ge 0: Y_t\in \bar B_R(y)\}$.  Then
\begin{align*}
\locg_\omega(x,y;\eta)
&\le 
P_\omega^x(T>K)\max_{z\in\partial B_R(y)}\locg_\omega(z,y;\eta)\\
&\lesssim
P_\omega^x(T>\tau_{|x-y|/(2R)}(x))\tx_y^{d-1}
\\&
\stackrel{\eqref{eq:comp-tau-T}}\lesssim
e^{-c|x-y|/R}
\tx^{d-1}
\asymp
e^{-c|x-y|/R}
\tx^{d-1}\nu_R(|x-y|).
\end{align*}

Proposition~\ref{prop:new-corr-prop}\eqref{item:new-corr-3} is proved.

\medskip

\noindent{\eqref{item:new-corr-4}}
By Proposition~\ref{prop:new-corr-prop}\eqref{item:new-corr-3}, when $\eta\in H$,
\[
\abs{\nabla \locg_{\omega;y}(x;\eta)}
\lesssim
\max_{\bar B_1(x)}\locg_\omega(\cdot, y;\eta)\lesssim \tx_y^{d-1}e^{-c|x-y|/R}\nu(|x-y|).
\]
Hence the statement is true when $|x-y|\wedge R\le \tx_{x,y}$.
We only need to consider the case $R>\tx_{x,y}$ and $|x-y|>\tx_{x,y}$.

When $\tx_{x,y}<|x-y|$,  since $B_{(|x-y|\wedge R)/2}(x)\subset B_{2R}\setminus\{y\}$, 
the function $z\mapsto\locg_\omega(z,y;\eta)$ is $\omega$-harmonic on $B_{(|x-y|\wedge R)/2}(x)$. 
Hence, by Theorem~\ref{thm:c11} and \eqref{item:new-corr-3},
\begin{align}\label{eq:c01-locg}
|\nabla\locg_{\omega;y}(x)|
&\lesssim
\frac{\tx_{x,y}}{|x-y|\wedge R}\osc_{B_{(|x-y|\wedge R)/2}(x)}\locg_{\omega;y}\\
&\lesssim
\tx_{x,y}^{d}(|x-y|\wedge R)^{-1}e^{-c|x-y|/R}\nu(|x-y|).\nn
\end{align}

Our proof is complete.
\end{proof}

\begin{remark}
As can be seen in \eqref{eq:c01-locg},  the fact that $\locg_\omega(\cdot,y;\eta)$, $\eta\in H$, is $\omega$-harmonic in $B_{2R}\setminus\{y\}$, i.e., $\eta|_{B_{2R}}=0$, is crucial for the bound  of $\nabla\locg_\omega$ as above.

Note that the  Green function $\apg_\omega$ (as defined in \eqref{eq:def-apg}) of the approximate corrector solves the equation (which corresponds to the case $\eta\equiv 1$ of \eqref{eq:LG})
\[
L_\omega \apg(\cdot,y)=\tfrac{1}{R^2}\apg(\cdot,y)-\mathbbm1_y.
\]
Since $\apg_\omega$ is nowhere $\omega$-harmonic,   the gradient bound as \eqref{eq:c01-locg} cannot be obtained for $\apg_\omega$ using Theorem~\ref{thm:c11}. 
\end{remark}

\subsection{Homogenization of the local corrector} \label{subsec:local corr}

From now on, we {\it fix} a  smooth function $\eta_0\in C^\infty(\R^d)$ that satisfies
\[
\eta_0\in[0,1], \, 
\eta_0|_{\R^d\setminus\B_{8/3}}=1,\,
\eta_0|_{\B_{7/3}}=0
\]
and set, for $R>0$,
\begin{equation}
\label{eq:eta-R}
\eta_R(x)=\eta_0(\tfrac{x}{R\vee 1}),\quad x\in\R^d.
\end{equation}
Recall $H=H_R$ in \eqref{eq:eta0}. Note that $\eta_R\in H_R$ for $R\ge 1$.

The goal of this subsection is to obtain the optimal sizes for $\lock(x;\eta_R,\psi)$ and its discrete gradient $\nabla\lock(x;\eta_R,\psi)$.  

\begin{theorem}\label{thm:new-corrector-bd} Recall $\delta(R)$, $\mu(R)$, $\eta_R$ in  \eqref{eq:def-delta},\eqref{eq:def-mu},\eqref{eq:eta-R}.  Let $\lock=\lock_{\omega,R}(\cdot;\eta,\psi)$.
For any $q_1\in(0,\tfrac{2d}{3d+2})$, $q_2\in(0,\tfrac{d}{2d+2})$,  
\begin{enumerate}[(a)]
\item\label{item:lock-bd} when $\eta\in H_R$, we have $\mb E [\exp(c|\tfrac{\lock(x)}{\mu(R)}|^{q_1})]<C_{q_1} \quad \text{ for all }x\in\Z^d$;
\item\label{item:lock-c01} when $\eta=\eta_R$, we have $\mb E [\exp(c|\tfrac{\nabla\lock(x)}{\delta(R)}|^{q_2})]<C_{q_2} \quad \text{ for }x\in B_R.$
\end{enumerate}
\end{theorem}

\begin{lemma}\label{lem:elock}
Recall $\mu(R)$, $H_R$ in \eqref{eq:def-mu},\eqref{eq:eta0}. Let $R>1$. When $\eta\in H_R$,  we have
\[
\mb E[|\lock(x)|]\lesssim\mu(R) \quad\text{ for all }x\in\Z^d.
\]
\end{lemma}

\begin{proof}
Our proof is through the comparison between $\lock$ and $\apk$. By Lemma~\ref{lem:approx-cor-concentration},  there exists a random variable $\ms Y(\omega)$ with $\mb E [\ms Y^2]<C$ such that $\forall y\in\Z^d$,
\begin{equation}\label{eq:the-apk-bd}
|\apk(y)|\lesssim\mu(R)\ms Y(\theta_y\omega)\quad \text{for $\mb P$-a.s. $\omega$}.
\end{equation}

Consider $u=\lock-\apk$. Then $u$ solves the equation
\[
L_\omega u=\tfrac 1{R^2}u\eta-\tfrac 1{R^2}\apk(1-\eta)
\quad\text{ in }\Z^d.
\]
Hence, by formula \eqref{eq:phi-in-green},  the bounds \eqref{eq:the-apk-bd} and Proposition~\ref{prop:new-corr-prop}\eqref{item:new-corr-3}, for $x\in\Z^d$,
\begin{align*}
|u(x)|
&=\Abs{\sum_{y\in B_{3R}} \locg_\omega(x,y)\tfrac1{R^2}\apk(y)(1-\eta(y))}\\
&\lesssim
\frac{\mu(R)}{R^2}\sum_{y\in B_{3R}} \ms Y(\theta_y\omega)\tx_y^{d-1}e^{-c|x-y|/R}\nu_R(|x-y|).
\end{align*}
Note that $e^{-cr/R}\nu_R(r)$ is non-increasing in $r$, and 
\[
\mb E[\ms Y\tx^{d-1}]\le \norm{\ms Y}_{L^2(\mb P)}\norm{\tx^{d-1}}_{L^2(\mb P)}<C. 
\]  
Hence,  taking expectations in the above inequality,  for any $x\in\Z^d$,
\[
\mb E[|u(x)|]\lesssim 
\frac{\mu(R)}{R^2}\sum_{y\in B_{3R}}e^{-c|x-y|/R}\nu_R(|x-y|) 
\lesssim
\frac{\mu(R)}{R^2}\sum_{z\in B_{3R}}\nu_R(|z|) \lesssim \mu(R).
\]
This bound, together with \eqref{eq:the-apk-bd} and a triangle inequality, yields the lemma.
\end{proof}

\begin{proof}
[Proof of Theorem~\ref{thm:new-corrector-bd}\eqref{item:lock-bd}]
By \eqref{eq:def-lock} and formula \eqref{eq:vdlaplace},  $\vd{y}\lock$ satisfies, for $x,y\in\Z^d$, the equation
\[
L_{\omega}\vd{y}\lock(x)
=
\tfrac{1}{R^2}\vd y\lock\eta+[\vd y\psi(y)-\tfrac12\tr(\vd y a\nabla^2\lock_{\omega'_y})(y)]\mathbbm{1}_{y=x},
\]
which, by formula \eqref{eq:phi-in-green}, yields
\begin{equation}\label{eq:sensiti-formu-phi}
\vd{y}\lock(x)=
-[\vd y \psi(y)-\tfrac12\tr(\vd y a\nabla^2\lock_{\omega'_y})(y)]\locg_{\omega}(x,y).
\end{equation}
One may compare this expression to \eqref{eq:w-appr-gf}.  Then, using the bounds of $\nabla^2\lock$ and $\locg_\omega$ in Lemma~\ref{lem:phi-bd-c11}\eqref{item:new-corr-2} and Proposition~\ref{prop:new-corr-prop}\eqref{item:new-corr-3},we obtain 
\[
|\vd{y}\lock(x)
|\lesssim
{{\tx}_y'}^{2}\tx_y^{d-1}\norm{\psi}_\infty e^{-c|x-y|/R}
\nu(|x-y|).
\]
By \eqref{eq:nu2-sum} and applying Lemma~\ref{lem:efron-stein} to  $Z:=\frac{\lock(x)}{\norm{\psi}_\infty \mu(R)}$,  we get,
for any $0<p<\tfrac{2d}{3d+2}$, 
\[
\mb E[\exp(c_p|Z-\mb EZ|^p)]\le C.
\]
Since it was shown  in Lemma~\ref{lem:elock} that $\mb E[|Z|]\lesssim 1$,  Theorem~\ref{thm:new-corrector-bd}\eqref{item:lock-bd} follows.
\end{proof}

\begin{lemma}\label{lem:ball-modif-a-bit}
Recall $\mu(R)$, $\eta_R$ in \eqref{eq:def-mu},\eqref{eq:eta-R}.  
We fix $R>1$ and let $\lock_\omega=\lock_{\omega,R}(\cdot;\eta_R,\psi)$. 
Then, for any $x\in\Z^d, z\in B_{R/3}$,  we have
\begin{equation}\label{eq:230225-1}
\mb E
\abs{
\lock_\omega(x+z)-\lock_{\theta_z\omega}(x)
}
\lesssim
\frac{|z|\mu(R)}{R}.
\end{equation}
As a consequence,  for any $x\in\Z^d, z\in B_{R/3}$,
\begin{equation}
\label{eq:230225-2}
\abs{\mb E
[\lock_\omega(x+z)-\lock_\omega(x)]
}
\lesssim
\frac{|z|\mu(R)}{R}.
\end{equation}
\end{lemma}
\begin{proof}
Let $\eta_z$ denote the function $\eta_z(x):=\eta_R(x-z), \, \forall x\in\Z^d$.
By our definitions of $H_R$ and $\eta_R$ in \eqref{eq:eta0} and \eqref{eq:eta-R},  we still have $\eta_z\in H_R$ for $|z|<R/3$.
Write $\phi_1=\lock_\omega(x;\eta_R,\psi)$ and $\phi_2=\lock_\omega(x;\eta_z,\psi)$. 

First, observe that for any $z\in B_{R/3}, x\in\Z^d$,
\[
\lock_\omega(x+z;\eta_z,\psi)=\lock_{\theta_z \omega}(x;\eta_R,\psi).
\]
Hence, for $z\in B_{R/3}, x\in\Z^d$,
\[
\abs{\lock_\omega(x+z;\eta_R,\psi)-\lock_{\theta_z\omega}(x;\eta_R,\psi)}=\abs{\phi_1(x+z)-\phi_2(x+z)}
\]

Next, notice that $u:=\phi_1-\phi_2$ solves the equation
\[
L_\omega u=\tfrac1{R^2}u \eta_R+\tfrac1{R^2}\phi_2(\eta_R-\eta_z).
\]
Hence, by formula \eqref{eq:phi-in-green},  Theorem~\ref{thm:new-corrector-bd}\eqref{item:lock-bd} and Proposition~\ref{prop:new-corr-prop}\eqref{item:new-corr-3},  we have, with $\locg_\omega=\locg_\omega(\cdot,\cdot;\eta_R,\psi)$, for all $x\in\Z^d$, 
\begin{align*}
|u(x)|
&=\sum_{y\in B_{3R+1}} \locg_\omega(x,y)\tfrac1{R^2}\phi_2(y)[\eta_R(y)-\eta_R(y-z)]\\
&\lesssim
\frac{\mu(R)|z|}{R^3}\sum_{y\in B_{3R+1}} \ms Y_1(y)\tx_y^{d-1}e^{-c|x-y|/R}\nu_R(|x-y|).
\end{align*}
Taking expectations on both sides, we get, for $z\in B_{R/3}, x\in\Z^d$,
\[
\mb E[|u(x)|]
\lesssim
\frac{\mu(R)|z|}{R^3}\sum_{y\in B_{3R}}\nu_R(|z|) \lesssim \frac{\mu(R)|y|}{R}.
\]
We have proved \eqref{eq:230225-1}. 

Finally, display \eqref{eq:230225-2} follows from the fact that  $\mb E[\lock_\omega(x)]=\mb E[\lock_{\theta_z\omega}(x)]$ (by the translation-invariance of $\mb P$).
\end{proof}

\begin{proof}[Proof of Theorem~\ref{thm:new-corrector-bd}\eqref{item:lock-c01}]
 First, we will show that 
for any $p\in(0,\tfrac{d}{2d+2})$, there exists a constant $c_p>0$  such that
\begin{equation}\label{eq:230215}
\mb E
\left[
\exp\left(
c_p|\nabla\lock(x)-\mb E\nabla\lock(x)|/\delta(R)|^p
\right)\right]\le C
\quad\text{ for all }x\in B_R.
\end{equation}

By \eqref{eq:sensiti-formu-phi}, we have
$
\vd{y}\nabla\lock(x)=
-[\vd y \psi(y)-\tfrac12\tr(\vd y a\nabla^2\lock_{\omega'_y})(y)]
\nabla_x \locg_{\omega}(x,y).
$
Here the subscript of $\nabla_x$ indicates that $\nabla$ is applied only to the coordinate $x$. 
Hence, by Lemma~\ref{lem:phi-bd-c11}\eqref{item:new-corr-2} and Proposition~\ref{prop:new-corr-prop}\eqref{item:new-corr-4}, for any $x\in B_R$, $y\in\Z^d$,
\[
|\vd{y}\nabla\lock(x)|\lesssim\tx'^2_y\tx_{x,y}^d\norm{\psi}_\infty 
F_R(|x-y|),
\]
where (Recall $\nu_R$ in \eqref{eq:nu})
$F_R(r)=
\frac{e^{-c_1r/R}}{r\wedge R+ 1}\nu_R(r)$.
Note that $\mb E[(\tx'^2_y\tx_{x,y}^d)^n]\lesssim\mb E[\tx^{(2+d)n}]$ for all $y\in\Z^d,n\ge 1$.
Further, computations show that, for $x\in B_R$,
\[
\sum_y F_R(|x-y|)^2\lesssim
\int_0^\infty F_R^2(r)r^{d-1}\dd r
\asymp \delta(R)^2.
\]
Applying Lemma~\ref{lem:efron-stein} to $Z=\nabla\lock(x)$,  inequality \eqref{eq:230215} follows.

It remains to show that $\abs{\mb E[\nabla\lock(x)]}\lesssim\delta(R)$. By \eqref{eq:230225-2} in  Lemma~\ref{lem:ball-modif-a-bit},  we have
\begin{align*}
\abs{\mb E[\nabla\lock(x)]}
&\lesssim
\frac{\mu(R)}{R}\lesssim 1\le \delta(R).
\end{align*}
This estimate, together with \eqref{eq:230215}, yields
Theorem~\ref{thm:new-corrector-bd}\eqref{item:lock-c01}.
\end{proof}

\subsection{Quantitative homogenization for the Dirichlet problem \eqref{eq:elliptic-dirich}: proof of Theorem~\ref{thm:opt-quant-homo}}

\begin{proof}
[Proof of Theorem~\ref{thm:opt-quant-homo}] 
Recall $\lock=\lock_{\omega,R}$ in \eqref{def:new-corrector}.  Let $v^k$, $\xi$ be the functions
\begin{align*}
&v^k(x)=\lock(x;\eta_R,\tfrac{\omega_k(x)}{2\tr\omega(x)}), \quad k=1,\ldots,d\\
&\xi(x)=\lock(x;\eta_R,\tfrac{\zeta}{\tr\omega(0)}).
\end{align*}
For $q\in(0,\tfrac{d}{2d+2})$,  setting $\ms Y_1=\frac{1}{\delta(R)}(\sum_{k=1}^d
\norm{\nabla v^k}_{d;B_R}+
\norm{\nabla \xi}_{d;B_R})
$ and
\[
\ms Y_2=\frac{1}{\mu(R)}(\sum_{k=1}^d \osc_{\bar B_R}v^k+\osc_{\bar B_R}\xi)-C(\log R)^{1/q},
\]
by Theorem~\ref{thm:new-corrector-bd} and Lemma~\ref{fact:max}, we have, with $\ms Y=\ms Y_1+\ms Y_2$,
\[
\mb E[\exp(c\ms Y^q)]<C_q, \quad\forall x\in B_R.
\]
Thus, by Lemma~\ref{lem:two-scale-exp}, 
\[
\max_{x\in B_R}|u(x)-\bar{u}(\tfrac xR)|
\lesssim
\tfrac{\ms Y}R\norm{\bar u}_{C^4(\bar{\B}_1)}
[\delta(R)+\tfrac{1}{R}\mu(R)(\log R)^{1/q}].
\]

Theorem~\ref{thm:opt-quant-homo}  follows.
\end{proof}

\section{The global correctors: existence, uniqueness,  and stationarity}
Although the local correctors $\lock$ constructed in the previous subsections are good enough for us to obtain optimal rates for the homogenization of Dirichlet problems,  the global correctors,  if they exist,  could be more useful in investigating homogenization problems with (or without) other boundary conditions.

The goal of this section is to prove Theorem~\ref{thm:global_krt}.

Throughout this section,  we  let $\lock_R=\lock_{\omega,R}(x;\eta_R,\psi)$ be as defined in  Definition~\ref{def:loc-global}, where $\eta_R$ is chosen as in \eqref{eq:eta-R}.

\begin{remark}
The global correctors will be constructed as a.s. limits of the local correctors tilted by appropriate ($\omega$-dependent) constants or affine functions. In this process the regularity property of the local Green function $\locg_\omega$ is crucial.  

The global correctors in Theorem~\ref{thm:global_krt} have the same sizes as the local correctors in Theorem~\ref{thm:new-corrector-bd} except in $d=2$ when \eqref{item:gkrt-2} only provides
a bound with an additional factor of size $(\log|x|)^{3/2}$.  We also lost some stochastic integrability in the exponents compared to Theorem~\ref{thm:new-corrector-bd}. 

Note that although the global corrector is unique up to the shift of an affine function,  an affine shift with $\omega$-dependent coefficients could drastically change the moment estimates of $\krt$.
\end{remark}

\begin{proof}
[Proof of the uniqueness in Theorem~\ref{thm:global_krt} assuming existence]

Suppose there are two global correctors $\phi_1,\phi_2$ with properties  \eqref{item:gkrt-2}-\eqref{item:gkrt-4} in Theorem~\ref{thm:global_krt}. Then  $u=\phi_1-\phi_2$ is a $\omega$-harmonic function on the whole $\Z^d$.  Moreover, when $d\ge 3$,
\begin{align*}
P(\max_{B_R}|u|>\error R)
&\lesssim 
\sum_{x\in B_R}P(|u(x)|>\error R)\\
&\le 
\sum_{x\in B_R}\mb E[\exp(c|\tfrac{u(x)}{\mu(R)}|^{1/3})]\exp(-c|\tfrac{\error R}{\mu(R)}|^{1/3})\\
&\lesssim
\exp(-c_\error R^{1/6}), \quad \forall \error>0, d\ge 3.
\end{align*}
Hence, by Borel-Cantelli's lemma, $\lim_{R\to\infty}\max_{B_R}|u|/R=0$ almost surely.
By \eqref{eq:c2}, this sublinear estimate of $\max_{B_R}|u|$ implies that $u$ is almost surely a constant.

When $d=2$, similar argument gives, for all $\varepsilon>0$,
\begin{align*}
P(\max_{B_R}|u|>\error R^2)
&\lesssim 
\sum_{x\in B_R}\mb E[\exp(c|\tfrac{u(x)}{R}|^{1/3})]\exp(-c|\tfrac{\error R^2}{R}|^{1/3})
\lesssim
\exp(-c_\error R^{1/6}).
\end{align*}
Hence, by Borel-Cantelli's lemma, $\lim_{R\to\infty}\max_{B_R}|u|/R^2=0$ almost surely.
By \eqref{eq:c2}, this subquadratic estimate of $\max_{B_R}|u|$ implies that $\nabla^2 u=0$ almost surely. That is,  $u$ is a.s. an affine function (with coefficients possibly dependent on $\omega$).
\end{proof}

\subsection{Stationary global corrector in \texorpdfstring{$d\ge 5$}{d>4}}
As a consequence of Theorem~\ref{thm:quant-ergo}, we can show the existence of a stationary corrector in $d\ge 5$.

In the continuous PDE setting,  the existence of the stationary corrector in $d\ge5$ and stochastic integrability \eqref{item:gkrt-2} with $p=\tfrac12$ was proved in \cite[Theorem~7.1]{AL-17}.  
In the discrete setting, the existence of the stationary corrector in $d\ge5$ 
with $L^2$-stochastic integrability was proved in \cite[Corollary 7]{GT-22}.
\begin{proof}
[Proof of Theorem~\ref{thm:global_krt} for $d\ge 5$]
Without loss of generality, assume $E_\Q[\psi]=0$. 

When $d\ge 5$,  we let $\krt_\omega(x)=\qe^x[\int_0^\infty \psi(\evp s)\dd s]$.  The existence of  $\krt_\omega$ (as an a.s.  limit of $\qe^x[\int_0^n \psi(\evp s)\dd s]$) follows immediately from 
Theorem~\ref{thm:var_decay}.  It clearly  solves \eqref{global-corrector-eq}.  Further, by Fatou's lemma and 
Theorem~\ref{thm:quant-ergo}, for $p\in(0,\tfrac{2d}{3d+2})$,
\[
\mb E[\exp(c|\krt|^p)]
\le 
\liminf_{n\to\infty}
\mb E[\exp(c|\int_0^n \psi(\evp s)\dd s|^p)]\le C_p.
\]
The stationarity of $\krt$ follows from the stationarity of $\apk$.
%
\end{proof}

\subsection{Global corrector in $d=3,4$ with stationary gradient}

\begin{proof}
[Proof of Theorem~\ref{thm:global_krt} for $d=3,4$]

Without loss of generality, assume $\bar\psi=0$.  
Recall $\lock_R=\lock_{\omega,R}(x;\eta_R,\psi)$ in Definition~\ref{def:loc-global}.
For $R>1,\omega\in\Omega$, let
\[
\phi_R(x)=\phi_{\omega,R}(x):=\lock_R(x)-\lock_R(0), \quad x\in\Z^d.
\]
In what follows we will show that, by taking $R\to\infty$, $\phi_R$ will converge to a desired global corrector up to a subsequence.

\noindent{\bf Step 1.} {\it (Moment bounds.)} First, when $d\ge 3$, we will obtain the following moment bounds: for any $x\in\Z^d$,  (note that $\delta(\cdot)\equiv 1$ for $d\ge 3$.)
\begin{equation}\label{eq:230215-2}
\mb E\left[\exp\big(
C|\tfrac{\phi_R(x)}{\mu(|x|)}|^p
\big)\right]
\le C_p,  \quad\forall p\in(0,\tfrac{2d}{3d+4}),R\ge |x|^2\vee 1.
\end{equation}
\begin{equation}
\label{eq:230226-1}
\mb E[\exp(c|\nabla\phi_R(x)/\delta(|x|)|^q)]\le C_q \quad \forall
q\in(0,\tfrac{d}{2d+2}).
\end{equation}
Since $\nabla\phi_R=\nabla\lock_R$,  display \eqref{eq:230226-1} follows directly from Theorem~\ref{thm:new-corrector-bd}\eqref{item:lock-c01}.

We will prove \eqref{eq:230215-2} via sensitivity estimates. Indeed,
by \eqref{eq:sensiti-formu-phi},  
\begin{align}
\label{eq:phiR-sensi}
|\vd{y}\phi_R(x)|
&=
\abs{[\vd y \psi(y)-\tfrac12\tr(\vd y a\nabla^2\lock_{\omega'_y})(y)][\locg_{\omega}(x,y)-\locg_{\omega}(0,y)]}\nn\\
&\stackrel{Lemma~\ref{lem:phi-bd-c11}\eqref{item:new-corr-2}}\lesssim
\tx_y^{'2}\abs{\locg_{\omega}(x,y)-\locg_{\omega}(0,y)}
\quad \text{ for }x,y\in\Z^d.
\end{align}
We consider two cases: $|y|\le 2|x|$ and $|y|>2|x|$.

When $|y|\le 2|x|$, by \eqref{eq:phiR-sensi} and Proposition~\ref{prop:new-corr-prop}, (recall that $\nu(r)=r^{2-d}$ for $d\ge 3$)
\[
|\vd y\phi_R(x)|\lesssim 
\tx_y^{'2}\tx_y^{d-1}(\nu(|x-y|)+\nu(|y|)).
\]
When $|y|>2|x|$, noting that $|y|-\tfrac{|y|\wedge R}{2}\asymp |y|$ and that the function $z\mapsto \locg_{\omega}(z,y)$ is $\omega$-harmonic on $B_{2R}$, by \eqref{eq:phiR-sensi} and Theorem~\ref{thm:c11},
\[
|\vd y\phi_R(x)|\lesssim 
\tx_y^{'2}
\tfrac{|x|+\tx}{|y|\wedge R}\osc_{B_{|y|\wedge R}}\locg_{\omega}(\cdot,y)
\lesssim
\tx_y^{'2}\tx_y^{d-1}
\tfrac{|x|+\tx}{|y|\wedge R}e^{-c|y|/R}\nu(|y|).
\]
Fix $x$ and let $f(y):=(\nu(|x-y|)+\nu(|y|))\mathbbm1_{|y|\le 2|x|}+\tfrac{|x|+1}{|y|\wedge R}e^{-c|y|/R}\nu(|y|)\mathbbm1_{|y|>2|x|}$. Combining the two cases above, we have, for $n\ge 1$,
\[
\mb E[c|\vd y\phi_R(x)/f(y)|^n]\lesssim
\mb E[c\tx^{(d+2)n}].
\]
Note that \eqref{eq:nu2-sum} implies $\sum_{y\in B_r}\nu(|y|)^2\lesssim\mu(r)^2$.  Thus, for $x\neq 0$, $d\ge 3$, $R\ge |x|^2$,
\begin{align*}
\sum_y f(y)^2
&\lesssim
\sum_{y:|y|\le2|x|}\nu(|x-y|)^2+\nu(|y|)^2
+\sum_{y:|y|>2|x|}\frac{|x|^2}{(|y|\wedge R)^2}e^{-c|y|/R}\nu(|y|)^2\\
&\lesssim
\sum_{y\in B_{3|x|}}\nu(|y|)^2
+|x|^2\sum_{2|x|<|y|<R}\frac{1}{|y|^2}\nu(|y|)^2+
\frac{|x|^2}{R^2}\sum_{|y|>R}e^{-c|y|/R}\nu(|y|)^2\\
&\lesssim
\mu(|x|)^2
+|x|^2\int_{2|x|}^R\frac{1}{r^2}r^{2(2-d)}r^{d-1}\dd r+\frac{|x|^2}{R^2}\int_R^\infty e^{-cr/R}r^{2(2-d)}r^{d-1}\dd r\\
&\lesssim 
\mu(|x|)^2+|x|^2 R^{2-d}\lesssim\mu(|x|)^2.
\end{align*}
Hence, applying Lemma~\ref{lem:efron-stein} to $Z=\phi_R(x)$ we have,   for $R\ge |x|^2$, $p\in(0,\tfrac{2d}{3d+4})$, 
\[
\mb E\left[\exp\big(
C|(Z-\mb EZ)/\mu(|x|)|^{p}
\big)\right]
\le C_p, \text{ for }x\neq 0.
\]
Since by Lemma~\ref{lem:ball-modif-a-bit}, $|\mb E[Z]|\lesssim|x|\tfrac{\mu(R)}{R}\lesssim\mu(|x|)$, display \eqref{eq:230215-2} is proved.

\noindent{\bf Step 2.} {\it (Point-wise a.s. convergence.)} Next, we will show that, for any $x\in\Z^d$, 
every subsequence of  $\{\phi_R(x)\}_{R>1}$ contains an
 a.s. convergent subsequence. To this end, it suffices to show that (cf. \cite[Theorem~5, pg 258]{Shiryaev-96}) the sequence  $\{\phi_R(x)\}_{R>1}$ is {\it Cauchy  in probability}, i.e.,  for every $\varepsilon>0$,
\begin{equation}\label{eq:cauchy-in-prob}
\mb P(|\phi_n(x)-\phi_m(x)|>\varepsilon)\to 0 \quad \text{ as }m,n\to\infty.
\end{equation}
Observe that for any $n>m>1$, the function $x\mapsto\phi_n(x)-\phi_m(x)$ is $\omega$-harmonic in $B_{2m}$. Hence, by Theorem~\ref{thm:c11},  when $m>|x|^2$, 
\begin{align*}
|\phi_n(x)-\phi_m(x)|
\le\osc_{\bar B_{|x|}}(\phi_n-\phi_m)
\lesssim
\tfrac{|x|+\tx}{\sqrt m}\tfrac{1}{\#B_{\sqrt m}}\sum_{z\in B_{\sqrt m}}
|\phi_n(z)-\phi_m(z)|
\end{align*}
and thus, by \eqref{eq:230215-2} and H\"older's inequality we have, for $x\neq 0$, $n>m>|x|^2$, 
\begin{align*}
\norm{\phi_n(x)-\phi_m(x)}_{L^2(\mb P)}
&\lesssim
\tfrac{|x|+1}{\sqrt m}\mu(\sqrt m)
\stackrel{m\to\infty}\to 0.
\end{align*}
Hence \eqref{eq:cauchy-in-prob} follows by Chebyshev's inequality.
\medskip

\noindent{\bf Step 3.} {\it (Existence of the global corrector.)}
Since $\Z^d$ is a countable set,  by Step 2 and a diagonal argument, we can select a subsequence  $\{\phi_{R_n}\}_{n\in\N}$ such that $\phi_{R_n}(x)$ converges $\mb P$-almost surely to a function $\krt(x)=\krt_\omega(x)$ for all $x\in\Z^d$ as $n\to\infty$.  

Clearly, $\krt_\omega$ is a global corrector.

Moreover, by Fatou's lemma and \eqref{eq:230215-2} \eqref{eq:230226-1} in Step 1, for any 
$p\in(0,\tfrac{2d}{3d+4}), q\in(0,\tfrac{d}{2d+2})$, and any $x\in\Z^d\setminus\{0\}$,
\begin{align*}
\mb E\left[\exp\big(
C|\tfrac{\krt(x)}{\mu(|x|)}|^{p}
\big)\right]
&\le 
\liminf_{n\to\infty}
\mb E\left[\exp\big(
C|\tfrac{\phi_{R_n}(x)}{\mu(|x|)}|^{p}
\big)\right]
\le 
C_p,\\
\mb E[\exp(c|\nabla\krt(x)|^q)]
&\le 
\liminf_{n\to\infty} \mb E[\exp(c|\nabla\phi_{R_n}(x)|^q)]\le C_q.
\end{align*}

\noindent{\bf Step 4.} {\it (Stationarity of the gradient.)}
Observe that, for $x\in\Z^d$, the map $y\mapsto\krt_\omega(x+y)$ is a global corrector in the environment $\theta_x\omega$. Moreover, since $\krt_\omega$ is unique up to an additive constant (which may depend on $\omega$), the gradient field $\{\nabla\krt_\omega(x):x\in\Z^d\}$ is unique, and hence we have $\nabla\krt_{\theta_x\omega}(\cdot)=\nabla\krt_\omega(x+\cdot)$ for a.e. $\omega$.  In particular, $\nabla\krt_{\theta_x\omega}(0)=\nabla\krt_\omega(x)$ for a.e. $\omega$ and all $x\in\Z^d$.

Therefore,  $\nabla\krt_\omega(x)$ and $\nabla\krt_\omega(0)$ are identically distributed for all $x\in\Z^d$ by the stationarity of $\mb P$.
\end{proof}

\subsection{Global corrector in $d=2$ with stationary second order differences}


 Recall $\locg, \eta_R$ in \eqref{eq:LG}, \eqref{eq:eta-R}.  Let $R>1$ be fixed. For any $x,y\in\Z^d$, consider the potential kernel $\loca=\loca_\omega(\cdot,\cdot;\eta_R)$ corresponding to the local corrector:
\[
\loca(x,y):=\locg(y,y;\eta_R)-\locg(x,y;\eta_R).
\]
We will first show the following estimate.
\begin{lemma}
\label{lem:loca_bound}
Assume {\rm (A1), (A2)}, and fix $R>1$. For $\error>0$, let $\tx$ be the same as in Theorem~\ref{thm:hk-bounds}. 
For $\mb P$-a.s. $\omega$ and $x,y\in B_{R/2}$, we have
\[
0\le \loca(x,y)\lesssim \tx_y \log(|x-y|\vee 2)\quad \text{when }d=2.
\]
\end{lemma}
\begin{proof}
Let $x,y\in B_R$ be fixed. We only consider the non-trivial case $y\neq x$. 

Note that by \eqref{eq:LG} and the definition of $\eta_R$ in \eqref{eq:eta-R}, 
the function $z\mapsto\loca(z,y)$ is $\omega$-harmonic for $z\in B_{7R/3}\setminus\{y\}$.  Hence, applying the Harnack inequality (Theorem~\ref{thm:harnack-para}) for $\omega$-harmonic functions to a constant numbers of balls centered on $\partial B_{|x-y|}(y)$ with radii $|x-y|/2$,  we have
\[
\loca(z,y)\asymp\loca(x,y) \quad\text{ for all }z\in\partial B_{|x-y|}(y).
\]
In particular, letting $\tau_{r,y}=\inf\{t:Y_t\notin B_r(y)\}$, we have
\begin{equation}
\label{eq:loca-harnack}
\loca(x,y)\asymp\qe^y[\loca(Y_{\tau_{|x-y|,y}}, y)].
\end{equation}

We claim that (Recall $G_R^\omega$ in \eqref{def:green})
\begin{equation}\label{eq:loca-claim}
\qe^y[\loca(Y_{\tau_{|x-y|,y}}, y)]
=G_{|x-y|}^{\theta_y\omega}(0,0).
\end{equation}
Indeed, the function 
\[
v(z):=
\locg(z,y)-\qe^z[\locg(Y_{\tau_{|x-y|,y}},y)]
\]
satisfies $L_\omega v(z)=-\mathbbm 1_{z=y}$ for $z\in B_{|x-y|}(y)$ and $v|_{\partial B_{|x-y|}(y)}=0$, whereas 
$u(z)=G_{|x-y|}^{\theta_y\omega}(z-y,0)$ 
satisfies $L_\omega u(z)=-\mathbbm 1_{z=y}$ for $z\in B_{|x-y|}(y)$ and $u|_{\partial B_{|x-y|}(y)}=0$. Hence
$v(z)=u(z)$ for all $z\in\bar{B}_{|x-y|}(y)$. In particular, 
\begin{align*}
u(y)=v(y)=\locg(y,y)-\qe^y[\locg(Y_{\tau_{|x-y|,y}},y)]=\qe^y[\loca(Y_{\tau_{|x-y|,y}}, y)].
\end{align*}
This  implies \eqref{eq:loca-claim}.

Finally, the lemma follows from \eqref{eq:loca-harnack}, \eqref{eq:loca-claim} and  Theorem~\ref{thm:Green-bound}.
\end{proof}

\begin{lemma}
\label{lem:230313}
Recall $\nabla$ in Definitions~\ref{def:differences}.
Let $\nabla^+:=(\nabla_{e_i})_{1\le i\le d}$. 
For any $\omega$-harmonic function $u$ on $B_R$, we have
\[
\abs{u(x)-x\cdot\nabla^+u(0)-u(0)}\lesssim
\frac{\tx^2|x|^3}{R^2}\osc_{B_R}u \quad\text{ for all }x\in B_R.
\]
\end{lemma}
\begin{proof}
There is nothing to prove when $x=0$. 

When $x\neq 0$, let $B^x$ denote the smallest ball centered at $0$ that contains $x$.
Observe that, for any affine function $\ell(x)=a\cdot x+b$,  with $u_\ell=u-\ell$, we have $u(x)-x\cdot\nabla^+u(0)-u(0)=u_\ell(x)-x\cdot\nabla^+u_\ell(0)-u_\ell(0)$. Hence
\[
\abs{u(x)-x\cdot\nabla^+u(0)-u(0)}
\le \der{B^x}^2(u)
\lesssim
\frac{|x|(\tx+|x|)^2}{R^2}\osc_{B_R}u
\]
where we used Theorem~\ref{thm:c11} in the last inequality.
The lemma is proved.
\end{proof}

\begin{proof}
[Proof of Theorem~\ref{thm:global_krt} for $d=2$]
Without loss of generality, assume $\bar\psi=0$.  
Recall $\nabla$ and $\lock_R=\lock_{\omega,R}(x;\eta_R,\psi)$ in Definitions~\ref{def:differences} and \ref{def:loc-global}. 
For $R>1,\omega\in\Omega$,  set
\begin{equation}\label{eq:phiR-construct}
\phi_R(x)=\phi_{\omega,R}(x):=\lock_R(x)-x\cdot\nabla^+\lock_R(0)-\lock_R(0), \quad x\in\Z^d,
\end{equation}
where $\nabla^+=(\nabla_{e_i})_{1\le i\le d}$. 
In what follows we will show that, by taking $R\to\infty$, $\phi_R$ will converge to a desired global corrector up to a subsequence.

\noindent{\bf Step 1.} {\it (Moment bound for $\phi_R$.)} We will establish the moment bound for $d=2$:
\begin{equation}\label{eq:2d-gkrt-mmb}
\mb E\left[\exp\big(
C|\tfrac{\phi_R(x)}{|x|\log(|x|\vee 2)^{3/2}}|^p
\big)\right]
\le C_p,  \quad\text{ for }p\in(0,\tfrac{1}{3}),0<3|x|<R^{1/3}.
\end{equation}

To prove \eqref{eq:2d-gkrt-mmb},  writing $u(z)=u_{\omega;y}(z)=\loca(z,y)$,  by formula \eqref{eq:sensiti-formu-phi} we have
\begin{align}\label{eq:230308-1}
\abs{\vd y\phi_R(x)}&=
|\vd y \psi(y)-\tfrac12\tr(\vd y a\nabla^2\lock_{\omega'_y})(y)||\locg_{\omega}(x,y)-
x\cdot\nabla^+\locg(0,y)-\locg(0,y)|\nn\\ 
&\stackrel{Lemma~\ref{lem:phi-bd-c11}\eqref{item:new-corr-2}}\lesssim
\tx_y^{'2}\abs{
\locg_{\omega}(x,y)-
x\cdot\nabla^+\locg(0,y)-\locg(0,y)}\nn\\
&=\tx_y^{'2}\abs{
u(x)-x\cdot\nabla^+u(0)-u(0)}.
\end{align}
We consider two cases: $|y|\le 3|x|\le R^{1/3}$ and $3|x|< |y|$. 

When $x\neq 0$ and $|y|\le 3|x|\le R^{1/3}$, by Lemma~\ref{lem:loca_bound} and Theorem~\ref{thm:c11},
\begin{align*}
&|u(x)|+|u(0)|\lesssim\tx_y(\log(|x-y|\vee2)+\log(|y|\vee2))\lesssim\tx_y \log(|x|+2)\\
&|\nabla^+ u(0)|\lesssim \frac{\tx\tx_y}{|y|+1}\log(|y|+2).
\end{align*}
Thus, by \eqref{eq:230308-1}, 
\begin{equation}
\label{eq:230308-2}
\abs{\vd y\phi_R(x)}\lesssim
\tx\tx_y\tx_y^{'2}[\log(|x|+2)+\tfrac{|x|}{|y|+1}\log(|y|+2)].
\end{equation}

When $3|x|<|y|$,  noting that $u$ is $\omega$-harmonic in $B_{|y|}$, 
by Lemma~\ref{lem:230313}, Proposition~\ref{prop:new-corr-prop} and Theorem~\ref{thm:c11}, we obtain
\[
\abs{u(x)-x\cdot\nabla^+u(0)-u(0)}
\lesssim
\frac{|x|^3\tx^2\tx_y}{(|y|\wedge R)^2}\log(\frac {2R}{|y|\wedge R})e^{-c|y|/R}.
\]
As a result, by \eqref{eq:230308-1}, when $3|x|<|y|$,
\begin{align}
\label{eq:230308-3}
\abs{\vd y\phi_R(x)}
&\lesssim
\tx_y(\tx\tx_y')^2\frac{|x|^3}{(|y|\wedge R)^2}\log(\frac{2R}{|y|\wedge R})e^{-c|y|/R}
\end{align}
Combining \eqref{eq:230308-2}, \eqref{eq:230308-3}, with
\[
f(y):=\left\{
\begin{array}{lr}
&\log(|x|+2)+\tfrac{|x|}{|y|+1}\log(|y|+2) \quad\text{ when }|y|\le 3|x|,\\
&
\frac{|x|^3}{(|y|\wedge R)^2}\log(\frac{2R}{|y|\wedge R})e^{-c|y|/R} \quad\text{ when }|y|>3|x|
\end{array}
\right.
\]
we obtain, for $x$ with $3|x|<R^{1/3}$ and $n\ge 1$,
\[
\mb E[|\vd y\phi_R(x)/f(y)|^n]
\lesssim\mb E[\tx^{5n}].
\]
Computations show that, for $3|x|<R^{1/3}$, 
\[
\sum_y f(y)^2\lesssim 
|x|^2\log(|x|\vee 2)^3+\frac{|x|^6}{R^2}\lesssim
|x|^2\log(|x|\vee 2)^3.
\]
Hence,  setting $F(x)=|x|\log(|x|\vee 2)^{3/2}$ and applying Lemma~\ref{lem:efron-stein} to $Z=\phi_R(x)$ we have,  when $3|x|<R^{1/3}
$, $p\in(0,\tfrac{2d}{10+d})$,
\[
\mb E\left[\exp\big(
C|(Z-\mb EZ)/F(x)|^{p}
\big)\right]
\le C_p, \text{ for }x\neq 0.
\]
Furthermore,  by \eqref{eq:230225-2} we get $|EZ|\lesssim |x|$. Display \eqref{eq:2d-gkrt-mmb} is proved.

\medskip

\noindent{\bf Step 2.} {\it (Existence of the global corrector.)} Next, we will show the existence of a global corrector $\krt=\krt_\omega$ (as an a.s.  subsequential limit of $(\phi_R)_{R>0}$) in $d=2$ with moment bound \eqref{item:gkrt-2}. 

Similar to Steps 2 and 3 in the proof for the cases $d=3,4$,  it suffices to show \eqref{eq:cauchy-in-prob}.  Observe that, for any $n>m>1$, the function $u(x)=\phi_n(x)-\phi_m(x)$ is $\omega$-harmonic in $B_{2m}$ with $u(0)=0$.  
Hence, by Lemma~\ref{lem:230313}, when $3|x|<m^{1/3}$
\begin{align*}
\abs{\phi_n(x)-\phi_m(x)}
\lesssim
\frac{\tx^2|x|^3}{m^2}\osc_{B_{m/2}}(\phi_n-\phi_m)
\lesssim
\frac{\tx^2|x|^3}{m^2}\frac{1}{\#B_{m}}\sum_{z\in B_{m}}|\phi_n(z)-\phi_m(z)|
\end{align*}
and thus, by \eqref{eq:2d-gkrt-mmb} and H\"older's inequality we have, for $x\neq 0$, $n>m>27|x|^3$,
\[
\norm{\phi_n(x)-\phi_m(x)}_{L^2(\mb P)}
\lesssim
\frac{|x|^3}{m^2}m(\log m)^{3/2}
\stackrel{m\to\infty}\to 0.
\]
Hence, when $d=2$, \eqref{eq:cauchy-in-prob} follows by Chebyshev's inequality.

\noindent{\bf Step 3.} {\it (Stationarity and moment bound of $\nabla^2\krt$.)}
Since $\krt_\omega$ is unique up to adding an affine function (whose coefficients may depend on $\omega$), the field $\{\nabla^2\krt_\omega(x):x\in\Z^d\}$ is $\mb P$-a.s. unique.  In other words, this field is uniquely determined by $\omega$, and hence its stationarity follows from the stationarity of $\mb P$.

To obtain the moment bound of $\nabla^2\krt$,  observe that $\nabla^2\phi_R=\nabla^2\lock_{\omega,R}$.  Since $\nabla^2\krt$ is an a.s.  subsequential limit of $\nabla^2\phi_R$, by Lemma~\ref{lem:phi-bd-c11} and taking $R\to\infty$, we get, for $\mb P$-a.e.  $\omega$,
\begin{equation}\label{eq:krt-nabla2}
\abs{\nabla^2\krt(x)}\lesssim\tx_x^2\norm{\psi}_\infty \quad\text{ for all }x\in\Z^d.
\end{equation}
The bound in Theorem~\ref{thm:global_krt}\eqref{item:gkrt-2nd-der} follows.

\noindent{\bf Step 4.} {\it (Moment bound for $\nabla\krt$.)} 
Without loss of generality, we only derive the moment bound for $\nabla^+\krt$,  since $\nabla_{-e_i}\krt(x)=\nabla_{e_i}\krt(x-e_i)$ for $i=1,\ldots,d$.

By formula \eqref{eq:sensiti-formu-phi} we have, for $x,y\in\Z^d$,
\begin{align*}
\vd y\nabla^+\phi_R(x)=
-[\vd y \psi(y)-\tfrac12\tr(\vd y a\nabla^2\lock_{\omega'_y})(y)][\nabla^+\locg_{\omega}(x,y)-
\nabla^+\locg_\omega(0,y)],
\end{align*}
where $\nabla^+$ only acts on the first coordinate of $\locg$, i.e., $\nabla^+\locg(x,y)=\nabla^+\locg(\cdot,y)|_{\cdot=x}$.
Taking a subsequential of limit $R\to\infty$ we obtain, almost surely,
\[
\vd y\nabla^+\krt(x)
=
[\vd y \psi(y)-\tfrac12\tr(\vd y a\nabla^2\krt_{\omega'_y})(y)][\nabla^+ A(x,y)-
\nabla^+A(0,y)]
\]
where $A(x,y)$ is as in \eqref{eq:def-potential}.
Further, by \eqref{eq:krt-nabla2}, Theorem~\ref{thm:Green-bound}, and Theorem~\ref{thm:c11},  we get
\begin{align*}
\abs{\vd y\nabla^+\krt(x)}
\lesssim
\tx_{x+y}^{'2}[\frac{\tx_{x,y}^2}{|x-y|+1}\log(|x-y|\vee 2)+\frac{\tx_{0,y}^2}{|y|+1}\log(|y|\vee 2)],
\end{align*}
where $\tx_{x,y}=\tx_x+\tx_y$. Moreover, for any affine function $\ell$, with $A_\ell(\cdot,y)=A(\cdot,y)+\ell$, we have 
\[
\nabla^+ A(x,y)-
\nabla^+A(0,y)=\nabla^+ A_\ell(x,y)-
\nabla^+A_\ell(0,y)\le \osc_{B_{|x|+1}}A_\ell(\cdot,y). 
\]
 Hence, 
for $3|x|<|y|$,  by  Theorem~\ref{thm:c11} we get
\[
\abs{\vd y\nabla^+\krt(x)}
\lesssim
\tx_{x+y}^{'2}\der{B_{|x|+1}}^2(A(\cdot,y))
\lesssim
\tx_{x+y}^{'2}\frac{(\tx\vee |x|)^2}{|y|^2}\tx_y\log|y|.
\]
%
%
%
Hence, for $x\neq 0$, with 
\[
g(y):=\left\{
\begin{array}{lr}
\frac{1}{|x-y|+1}\log(|x-y|\vee 2)+\frac{1}{|y|+1}\log(|y|\vee 2) &\text{ when }|y|\le 3|x|\\
\frac{|x|^2}{|y|^2}\log|y| &\text{ when }|y|>3|x|,
\end{array}
\right.
\]
we obtain,  for  $x\in\Z^d$ and $n\ge 1$,  
\[
\mb E[|\vd y\nabla^+\krt(x)/g(y)|^n]\lesssim\mb E[\tx^{5n}].
\]
Computations show that
\[
\sum_y g(y)^2\lesssim (\log(|x|\vee2))^3.
\]
Hence,  applying Lemma~\ref{lem:efron-stein} to $Z=\nabla^+\krt(x)$ we have
\[
\mb E\big[\exp\big(C|\frac{Z-\mb EZ}{(\log(|x|\vee 2))^{3/2}}|^p\big)\big]\le C_p
\quad\text{ for }p\in(0,\tfrac{2d}{10+d}).
\]
\end{proof}

\newpage
\newtheorem{atheorem}{Theorem}
\numberwithin{atheorem}{section}
\newtheorem{alemma}[atheorem]{Lemma}
\newtheorem{acorollary}[atheorem]{Corollary}

\appendix
\section{Appendix}
Define the parabolic operator $\ms L_\omega$ as
\[
\ms L_\omega u(x,t)=\sum_{y:y\sim x}\omega(x,y)[u(y,t)-u(x,t)]-\partial_t u(x,t)
\]
for every function $u:\Z^d\times\R\to\R$ which is differentiable in $t$. 
The following results are used in the paper.
\begin{atheorem}(\cite[Theorem~17]{DG-19})
\label{thm:harnack-para}
Assume $\frac{\omega}{\tr\omega}>2\kappa I$ for some $\kappa>0$. Any non-negative function $u$ with $\ms L_\omega u=0$ in $B_{2R}\times(0,4R^2)$ for $R>0$ satisfies
\[
\sup_{B_{R}\times(R^2,2R^2)}u\le 
C\inf_{B_R\times(3R^2,4R^2)}u.
\]
\end{atheorem}

As a consequence, we have the following H\"older regularity for $u$.
\begin{acorollary}\label{acor:Kry-Saf}
Assume $\frac{\omega}{\tr\omega}>2\kappa I$ for some $\kappa>0$.
There exists $\gamma=\gamma(d,\kappa)\in(0,1)$ such that any non-negative function $u$ with $\ms L_\omega u=0$ in $B_{R}(x_0)\times(t_0-R^2,t_0)$, for  some $(x_0,t_0)\in\Z^d\times\R$ and $R>0$, satisfies
\[
|u(\hat x)-u(\hat y)|\le C\left(\frac{r}{R}\right)^\gamma\sup_{B_R(x_0)\times(t_0-R^2,t_0)}u
\]
for all $\hat x,\hat y\in B_r(x_0)\times(t_0-r^2,t_0)$ and $r\in(0,R)$.
\end{acorollary}

\subsection{Proof of Proposition~\ref{prop:harmonic}}
\label{asec:harmonic}
\begin{proof}
Let $p\in\fct_{j}$ be the $j$-th order Taylor polynomial (around 0) of $v$. Then
\[
\sup_{\B_{\theta R}}|v-p|\le C (\theta R)^{j+1}\sup_{\B_{R/3}}|D^{j+1} v|.
\]
This gives $\der{\B_{\theta R}}^{j+1}(v)\lesssim (\theta R)^{j+1}\sup_{\B_{R/3}}|D^{j+1} v|$. Furthermore, for any $q\in\fct_{j}$, $j\le 2$, note that $D(v-q)$ is an $\bar a$-harmonic function. Hence, by \cite[Theorem 2.10]{GiTr}, 
\begin{align*}
\sup_{\B_{R/3}}|D^{j+1} v|
&=\sup_{\B_{R/3}}|D^{j+1} (v-q)|\\
&\le 
\frac{C}{R^{j}}\sup_{\B_{5R/12}}|D(v-q)|\\
&=
\frac{C}{R^{j}}\sup_{x\in\B_{5R/12}}|\fint_{\B_{R/12}(x)}D(v-q)|
\le \frac{C}{R^{j+1}}\sup_{\B_{R/2}}|v-q|
\end{align*}
for $j\le 2$. 
Hence, taking infimum over $q\in\fct_j$,  we get $\der{\B_{\theta R}}^{j+1}(v)\lesssim\theta^{j+1}\der{\B_{R/2}}^{j+1}(v)$ for $j\le 2$. The first statement is proved.

To prove the second statement, observe that for any $x\in\B_{R/2}$, there are $2^d$ points $y_i\in\bar B_{R/2}$, $i\in\Lambda=\{1,\ldots,2^d\}$, such that $|y_i-x|\le 1$ and $x$ is a convex combination of the $y_i$'s. That is, $x=\sum_{i\in\Lambda}\alpha_i y_i$ for some $\alpha_i\ge0$ with $\sum_{i\in\Lambda}\alpha_i=1$. Let $p\in\fct_j$, $j\le 2$, be such that $\max_{B_{2R/3}}|v-p|\le 2\der{2R/3}^{j+1}(v)$ and denote the Hessian matrix of $p$ by $M_p$. Then, for $x\in\B_{R/2}$,  $j\le 2$,
\begin{align}\label{eq:221208-1}
|v(x)-p(x)|
&\le [v]_{1;\B_{R/2+1}}+
\sum_{i\in\Lambda}\alpha_i |v(y_i)-p(y_i)|+\Abs{p(x)-\sum_{i\in\Lambda}\alpha_ip(y_i)}\nn\\
&\le 
 [v]_{1;\B_{R/2+1}}+\max_{\bar B_{R/2}}|v-p|+CR|M_p|.
\end{align}	
Further,  using the fact (see \cite[Cor.6.3]{GiTr}) that 
\[
R[v]_{1;\B_{R/2+1}}\lesssim
\sup_{\B_{2R/3}}|v|+R^2|c_0|
\lesssim
\sup_{\partial\B_{2R/3}}|v|+R^2|c_0|
\]
and (Note that the following bound is not needed for the case $j=1$ where $M_p\equiv 0$.)
\[
R^2|M_p|
\lesssim
\max_{y\in B_{R/2}}|p(y)+p(-y)-2p(0)|
\lesssim
\max_{B_{R/2}}|v-p|+\max_{B_{R/2}}|v|
\lesssim
\der{2R/3}^3(v)+\max_{B_{R/2}}|v|,
\]
display \eqref{eq:221208-1} implies, for $j\le 2$,
\begin{align*}
\der{\B_{R/2}}^{j+1}(v)
\lesssim
\tfrac{1}{R}\sup_{\partial\B_{2R/3}}|v|+R|c_0|+\der{2R/3}^{j+1}(v).
\end{align*}
The second claim follows.
\end{proof}
\subsection{Verification of \eqref{eq:verti-der-phi}}
In this subsection we will verify the inequality
\[
\int_0^\infty (1+t)^{-d/2}\exp\left[
-\tfrac{t}{R^2}-c\mf h(|y|,t)
\right]\dd t
\lesssim 
e^{-c|y|/R}
v(|y|+1),
\quad
\forall y\in\Z^d.
\]
We break the integral on the left side of the above inequality as 
\[
\int_0^\infty=\int_0^{|y|/2}
+\int_{|y|/2}^{|y|^2}+\int_{|y|^2}^{\infty}
=:\text{I}+\text{II}+\text{III}.
\]
It suffices to consider the case $|y|\ge 1$.
First, with $c_2>0$  sufficiently small,
\begin{align*}
\text{I}=\int_0^{|y|/2}(1+t)^{-d/2}\exp\left(
-\tfrac{t}{R^2}-c|y|\log\tfrac{|y|}{t}
\right)\dd t
\le 
|y|e^{-c|y|}\lesssim
e^{-c_2|y|/R}v(|y|).
\end{align*}	
Moreover, noting that $-\tfrac{t}{2R^2}-c\tfrac{|y|^2}{t}\lesssim-\tfrac{|y|}{R}$,
\begin{align*}
\text{II}&=\int_{|y|/2}^{|y|^2}(1+t)^{-d/2}\exp\left(-\tfrac{t}{R^2}-c\tfrac{|y|^2}{t}\right)\dd t\\
&\lesssim
e^{-c|y|/R}
\int_{0}^{|y|^2}t^{-d/2}e^{-c|y|^2/t}\dd t\\
&\lesssim
e^{-c|y|/R}|y|^{2-d}\int_1^{\infty}s^{d/2-2}e^{-cs}\dd s
\lesssim
e^{-c|y|/R}v(|y|).
\end{align*}	
Similarly, for $d=2$,
\begin{align*}
\text{III}&\lesssim 
e^{-c|y|/R}\int_{|y|^2}^\infty
(1+t)^{-d/2}\exp\left(-\tfrac{t}{2R^2}\right)\dd t\\
&\lesssim 
e^{-c|y|/R}\int_{|y|^2/R^2}^{\infty}s^{-1}e^{-s/2}\dd s\\
&\lesssim
e^{-c|y|/R}\left[1+\int_0^\infty s^{-1}\mathbbm{1}_{\{|y|^2/R^2\le s\le 1\}}\dd s\right]
\lesssim
e^{-c|y|/R}v(y|).
\end{align*}	
For $d\ge 3$, we have 
\[
\text{III}\lesssim 
e^{-c|y|/R}\int_{|y|^2}^\infty
(1+t)^{-d/2}\dd t
\lesssim
e^{-c|y|/R}v(|y|).
\] 
Therefore, the above bounds of I, II, III imply inequality\eqref{eq:verti-der-phi}.

\subsection{Verification of \eqref{eq:nu2-sum}}
When $d=2$,
\begin{align*}
\int_0^\infty
e^{-cr/R} v(r)^2r^{d-1}\dd r
&\lesssim
\int_0^\infty e^{-cr/R}[1+(\log\tfrac{R}{(r+1)\wedge R})^2]r\dd r\\
&\lesssim
R^2+\int_1^R e^{-cr/R}\left(\log\frac{R}{r}\right)^2r\dd r\lesssim R^2.
\end{align*}	
When $d=3$,
\begin{align*}
\int_0^\infty
e^{-cr/R} v(r)^2r^{d-1}\dd r
\lesssim
\int_0^\infty e^{-cr/R}\lesssim R.
\end{align*}	
When $d=4$,
\begin{align*}
\int_0^\infty
e^{-cr/R}v(r)^2r^{d-1}\dd r
&\lesssim
\int_0^R(1+r)^{-1}\dd r
+\int_R^\infty e^{-cr/R}R^{-1}\dd r\lesssim\log R.
\end{align*}	
When $d\ge 5$,
\begin{align*}
\int_0^\infty
e^{-cr/R}v(r)^2r^{d-1}\dd r
&\lesssim
\int_0^\infty(1+r)^{-2}\dd r=1.
\end{align*}


\end{document}